\documentclass[11pt]{amsart}
\usepackage{amsmath}
\usepackage{amsfonts}
\usepackage{stmaryrd}
\usepackage{amssymb}
\usepackage{graphicx}
\usepackage[all]{xy}
\usepackage{pdfsync}
\usepackage{enumerate}
\usepackage{hyperref}
\usepackage{multirow} % Required for multirows
\usepackage{tensor} %required for double fiber products
\usepackage{xcolor}

\newtheorem{theorem}{Theorem}[section]

\newtheorem{corol}[theorem]{Corollary}

\newtheorem{definition}[theorem]{Definition}
\newtheorem{example}[theorem]{Example}

\newtheorem{lemma}[theorem]{Lemma}

\newtheorem*{problem}{Problem}
\newtheorem{prop}[theorem]{Proposition}

\theoremstyle{definition}
\newtheorem{remark}[theorem]{Remark}

\numberwithin{equation}{section}

\newcommand{\Ss}{\mathbb S}
\newcommand{\Rr}{\mathbb R}
\newcommand{\Qq}{\mathbb Q}
\newcommand{\Zz}{\mathbb Z}

\newcommand{\Cc}{\mathbb C}

\newcommand{\Tt}{\mathbb T}

\newcommand{\eps}{\varepsilon}

\newcommand{\X}{\ensuremath{\mathfrak{X}}}

\renewcommand{\d}{\mathrm d}
\newcommand{\cG}{\ensuremath{\mathcal{G}}}
\newcommand{\cK}{\ensuremath{\mathcal{K}}}

\newcommand{\cN}{\ensuremath{\mathcal{N}}}
\newcommand{\cM}{\ensuremath{\mathcal{M}}}
\newcommand{\pr}{\ensuremath{\mathrm{pr}}}

\newcommand{\timesst}{\tensor[_s]{\times}{_t}}

 \DeclareMathOperator{\Der}{Der}

\DeclareMathOperator{\gl}{\mathfrak{gl}}

\DeclareMathOperator{\so}{\mathfrak{so}}

\DeclareMathOperator{\uu}{\mathfrak{u}}
\renewcommand{\aa}{\mathfrak{a}}
\renewcommand{\sl}{\mathfrak{sl}}
\renewcommand{\sp}{\mathfrak{sp}}
\DeclareMathOperator{\GL}{GL}
\DeclareMathOperator{\SL}{SL}
\DeclareMathOperator{\SO}{SO}
\DeclareMathOperator{\SU}{SU}
\DeclareMathOperator{\OO}{O}
\DeclareMathOperator{\UU}{U}
\DeclareMathOperator{\Sp}{Sp}
\DeclareMathOperator{\tr}{trace}
\DeclareMathOperator{\Diff}{Diff}
\DeclareMathOperator{\Curv}{Curv}
\DeclareMathOperator{\Tors}{Tors}

\DeclareMathOperator{\rank}{rank}
\renewcommand{\hom}{\mathrm{Hom}}
\newcommand{\reg}{\mathrm{reg}}
\newcommand{\can}{\mathrm{can}}

\newcommand{\G}{\mathit{\Gamma}}            % Lie groupoid
\newcommand{\s}{\mathbf{s}}             % source map
\renewcommand{\t}{\mathbf{t}}           % target map
\renewcommand{\H}{\mathcal{H}}          % Lie subgroupoid
\newcommand{\A}{A}                      % Lie algebroid
\newcommand{\al}{\alpha}                % section of Lie algebroid
\newcommand{\be}{\beta}                 % section of Lie algebroid
\newcommand{\Lie}{\mathcal{L}}          % Lie derivative
\renewcommand{\gg}{\mathfrak{g}}        % Lie algebra
\newcommand{\hh}{\mathfrak{h}}          % Lie subalgebra
\newcommand{\Ker}{\text{\rm Ker}\,}     % Kernel
\renewcommand{\Im}{\text{\rm Im}\,}     % Image
\newcommand{\Ad}{\text{\rm Ad}\,}       % Adjoint
\newcommand{\tto}{\rightrightarrows}    % Arrows of a groupoid
\newcommand{\wmc}{\omega_{\textrm{MC}}}   %Maurer-Cartan form
\DeclareMathOperator{\Inv}{Inv}

\newcommand{\B}{\mathrm{F}}                    %Frame Bundle
\begin{document}
\title[Cartan's Realization Problem]{The Global Solutions to a Cartan's Realization Problem}
\author{Rui Loja Fernandes}
\address{Department of Mathematics, University of Illinois at Urbana-Champaign, 1409 W. Green Street, Urbana, IL 61801 USA}
\email{ruiloja@illinois.edu}

\author{Ivan Struchiner}
\address{Departamento de Matem\'atica, Universidade de S\~ao Paulo, Rua do Mat\~ao 1010, S\~ao Paulo, SP 05508-090, Brasil}
\email{ivanstru@ime.usp.br}

\thanks{RLF was partially supported by NSF grants DMS-1710884, DMS-2003223, a Simons Fellowship in Mathematics and FCT/Portugal. IS was was partially supported by CNPq grant 307131/2016-5, and 
FAPESP grants 2015/50472-1 and 2015/22059-2,}

\begin{abstract} 
We introduce a systematic method to solve a type of Cartan's realization problem. Our method builds upon a new theory of Lie algebroids and Lie groupoids with structure group and connection. This approach allows to find local as well as complete solutions, their symmetries, and to determine the moduli spaces of local and complete solutions. We illustrate our method with the problem of classification of extremal K\"ahler metrics on surfaces.
\end{abstract}

\date{\today}

\maketitle

\setcounter{tocdepth}{1}
\tableofcontents

\setcounter{section}{-1}
%%%%%%%%%%%%%%%%%%%%%%%%%%%%%%%%%%%%%%%%%%%%%%%%%%%%%%%%%%%%%%%%%%
%%%%%%%%%%%%%%%%%%%%%%%%%%%%%%%%%%%%%%%%%%%%%%%%%%%%%%%%%%%%%%%%%%
\section{Introduction}
\label{sec:introd}
%%%%%%%%%%%%%%%%%%%%%%%%%%%%%%%%%%%%%%%%%%%%%%%%%%%%%%%%%%%%%%%%%%
%%%%%%%%%%%%%%%%%%%%%%%%%%%%%%%%%%%%%%%%%%%%%%%%%%%%%%%%%%%%%%%%%%

% Classification problems
% A central problem in geometry is that of classifying classes of geometric structures. 
This paper solves the existence, equivalence and classification problems for \emph{finite dimensional} families of $G$-structures that can be formalized as a \emph{Cartan's Realization Problem}. These include many interesting, non-trivial, geometric classification problems, such us Ricci type symplectic connections, Bochner bi-Lagrangian manifolds, Bochner-K\"ahler metrics, torsionless symplectic connections with exceptional holonomy groups $\mathrm{Spin}(2,10)$, $\mathrm{Spin}(6,6)$, and $\mathrm{Spin}(12,\Cc)$, surface metrics of Hessian type, among others (see \cite{Bryant:notes,Bryant,CaSch2, CaSch, ChiMerkulovSchwachhofer:inventiones, MerkulovSchwachhofer:annals, Sch}).

%Let $G$ be a fixed Lie subgroup of $\GL(n,\Rr)(\Rr)$. Ideally, one is interested in the following local classification problem:
%What are all germs of $G$-structure up to equivalence? However, this problem is not feasible because the moduli space involved is ``too big". However, in practice, one is often interested in problems involving $G$-structures satisfying some additional conditions. These extra restrictions may lead to family of $G$-structures with a ``finite dimensional'' moduli space. Examples include surfaces of constant Gaussian curvature (see bellow), Einstein or Ricci flat metrics \cite{Besse} ???, Bochner-K\"ahler metrics \cite{Bryant}, special symplectic manifolds \cite{CaSch2, CaSch}, etc. In this paper we describe a method to deal with such problems. 

%In order to parameterize families of $G$-structures we will introduce the notion of a {\bf $G$--algebroid}, a special type of Lie algebroid equipped with $G$-action. 
In order to explain the type of Cartan's realization problem that we are interested in, recall that given a $G$-structure $\B_G(M)\to M$, the tautological 1-form  $\theta\in\Omega^1(\B_G(M), \Rr^n)$, together with a connection 1-form $\omega \in \Omega^1(\B_G(M), \gg)$, give a coframe $(\theta,\omega)$ on $\B_G(M)$ which satisfies the well-known structure equations
\begin{equation}
\label{eq:structure:1}
\left\{
\begin{array}{l}
\d \theta=c ( \theta \wedge \theta) - \omega \wedge \theta\\

\d \omega= R( \theta \wedge \theta) 
- \omega\wedge\omega
\end{array}
\right.
\end{equation}
We are interested in ``finite-dimensional'' families of $G$-structures obtained by imposing restrictions on the curvature $R$ and/or torsion $c$. This can be formulated by requiring that $R$ and $c$ factor through a $G$-equivariant map $h:\B_G(M)\to X$ into a $G$-manifold $X$, so that $R=R(h)$ and $c=c(h)$ for $G$-equivariant maps $R: X \to \hom(\wedge^2\Rr^n, \gg)$ and $c: X \to \hom(\wedge^2\Rr^n, \Rr^n)$.  Loosely speaking, $X$ is the space of possible values of the torsion and curvature.

Since the pair $(\theta,\omega)$ is a coframe in $\B_G(M)$, the map $h:\B_G(M)\to X$ also satisfies structure equations
\begin{equation}
\label{eq:structure:2}
\d h = F(h,\theta) + \psi(h,\omega),
\end{equation}
where $F: X \times \Rr^n \to TX$ is a $G$-equivariant vector bundle map and $\psi: X \times \gg \to TX$ is the infinitesimal action of $\gg$ on $X$.

% One should view $X$ as the space of complete invariants of the family, a versal space of deformations, and the $G$-equivariant map $h:F_G(M)\to X$ as specifying a particular $G$-structure in the family.

% Formulation of Cartan Problem
The $G$-manifold $X$, together with the maps $c: X \to \hom(\wedge^2\Rr^n, \Rr^n)$, $R: X \to \hom(\wedge^2\Rr^n, \gg)$ and $F: X \times \Rr^n \to TX$, is the data that specifies the family of $G$-structures of the classification problem at hand. This data is known \emph{a priori} and it usually arises from a differential analysis of the problem. Once this data is given, one would like to find all $G$-structures that satisfy the associated structure equations \eqref{eq:structure:1} and \eqref{eq:structure:2}. Bryant in the appendix of \cite{Bryant} calls this a \emph{Cartan's realization problem} and it can be formulated precisely as follows

\begin{problem}[{\bf Cartan's Realization Problem}]
\label{prob:Cartan}
One is given \textbf{Cartan Data}:
\begin{itemize}
\item a closed Lie subgroup $G \subset \GL(n,\Rr)$;
\item a $G$-manifold $X$ with infinitesimal action $\psi:X\times\gg\to TX$;
\item equivariant maps $c: X \to \hom(\wedge^2\Rr^n, \Rr^n)$, $R: X \to \hom(\wedge^2\Rr^n, \gg)$ and $F: X \times \Rr^n \to TX$;
\end{itemize}
and asks for the existence of \textbf{solutions}: 
\begin{itemize}
\item an n-dimensional orbifold $M$;
\item a $G$-structure $\B_G(M)\to M$ with tautological form $\theta\in\Omega^1(\B_G(M), \Rr^n)$ and connection 1-form $\omega \in \Omega^1(\B_G(M), \gg)$;
\item an equivariant map $h: \B_G(M) \to X$;
\end{itemize} 
satisfying the structure equations \eqref{eq:structure:1} and \eqref{eq:structure:2}.

%\begin{equation*}
%\label{eq:structure} \left\{
%\begin{array}{l}
%\d \theta=c (h)( \theta \wedge \theta) - \omega \wedge \theta\\
%\d \omega= R (h)( \theta \wedge \theta) 
%- \omega\wedge\omega
%\\
%\d h = F(h,\theta) + \psi(h,\omega)
%\end{array}
%\right.
%\end{equation*}
\end{problem}

For reasons to be discussed later, we allow for our solutions to be \emph{orbifolds}: while $\B_G(M)$ is a manifold, the $G$-action is proper but only locally free, so $M=\B_G(M)/G$ is an (effective) orbifold. For a concrete example of how such a realization problem arises see Section \ref{sec:example:EK:surfaces:0}, where we discuss the problem of finding extremal K\"ahler metrics on surfaces.

% Families rather than specific solutions and formulation of Cartan problem using $G$-structure algebroids

The ultimate goal is to find versal families of solutions to Cartan's Realization Problem. Such a family consists of a submersion $\s:\G\to X$ with a principle $G$-action along its fibers, such that each fiber $\s^{-1}(x)$ is a solution to the problem. The key observation is that such families of $G$-structures \emph{are not} arbitrary families. Rather they are families parameterized by the source map of a Lie groupoid. We formalize this idea precisely in the new notion of a {\bf $G$-structure groupoid $\G\tto X$}, namely a Lie groupoid with a $G$-action compatible with the groupoid multiplication and making the source fibers into $G$-structures. The corresponding infinitesimal notion, called a {\bf $G$-structure algebroid $A\to X$}, is precisely the object that one can associate to Cartan Data.

The idea of using Lie algebroids in connection to Cartan's realization problems is not new. It is goes back to the work of R.~Bryant \cite[Appendix]{Bryant}, who also popularized it in various talks (see also \cite{CrainicYudilevich}). The novelty in our work is twofold: 
\begin{enumerate}[(i)]
\item we include the structure group $G$, leading to the notion of a {\bf $G$-structure algebroid} and its global counterpart a {\bf $G$-structure groupoid}, and
\item we equip $G$-structure algebroids and $G$-structure groupoids with {\bf connections}, since our $G$-structures come equipped with connections.
\end{enumerate}
As we shall see, these two points are crucial and they also allow us to deal with global issues. 

% Formulation of various problems associated with Cartan Problem
Using $G$-structure Lie groupoid theory we give complete answers to the following classification problems:

\begin{description}
\item[Classification of solutions]
Characterize all solutions of a Cartan's realization problem up to equivalence.
\smallskip
\item[Classification of symmetries]
Determine the group of symmetries and the Lie algebra of infinitesimal symmetries of a given solution.
\smallskip
\item[Moduli space of solutions]
Find, if any, the differential structure of the moduli space of solutions of a Cartan's realization problem. 
\end{description}

In all these problems it is important to distinguish between local and global solutions to Cartan's problem. By a {\bf germ of solution} of Cartan's problem we mean a solution $(\B_G(M),(\theta,\omega),h)$ with $M=U/\Upsilon$, where $U$ is an open neighborhood of $0\in \Rr^n$ acted upon by some finite subgroup $\Upsilon\subset\Diff(\Rr^n)$. Two such solutions with domains $U_1/\Upsilon$ and $U_2/\Upsilon$ are identified if they coincide on some smaller neighborhood of $0$. Note that a germ of solution comes with a marked point, namely $0$. Hence, by restricting a realization to neighborhoods of different points one may obtain inequivalent germs of realizations. This means that the \emph{local versions} of the classification problems, the moduli space of germs of solutions, etc., in general look quite different from their \emph{global} counterpart.

On the global side, we shall introduce the notion of {\bf complete solutions}. These are, in some sense, maximal solutions which do not admit any extensions to  larger domains. For example, when $G\subset O(n)$ is a compact Lie group and $\omega$ is the Levi-Civita connection, we will see that the notion of complete solution is related with the usual notion of completeness of the Riemannian metric on $M$. We shall provide answers to the classification problems above for both local and complete solutions. The method we will present  also allows, in principle, to find {\bf explicit solutions} to a given realization problem. 

% Our results

After setting up the notions of $G$-structure groupoids and $G$-structure algebroids with connection, our first result extends a remark of Bryant (see \cite[Appendix]{Bryant}).

\begin{theorem}
\label{thm:Cartan:G:alg}
Let $(G, X,c,R,F)$ be a Cartan Data. If for each $x \in X$ there exists a solution $(\B_G(M), (\theta,\omega), h)$ with $x \in \Im h$, then the Cartan Data determines a $G$-structure algebroid with connection $A\to X$. In this case, solutions are in one to one correspondence with bundle maps 
\[
\xymatrix{
TP \ar[r]^--{(\theta,\omega)} \ar[d] & A\ar[d]\\
P \ar[r]_{h} & X}
\]
which are morphisms of $G$-structure algebroids with connection.
\end{theorem}

The special bundle maps that appear in the statement of the theorem are called {\bf $G$-realizations of $A$}. Therefore we have a 1:1 correspondence
\[
 \left\{\txt{Cartan Data\\ admitting solutions}\right\} 
\quad
\stackrel{1-1}{\longleftrightarrow}\
\quad
\left\{\txt{$G$-structure algebroid $A$\\ with connection}\right\}
\]
%
%\[ 
%\begin{minipage}[c]{5.5 cm}
%\framebox[5.4 cm][c]{
%Cartan Data\\ (admitting solutions)
%}
%\end{minipage}
%\stackrel{1-1}{\longleftrightarrow}\quad
%\begin{minipage}[c]{5.5 cm}
%\framebox[5.4 cm][c]{$G$-structure algebroid\\ with connection $A$}
%\end{minipage}
%\]
and under this correspondence, we have % solving Cartan's realization problem is equivalent to finding $G$-realizations of the associated $G$-structure algebroid $A$:
\[
 \left\{\txt{Solutions of Cartan's\\ realization problem}\right\} 
\quad
\stackrel{1-1}{\longleftrightarrow}\
\quad
\left\{\txt{$G$-realizations of the\\ algebroid $A$}\right\}\qquad
\]
%
%\[ 
%\begin{minipage}[c]{4.3 cm}
%$\qquad$
%\framebox[2.2 cm][c]{Solutions}
%\end{minipage}
%\stackrel{1-1}{\longleftrightarrow}\quad
%\begin{minipage}[c]{4.2 cm}
%$\qquad$
%\framebox[3.2 cm][c]{$G$-realizations of $A$}
%\end{minipage}
%\]
\medskip

% Solutions coming from $G$-integration
Hence, in order to solve Cartan problem, one needs to find the $G$-rea\-li\-za\-tions of a $G$-structure algebroid with connection. One way (but not the only way!) these can be obtained is as source fibers of a {\bf $G$-integration}, i.e., by producing a $G$-structure groupoid $\G\tto X$ with connection.

\begin{theorem}
\label{thm:integration:realizations}
If $\G\tto X$ is a $G$-structure groupoid integrating a $G$-structure algebroid with connection $A\to X$, and $\wmc$ denotes the Maurer-Cartan form, then $(\s^{-1}(x),\wmc|_{\s^{-1}(x)},\t)$ is a $G$-realization of $A$, for each $x\in X$. Moreover, in this case, any $G$-realization of $A$ has a cover isomorphic to a $G$-invariant open subset of one such $G$-realization.
\end{theorem}

The $G$-realizations of the form $(\s^{-1}(x),\wmc|_{\s^{-1}(x)},\t)$ and the corresponding solutions $M=\s^{-1}(x)/G$ of a Cartan's realization problem arising from $G$-integrations are examples of {\it complete solutions} (see below). However, a Cartan's realization problem does not need to have complete solutions and this result \emph{does not} solve completely Cartan Problem. 

On the one hand, not every Lie algebroid integrates to a Lie groupoid. Even more, we will see that a $G$-structure algebroid $A\to X$ may be integrable to a Lie groupoid, but not to a $G$-structure groupoid. The presence of the structure group $G$ implies that there are additional obstructions to $G$-integrability besides the well-known obstructions to integrability of Lie algebroids. We will identify these obstructions and give a method to compute them.

% Other solutions
On the other hand, a $G$-structure algebroid $A\to X$ which is not $G$-integrable still admits $G$-realizations. For a $G$-realization $(P,(\theta,\omega),h)$ of $A\to X$ the image $h(P)\subset X$ is an open $G$-saturated set $U$ in a leaf $L$ of $A$. The restriction $A|_U$ is still a $G$-structure algebroid and if this restriction is $G$-integrable one can apply the previous theorem. In fact, one has local $G$-integrability, as expressed by the following result.

\begin{theorem}
Let $A\to X$ be a $G$-structure algebroid with connection, $x_0\in X$ and denote by $L$ the leaf of $A$ containing $x_0$. Then the $G$-orbit of $x_0$ has a $G$-invariant neighborhood $U\subset L$ such that $A|_U$ is $G$-integrable.
\end{theorem}

We shall call a $G$-integration of $A|_U$ a {\bf local $G$-integration} of $A$. As a corollary we conclude that a Cartan's realization problem always has solutions.

\begin{corol}
\label{cor:existence}
Given a $G$-structure algebroid $A\to X$ with connection, for any $x_0\in X$ there exists a $G$-realization $(P,(\theta,\omega),h)$ with $x_0\in \Im h$ obtained by local $G$-integration. Every germ of a solution to Cartan Problem arises in this way.
\end{corol}

% Complete solutions
On the global side, we will introduce the notions of complete and strongly complete solution of a Cartan's realization problem. A {\bf strongly complete solution} can be defined as a solution that covers a leaf $L$ of the algebroid $A$ and is such that every local symmetry extends to a global symmetry. A {\bf complete solution} can be characterized as a solution which is covered by a strongly complete solution.
%
%
%
%A $G$-realization $(P,(\theta,\omega),h)$ of $A$ is called a {\bf full realization} if $h:M\to X$ is surjective to the leaf $L$ of $A$ containing its image. Any such full $G$-realization comes equipped with an infinitesimal action of $A|_L$, denoted $\sigma: \Gamma(A|_L) \to \X(P)$, which is characterized by
%\[ (\theta,\omega)\sigma(\xi) = \xi, \text{ for all } \xi \in \Gamma(A|_L).\] 
%A {\bf complete $G$-realization} is a full realization for which this action is complete. The corresponding solutions of Cartan's realization problem are called {\bf complete solutions}. 
The existence of complete solutions is related to the $G$-integrability of the restriction of the $G$-structure algebroid $A$ to a leaf.

\begin{theorem}
Let $A \to X$ be a $G$-structure algebroid. There exists a complete solution covering a leaf $L$ of $A$ if and only if $A|_L$ is $G$-integrable. Every strongly complete solution covering $L$  is of the form $(\s^{-1}(x),\wmc|_{\s^{-1}(x)},\t)$ for some $G$-integration of $A|_L$.
\end{theorem}

A $G$-structure algebroid $A\to X$ can have many $G$-integrations. We will see that there is a maximal source connected $G$-integration, called the {\bf canonical $G$-integration} and denoted by $\Sigma_G(A)\tto X$. The canonical $G$-integration is characterized by the property that the orbifolds $\s^{-1}(x)/G$ are 1-connected. Moreover, it is the largest source-connected $G$-integration in the sense that every source connected $G$-integration $\G$ of $A$ is covered by $\Sigma_G(A)$.

%All these $G$-integrations give rise to complete $G$-realizations.
%
%Given a $G$-realization $(P,(\theta,\omega),h)$ of $A\to X$ with base $M=P/G$, we can form the fiber product with any (orbifold) cover $q:\tilde{M}\to M$:
%\[
%\xymatrix{
%q^*P\ar[r]\ar[d] & P\ar[d] \\
%\tilde{M}\ar[r] & M}
%\]
%The pull-back $(q^*P,(q^*\theta,q^*\omega),q^*h)$ is again a $G$-realization of $A$ which covers $(P,(\theta,\omega),h)$. In particular, if $q: \tilde{M} \to M$ is the universal covering orbifold, we call $(q^*P,(q^*\theta,q^*\omega),q^*h)$ the \textbf{universal cover} of $(P,(\theta,\omega),h)$. Moreover, we will say that 

All complete solutions with 1-connected base orbifolds are strongly complete and, in fact, correspond to $G$-realizations that are source fibers of the canonical $G$-integration.

\begin{theorem}\label{thm: complete covering}
Let $A \to X$ be a $G$-structure algebroid. A 1-connected complete solution covering a leaf $L$ of $A$ is isomorphic to $M=\s^{-1}(x_0)/G$, where $\s^{-1}(x_0)$ is a source fiber of the canonical $G$-integration of $A|_L$.
\end{theorem}

% In Section \ref{sec:solutions} we will also characterize which complete realizations are isomorphic to the source fiber of \emph{some} $G$-integration of $A$.

% Moduli space
We will not discuss in this introduction details on the symmetries of solutions and the moduli space of solutions. We only state the following theorem which illustrates our results in this direction.

\begin{theorem}
\label{thm:moduli:space}
Let $A\to X$ be an integrable $G$-structure algebroid with connection. Then for the associated Cartan's realization problem:
\begin{enumerate}[(i)]
\item The moduli space of germs of solutions is the geometric stack presented by the action Lie groupoid $G\times X\tto X$. In particular, the group of symmetries of a germ of solution $[x]\in X/G$ is isomorphic to the isotropy group $G_x$;
\item The moduli space of complete, 1-connected, solutions is the geometric stack presented by the Lie groupoid $\Sigma_G(A)\tto X$. In particular, the group of symmetries of a 1-connected, complete, solution $[x]\in X/\Sigma(A)$ is isomorphic to the isotropy group $\Sigma_G(A)_x$.
\end{enumerate}
\end{theorem}

Using these and other results to be discussed later in the paper, we are able to gain a complete understanding not only of existence of single solutions, but also of smooth families of solutions and smooth deformations.

% Geometric structures
When $G\subset \GL(n,\Rr)$ is a Lie group preserving some tensor $K$ in $\Rr^n$, very much like ordinary $G$-structures, a $G$-structure algebroid $A\to X$ with connection inherits an $A$-tensor $K_A$ of the same type as $K$. Moreover, for every $G$-realization $(P,(\theta,\omega),h)$ of $A$ the corresponding solution $M=P/G$ 
also inherits a tensor $K_M$. 

In particular, when  $G\subset \OO(n,\Rr)$ and the torsion $c$ vanishes, we say that a $G$-structure algebroid $A\to X$ is of {\bf metric type}. In this case, not only every solution is a Riemannian orbifold but also every leaf of $A$ inherits a natural Riemannian metric. Then completeness of solutions, as defined above, and metric completeness are related. This is made precise by the following result, which allows to determine the complete metric solutions of the corresponding Cartan's realization problem.

\begin{theorem}
Let $A\to X$ be a $G$-structure algebroid with connection of metric type and let $(P,(\theta,\omega),h)$ be a $G$-realization mapping into a leaf $L$. The metric $K_M$ on $M=P/G$ is complete if and only if the $G$-realization $(P,(\theta,\omega),h)$ is complete and the metric $K_L$ on $L$ is complete. In particular, if $M=P/G$ is metric complete and 1-connected, then $P$ is isomorphic to a source fiber of the canonical $G$-integration.
\end{theorem}

One can consider also other types of $G$-structure algebroids, such as {\bf symplectic type} ($G\subset \Sp(n,\Rr)$), {\bf complex type} $(G\subset\GL(n,\Cc)$ or {\bf K\"ahler type} ($G\subset \UU(n))$. In these cases, there are additional \emph{integrability conditions} which can be expressed in terms of the algebroid. For example, in the case of symplectic type the algebroid $A$ carries a canonical $A$-form $\Omega_A\in\Omega^2(A)$ and the integrability condition is $\d_A\Omega_A=0$, where $\d_A$ is the Lie algebroid differential. This is yet another illustration of how Lie algebroids and groupoids provide the appropriate language to deal with such problems.

In the first section of the paper, we start by setting up the Cartan Realization Problem associated with the classification of extremal K\"ahler metrics on surfaces. We also take a first look at how Lie algebroids arise in connection with these kinds of problems. We will come back to this problem only at the very last section, to illustrate all the theory. There we will see that, using our techniques, one can give complete answers to all the questions above. Due to length of the paper, we will report elsewhere on applications of our theory to more sophisticated problems (see also \cite{FernandesStruchiner1,FernandesStruchiner2} for the special case of the equivalence problem of a singular $G$-structure).

%\bigskip
% Organization of the paper
% This paper is organized as follows.

{\bf Acknowledgments.} 
Our work was originally inspired by Robert Bryant remarks in the appendix of \cite{Bryant} and subsequent discussions with Marius Crainic, to whom we are in debt. We also thank Ori Yudilevich for several conversations and comments at various stages of this work, Gon\c{c}alo Oliveira for pointing out a mistake in Section \ref{sec:complete:Kahler} in a previous version of this paper, and the two anonymous referees for various comments. Finally, we acknowledge the gracious support of IME-USP, the University of Illinois at Urbana-Champagne and IMPA, while this project was being developed.

%%%%%%%%%%%%%%%%%%%%%%%%%%%%%%%%
%%%%%%%%%%%%%%%%%%%%%%%%%%%%%%%%
\section{Prelude: an example}
\label{sec:algebroids}
%%%%%%%%%%%%%%%%%%%%%%%%%%%%%%%%
%%%%%%%%%%%%%%%%%%%%%%%%%%%%%%%%

Before we introduce the language of $G$-structure groupoids and $G$-structure algebroids, we will explain the type of Cartan's realization problems we are interested in and how they arise. In this section, we start by discussing one such problem:  the classification of extremal K\"ahler metrics on surfaces. We will also explain how to associate a Lie algebroid to a Cartan's realization problem. This is a simple (but interesting!) paradigmatic example of the kind of problems to which our theory applies. We will come back to it in the last section of the paper, after we have established all the theory. We will also give an example of a problem with an infinite dimensional space of invariants, to which our theory does not apply.

%%%%%%%%%%%%%%%%%%%%%%%%%%%%%%
\subsection{Extremal K\"ahler metrics on surfaces}
\label{sec:example:EK:surfaces:0}
%%%%%%%%%%%%%%%%%%%%%%%%%%%%%%

Recall that if $(M,g,\Omega,J)$ is a compact K\"ahler manifold then $g$ is an extremal metric (i.e., is a critical point of the Calabi functional) if and only if the Hamiltonian vector field $X_K$ associated with the scalar curvature $R$ is a Killing vector field (see \cite[Theorem 4.2]{Szekelyhidi14}). In the non-compact case the Calabi functional does not make sense and we take the later property as the definition of extremal metric.

\begin{definition}
We say that a K\"ahler surface $(M^2,g,\Omega,J)$ is an {\bf extremal K\"ahler surface}  if the Hamiltonian vector field $X_K$ associated with the Gaussian curvature $K$ is an infinitesimal isometry. 
\end{definition}

The problem of classification of extremal K\"ahler surfaces leads naturally to a Cartan's realization problem as follows. Given such a surface $(M^2,g,\Omega,J)$ we first pass to the unitary frame bundle $\pi:\B_{U(1)}(M)\to M$, where we have the tautological 1-form $\theta\in \Omega^1(\B_{U(1)}(M),\Cc)$  and the Levi-Civita connection 1-form $\omega\in \Omega^1(\B_{U(1)}(M),\uu(1))$. They satisfy the structure equations
\[
\left\{
\begin{array}{l}
\d \theta=- \omega \wedge \theta\\
\d \omega= \frac{K}{2}\, \theta \wedge \bar{\theta}
\end{array}
\right.
\]
where we use $K$ to denote both the Gaussian curvature and its pullback to $\B_{U(1)}(M)$.

The pullback of the symplectic form $\Omega$ to $\B_{U(1)}(M)$ is given by
\[ \pi^*\Omega=-\frac{i}{2}\, \theta\wedge\bar{\theta}. \]
Hence, if $\widetilde{X}_K$ is the lift of $X_K$ to $\B_{U(1)}(M)$, we then have
\[ \frac{i}{2}\, i_{\widetilde{X}_K} (\theta\wedge\bar{\theta})= -i_{\widetilde{X}_K}\pi^*\Omega=-\d K. \]
Letting $T:\B_{U(1)}(M)\to\Cc$ be the complex-valued function $T:=\frac{i}{2}\theta(\widetilde{X}_K)$, the last equation reads
\[ -\d K=\bar{T}\theta+T\bar{\theta}. \]

Since $\widetilde{X}_K$ is the lift of $X_K$, we have
\[ \Lie_{\widetilde{X}_K}\theta=0, \]
and this, together with the structure equations, yields
\begin{align*} 
\d T&=\frac{i}{2}\,\d\, i_{\widetilde{X}_K}\theta
      =-\frac{i}{2}\, i_{\widetilde{X}_K}\d\theta\\
      &=\frac{i}{2}\, i_{\widetilde{X}_K}(\omega\wedge\theta)=U\theta-T\omega,
\end{align*}
where $U:\B_{U(1)}(M)\to \Rr$ is the real-valued function $U:=\frac{i}{2}\omega(\widetilde{X}_K)$.

The assumption that $X_K$ is an infinitesimal isometry is equivalent to
\[ \Lie_{\widetilde{X}_K}\omega=0. \]
So differentiating $U$ we obtain
\begin{align*} 
\d U&=\frac{i}{2}\,\d\, i_{\widetilde{X}_K}\omega
      =-\frac{i}{2}\, i_{\widetilde{X}_K}\d\omega\\
      &=- \frac{i}{4} K\, i_{\widetilde{X}_K}(\theta\wedge\bar{\theta})
      =\frac{1}{2}K\d K=-\frac{K}{2}(\bar{T}\theta+T\bar{\theta}).
\end{align*}

In other words, finding an extremal K\"ahler surface is the same thing as finding a $U(1)$-structure $P$ with tautological form $\theta\in\Omega^1(P, \Cc)$ and connection 1-form $\omega \in \Omega^1(P, i\Rr)$, together with a map $(K,T,U):P\to \Rr\times \Cc\times \Rr$, satisfying
\begin{equation}
\label{eq:EK}
\left\{
\begin{array}{l}
\d \theta=- \omega \wedge \theta\\
\d \omega= \frac{K}{2}\, \theta \wedge \bar{\theta}\\
\d K=-(\bar{T}\theta+T\bar{\theta})\\
\d T=U\theta-T\omega\\
\d U=-\frac{K}{2}(\bar{T}\theta+T\bar{\theta})
\end{array}
\right.
\end{equation}
The manifold $M=P/U(1)$ will carry the desired extremal K\"ahler metric.

\begin{remark}
\label{rem:EK:metrics}
A Bochner-K\"ahler metric is a K\"ahler metric for which the Bochner tensor vanishes. In (real) dimension 2, it does not make sense to talk about the Bochner tensor. However, Bochner-K\"ahler metrics form a special class of extremal metrics, and Bryant in \cite{Bryant} takes ``extremal'' as the definition of a Bochner-K\"ahler metric in dimension 2.
\end{remark}

%%%%%%%%%%%%%%%%%%%%%%%%%%%%%%
\subsection{Cartan's realization problem and Lie algebroids}
\label{sec:Cartan:Realization}
%%%%%%%%%%%%%%%%%%%%%%%%%%%%%%

Consider Cartan data $(G, X,c,R,F)$ as defined in the Introduction. Then we consider the trivial vector bundle $A\to X$ with fiber $\Rr^n\oplus\gg$, and we introduce:
\begin{itemize}
\item a bundle map $\rho:A\to TX$:
\[ (u,\al)\stackrel{\rho}{\mapsto} F(u) + \psi(\al) \]
(recall that $\psi:X\times\gg\to TX$ denotes the infinitesimal $G$-action);
\item an $\Rr$-bilinear, skew-symmetric bracket $[\cdot,\cdot]$ on the space of sections $\Gamma(A)$. It is defined first on constant sections $(u,\al), (v, \be) \in \Gamma(A)$ by
\[[(u,\al), (v, \be)] = (\al\cdot v - \be\cdot u - c(u,v), [\al,\be]_\gg - R(u,v)),\]
and then extended to any pair of sections by imposing the Leibniz identity
\[ [s_1,fs_2]=f[s_1,s_2]+(\Lie_{\rho(s_1)} f)\ s_2,\]
for any sections $s_1,s_2\in\Gamma(A)$ and smooth function $f\in C^\infty(X)$.
\end{itemize}
The triple $(A,[\cdot,\cdot],\rho)$ will be a Lie algebroid provided the bracket satisfies the Jacobi identity.  

For an alternative, more natural approach, recall that one can also define a Lie algebroid structure on a vector bundle $A\to M$ as a degree 1 derivation $\d_A:\Omega^\bullet(A)\to \Omega^{\bullet+1}(A)$ satisfying $\d_A^2=0$. Here, $\Omega^k(A):=\Gamma(\wedge^k A^*)$ denotes the space of $A$-forms of degree $k$. Given Cartan Data $(G, X,c,R,F)$ one defines a differential $\d_A$ on the trivial vector bundle $A\to X$ with fiber $\Rr^n\oplus\gg$ as follows:
\begin{itemize}
\item for any smooth function $f\in \Omega^0(A)=C^\infty(X)$ we set
\[ \d_A f:=\langle \d f, F\rangle \theta-\langle \d f,\psi\rangle \omega; \]
\item $\theta=\pr_\Rr^n:A\to \Rr^n$ and $\omega:=\pr_\gg:A\to \gg$ together give a $C^\infty(X)$-basis for $\Omega^1(A)$ and we set
\begin{align*}
\d_A \theta&:=c ( \theta \wedge \theta) - \omega \wedge \theta,\\
\d_A \omega&:= R( \theta \wedge \theta) - \omega\wedge\omega;
\end{align*}
\item we extend $\d_A$ to any $A$-forms by imposing the Leibniz identity
\[ \d_A (\al\wedge \be)=\d_A\al\wedge\be+(-1)^{\deg\al} \al\wedge\d_A\be. \]
\end{itemize}
This will define a Lie algebroid structure on $A$ if and only if $\d_A^2=0$. 
\medskip

In conclusion, the triple $(A,[\cdot,\cdot],\rho)$ allows us to make sense of the structure equations \eqref{eq:structure:1} and \eqref{eq:structure:2} \emph{without knowing} the manifold $\B_G(M)$. We will see later that given Cartan Data $(G, X,c,R,F)$ then $(A,[\cdot,\cdot],\rho)$ is a Lie algebroid if and only if for any $x\in X$ there exists a solution $(\B_G(M), (\theta,\omega), h)$ of the associated Cartan's Realization Problem with $x \in \Im h$ (see Theorem \ref{thm:Cartan:G:alg:A} and Corollary \ref{cor:local:existence:2}).

%
%In fact, we have the following result of Bryant \cite[Appendix]{Bryant}):
%
%\begin{prop}
%Given Cartan Data $(G, X,c,R,F)$ such that for any $x\in X$ there exists a solution $(\B_G(M), (\theta,\omega), h)$ of the associated Cartan's Realization Problem with $x \in \Im h$, then $(A,[\cdot,\cdot],\rho)$ is a Lie algebroid.
%\end{prop}
%
%The proof of this proposition given by Bryant consists in working in local coordinates around $x\in X$ and observing that the differential equations that $(c,R,F)$ must satisfy for $(A,[\cdot,\cdot],\rho)$ to be a Lie algebroid are the same equations that one deduces using $\d^2=0$ by differentiating the structure equations \eqref{eq:structure:1} and \eqref{eq:structure:2} for a solution $(\B_G(M), (\theta,\omega), h)$ with $x \in \Im h$.
%
%We will come back to this result later after we introduce the language of $G$-structure groupoids and algebroids.

%%%%%%%%%%%%%%%%%%%%%%%%%%%%%%
\subsection{$G$-action, connections and tautological form}
\label{sec:example:EK:surfaces:2}
%%%%%%%%%%%%%%%%%%%%%%%%%%%%%%
The Lie algebroid $A\to X$ associated with Cartan Data $(G, X,c,R,F)$ is rather special. It comes equipped with:
\begin{itemize}
\item an action of the structure group $G$ on the bundle $A=X\times (\Rr^n\oplus\gg)\to X$ given by
\[ (x,u,\al)g=(xg,g^{-1}u,g^{-1}\al g),\]
where we use the $G$-action on $X$ and the fact that $G \subset \GL(n,\Rr)$.
\item algebroid 1-forms $\theta\in\Omega^1(A;\Rr^n)$ and $\omega\in\Omega^1(A;\gg)$ given by the projections from $A$ to $\Rr^n$ and $\gg$.
\end{itemize}
One should think of $\theta$ and $\omega$ as the tautological and the connection 1-forms on $A$. As we shall see, these make $A\to X$ into a $G$-structure algebroid with connection.
\medskip

In the next sections we will define precisely what are $G$-structure algebroids and groupoids, as well as connections on them, and how one can use  them to solve Cartan's realization problem.

%%%%%%%%%%%%%%%%%%%%%%%%%%%%%%
\subsection{The Lie algebroid of extremal K\"ahler surfaces}
\label{sec:example:EK:surfaces:1}
%%%%%%%%%%%%%%%%%%%%%%%%%%%%%%

Comparing the structure equations \eqref{eq:EK} with the general structure equations \eqref{eq:structure:1}, \eqref{eq:structure:2}, we can extract the relevant Cartan data $(G, X,c,R,F)$. First, on $X=\Rr\times \Cc\times \Rr$ we denote by $(K,T,U)$ global coordinates and we have the right $U(1)$-action
\[  (K,T,U)g=(K,g^{-1}T,U), \quad g\in U(1), \]
with associated infinitesimal action $\psi: X \times \mathfrak{u}(1) \to TX$
\[ \psi(\al)|_{(K,T,U)}=\left(0,-\al T,0\right), \quad \al\in  \mathfrak{u}(1).\]
The torsion $c: X \to \hom(\wedge^2\Cc, \Cc)$ vanishes identically while the curvature is the $U(1)$-equivariant map $R: X \to \hom(\wedge^2\Cc, \mathfrak{u}(1))$ given by
\[ R(K,T,U)(z\wedge w )=\frac{K}{2}(z\bar{w}-\bar{z}w).\]
On the other hand, the $U(1)$-equivariant map $F: X \times \Cc \to TX$, is given by
\[ F((K,T,U),z)=\left(-T\bar{z}-\bar{T}z,Uz,-\frac{K}{2}T\bar{z}-\frac{K}{2}\bar{T}z\right). \]

Hence, we can write the explicit Lie algebroid associated with the classification problem of extremal K\"ahler surfaces. We have the trivial bundle $A=X\times(\Cc\oplus\uu(1))\to X$, with base $X=\Rr\times \Cc\times \Rr$. The Lie bracket is defined on constant sections by
\[ [(z,\al),(w,\be)]|_{(K,T,U)}:=(\al w-\be z, -\frac{K}{2}(z\bar{w}-\bar{z}w)), \]
while the anchor is given by
\[ \rho(z,\al)|_{(K,T,U)}:=\left(-T\bar{z}-\bar{T}z,Uz-\al T,-\frac{K}{2}T\bar{z}-\frac{K}{2}\bar{T}z\right). \]
One can check easily that the bracket satisfies the Jacobi identity, so this is indeed a Lie algebroid.

%%%%%%%%%%%%%%%%%%%%%%%%%%%%%%
\subsection{Surface Metrics with $|\nabla K|=1$}
\label{sec:example:surfaces:gradient}
%%%%%%%%%%%%%%%%%%%%%%%%%%%%%%
Following Bryant \cite[Example 5.1]{Bryant:notes}, consider now the problem of classification of oriented Riemannian surfaces $(M,g)$ for which the Gaussian curvature $K$ satisfies
\[ |\nabla K|=1. \]
Passing to the orthogonal frame bundle $\pi:\B_{SO(2)}(M)\to M$, the tautological 1-form $\theta\in \Omega^1(\B_{SO(2)}(M),\Rr^2)$ and the Levi-Civita connection 1-form $\omega\in \Omega^1(\B_{SO(2)}(M),\so(2))$ satisfy the structure equations
\[\d \theta^1=- \omega \wedge \theta^2,\qquad \d \theta^2=\omega \wedge \theta^1,\qquad \d \omega= K\, \theta^1 \wedge \theta^2.\] 
Let us attempt a differential analysis, as in the previous example. Differentiating the Gaussian curvature, we obtain
\[ \d K=K_1\theta^1+K_2\theta^2.\]
Our restriction on the curvature now says that $K_1^2+K_2^2=1$, so we set $K_1=\cos\phi$ and $K_2=\sin\phi$. From $\d^2 K=0$, we see that we must have
\[ \d\phi=\omega+J_1(-\sin\phi\, \theta^1+\cos\phi\, \theta^2),\]
for some new invariant function $J_1:\B_{SO(2)}(M)\to \Rr$. If we next impose $\d^2\phi=0$, we obtain
\[ \d J_1=-(K+J_1^2)(\cos\phi\, \theta^1+\sin\phi\, \theta^2)+J_2(-\sin\phi\, \theta^1+\cos\phi\, \theta^2), \]
for a new invariant function $J_2:\B_{SO(2)}(M)\to \Rr$. We can again impose $\d J_2=0$ and this will lead to a new invariant $J_3$. In fact, this process will produce an infinite sequence of invariant functions $J_k:\B_{SO(2)}(M)\to \Rr$ satisfying
\[ \d J_k=F_k(K,J_1,\dots,J_k)(\cos\phi\, \theta^1+\sin\phi\, \theta^2)+J_{k+1}(-\sin\phi\, \theta^1+\cos\phi\, \theta^2),\]
for some polynomial functions $F_k$.

The theory developed in this paper does not apply to this kind of problem, which has an infinite dimensional space of invariants. This requires allowing for $G$-manifolds $X$ which are profinite dimensional. These pose considerably issues and will be dealt with in future work.

%In order to solve Cartan Realization problem one needs to integrate the Lie algebroid $A$ to a Lie groupoid $\G$, and here comes our first key observation: 

%%%%%%%%%%%%%%%%%%%%%%%%%%%%%%%%%%%%%%%%%%%%%%%%%%%%%%%%%%%%%%%%%%
%%%%%%%%%%%%%%%%%%%%%%%%%%%%%%%%%%%%%%%%%%%%%%%%%%%%%%%%%%%%%%%%%%
\section{$G$-structure groupoids and connections}
\label{sec:G:grpds}
%%%%%%%%%%%%%%%%%%%%%%%%%%%%%%%%%%%%%%%%%%%%%%%%%%%%%%%%%%%%%%%%%%
%%%%%%%%%%%%%%%%%%%%%%%%%%%%%%%%%%%%%%%%%%%%%%%%%%%%%%%%%%%%%%%%%%

In this section we will introduce $G$-principal groupoids and $G$-structure groupoids, which allow to parameterize families of $G$-principal bundles and $G$-structures, respectively. We will also introduce connections on them. The main property of these families, compared with arbitrary families of $G$-principal bundles and $G$-structures, is that they have infinitesimal counterparts (to be introduced in the next section).
\medskip

{\bf Convention.} Henceforth we will assume that $G$ is a compact, connected, Lie group. This covers most applications, but not all. This hypothesis is made in order to simplify the exposition and can be removed by including properness assumptions in various definitions and propositions. We will comment on these as we go along.

%%%%%%%%%%%%%%%%%%%%%%%%%%%%%%%
\subsection{Principal bundles and $G$-structures over orbifolds}
%%%%%%%%%%%%%%%%%%%%%%%%%%%%%%%

The solutions to Cartan's Realization Problem are naturally $G$-structures over orbifolds. We collect here some basic facts about principal bundles and $G$-structures over an orbifold $M$. We will consider only \emph{effective} orbifolds for which a good general reference is \cite{MoerdijkMrcun}. Although $G$-structures over orbifolds are not treated in \cite{MoerdijkMrcun} the following facts can easily be proved with the techniques developed there, by mimicking the usual proofs for $G$-structures over smooth manifolds.

We will work mainly with right actions of Lie groups. By a {\bf principal action} of a Lie group $G$ on a manifold $P$ we mean a proper, effective and locally free action $P\rtimes G\to P$. According to our convention, $G$ is assumed to be compact so properness is automatic. In this case, $M=P/G$ has a natural structure of an effective orbifold (see  \cite{MoerdijkMrcun}), and we call $\pi:P\to M$ a {\bf $G$-principal bundle} over the orbifold $M$. Principal bundle connections are defined in the usual way by horizontal distributions which are preserved by the the $G$-action. They are determined by a connection 1-form $\omega\in\Omega^1(P;\gg)$ which satisfies the two properties:
\begin{itemize}
\item \emph{vertical}: 
\[ T_p P=\ker\d_p\pi\oplus \ker\omega_p\quad\text{and}\quad \omega(\al_P)=\al,\ \forall\al\in\gg; \]
\item \emph{$G$-equivariance}: 
\[ g^*\omega=\Ad_{g^{-1}}\cdot\omega,\ \forall g\in G. \]
\end{itemize}

For a $n$-dimensional orbifold $M$ we denote by $\pi: \B(M) \to M$ its {\bf frame bundle}. It is a principal $\GL(n,\Rr)$-bundle: $\B(M)$ is a smooth manifold (\cite[Section 2.4]{MoerdijkMrcun}) and it has a proper, effective, locally free action of $\GL(n,\Rr)$, such that $M=\B(M)/\GL(n,\Rr)$. The frame bundle has a canonical $\Rr^n$-valued {\bf tautological 1-form} $\theta \in \Omega^1(\B(M), \Rr^n)$ defined by
\[\theta_p(\xi) = p^{-1}(\d_p\pi(\xi)),\quad (p\in\B(M),\ \xi\in T_p \B(M)).\]
Any local diffeomorphism of orbifolds $\varphi: M \to N$ lifts to a bundle map of the associated frame bundles
\[\tilde{\varphi}: \B(M) \to \B(N), \quad \tilde{\varphi}(p)(v) = \d\varphi (pv),\]
preserving the tautological 1-forms, $\tilde{\varphi}^*\theta_N=\theta_M$.

More generally, if $G \subset \GL(n,\Rr)$ is a closed subgroup, we consider \textbf{$G$-structures} over $M$: they are given by a principal $G$-subbundle $\B_G(M)\subset \B(M)$. We say that two $G$-structures $\B_G(M)$ and $\B_G(N)$ are \textbf{(locally) equivalent} if there exists a (locally defined) diffeomorphism of orbifolds $\varphi: M \to N$ such that 
\[ \B_G(N) = \tilde{\varphi}(\B_G(M)). \]
For a general $G$-structure $\B_G(M)$ its tautological form is obtained by restricting the tautological form of $\B(M)$ and will still be denoted by the same letter $\theta$. It is characterized by the following two properties:
\begin{itemize}
\item \emph{strong horizontality}: 
\[ \theta_p(\xi)=0\text{ if and only if } \xi=\al_P|_p, \text{ for some }\al\in\gg;\]
\item \emph{$G$-equivariance}: 
\[ g^*\theta=g^{-1}\cdot\theta,\ \forall g\in G. \]
\end{itemize}

We will make use of the following well-known characterization of $G$-structures. The proof for the case of smooth manifolds extends easily to the case of orbifolds.

\begin{prop}
\label{prop:G-structures}
Let $G$ be a closed subgroup of $\GL(n, \Rr)$, and let $P\to M$  be a $G$-principal bundle over an orbifold, equipped with a pointwise surjective $\Rr^n$-valued form $\widetilde{\theta}\in\Omega^1(P;\Rr^n)$ which is both strongly horizontal and $G$-equivariant. Then there is a unique isomorphism $P\simeq \B_G(M)$ with a $G$-structure over $M$ which maps $\widetilde{\theta}$ to the tautological form $\theta$.
\end{prop}

%%%%%%%%%%%%%%%%%%%%%%%%%%%%%%%
\subsection{$G$-principal groupoids}
%%%%%%%%%%%%%%%%%%%%%%%%%%%%%%%

We denote by $\G\tto X$ a Lie groupoid with space of arrows $\G$ and space of objects $X$. We use the letters $\s$ and $\t$ for the source and target maps, the symbol $1_x$ for the identity arrow at $x\in X$ and $\gamma_1\cdot\gamma_2$ for the product of the composable arrows $(\gamma_1,\gamma_2)\in \G^{(2)}:=\G\timesst\G$. Also, we denote by $T^{\s} \G$ the tangent distribution to the source fibers.

\begin{definition}
\label{def:principal:G:action}
A {\bf $G$-principal action on a Lie groupoid $\G\tto X$ along its s-fibers} is given by a principal action $\G\times G\to \G$, preserving source fibers and satisfying
\begin{equation}
\label{eq:principal:action} 
(\gamma_1\cdot \gamma_2)\, g= (\gamma_1\, g)\cdot \gamma_2,\quad \forall (\gamma_1,\gamma_2)\in\G^{(2)},\ g\in G.
\end{equation}
We will call such an action simply a {\bf $G$-principal action} and we will call $\G$ a {\bf $G$-principal groupoid}. By a {\bf morphism of $G$-principal groupoids} we mean a groupoid morphism $\Phi:\G_1\to\G_2$ which is $G$-equivariant.
\end{definition}

For a $G$-principal groupoid $\G\tto X$ each source fiber $\s^{-1}(x)$ is a $G$-principal bundle over the orbifold $M=\s^{-1}(x)/G$. The main reasons we are interested in such groupoids are (i) they parameterize families of $G$-principal bundles and (ii) they have infinitesimal versions, namely their Lie algebroids.

\begin{remark}
If we write a principal action as a left action $G\times \G\to \G$, then \eqref{eq:principal:action} becomes
\[ g\, (\gamma_1\cdot \gamma_2)=(g\,\gamma_1)\cdot \gamma_2. \]
This looks more natural than \eqref{eq:principal:action}, but since our principal bundles are equipped with right actions it is more efficient to define $G$-principal groupoids using right actions. Another possibility would be to consider right $G$-actions along $\t$-fibers, in which case \eqref{eq:principal:action} becomes
\[ (\gamma_1\cdot \gamma_2)\, g= \gamma_1\cdot (\gamma_2\, g). \]
However, we decided to follow the more common convention of defining algebroids as tangent spaces to $\s$-fibers along identity sections, so we need to consider actions along $\s$-fibers.
\end{remark}

\begin{remark}
It is more common to consider (principal) actions of Lie groups on a Lie groupoid of different nature than Definition \ref{def:principal:G:action} -- see for example \cite{BNZ20} in the context of stacks. Still, actions as in this definition have been considered before in the literature, including in algebraic geometry, and we refer to \cite[Section 2.1]{BIL19} for a historical account.
\end{remark}

There is a slightly different point of view on $G$-principal groupoids, which will be useful when we introduce their infinitesimal versions.

First, given a $G$-principal groupoid $\G\tto X$ we have a $G$-action on $X$ defined by
\begin{equation}
\label{eq:X:action}
X\times G\to X,\quad x\, g:=\t(1_x\, g),
\end{equation}
for which $\s:\G\to X$ is $G$-invariant and $\t:\G\to X$ is $G$-equivariant. 
 
Next, we recall that given a (right) $G$-action on a manifold $X$ we can form the action Lie groupoid $X\rtimes G\tto X$, where an arrow is a pair $(x,g)$ with source $x$ and target $xg$, and composition of arrows is given by
\[ (y,h)\cdot (x,g)=(x,gh),\quad \text{if }y=xg. \]

Finally, we define the groupoid morphism
\begin{equation}
\label{eq:action:morphism}
\iota: X\rtimes G\to \G,\quad (x,g)\mapsto 1_x\, g.
\end{equation}

Given such a groupoid morphism it will be useful to adopt the following notation. We will say that $\iota: X\rtimes G\to \G$ is
\begin{enumerate}[(a)]
\item {\bf effective} if given $e\not=g\in G$ there exists an $x\in X$ with $\iota(x,g)\not=1_x$;
\item {\bf injective} if for any $x\in X$ and $g\in G$, $\iota(x,g)=1_x$ if and only if $g=e$;
\item {\bf locally injective} if for any $x\in X$ there is an open $e\in U_x\subset G$ such that $g\in U_x$ and $\iota(x,g)=1_x$ if and only if $g=e$.
\end{enumerate}

\begin{prop}
\label{prop:principal:G:action}
Let $\G\tto X$ be a Lie groupoid. If $\G\times G\to \G$ is a $G$-principal action, then \eqref{eq:X:action} defines a $G$-action on $X$ and \eqref{eq:action:morphism} defines 
an effective, locally injective, Lie groupoid morphism. Moreover, the $G$-action on $\Gamma$ takes the form
\begin{equation}
\label{eq:principal:action:formula}
\G\times G\to \G,\quad \gamma\, g= \iota(\t(\gamma),g)\cdot \gamma.
\end{equation}

Conversely, given an action $X\times G\to X$ and an effective, locally injective, Lie groupoid morphism $\iota: X\rtimes G\to \G$, \eqref{eq:principal:action:formula} defines a $G$-principal action on $\G$.
\end{prop}

\begin{proof}
Let $\G \times G \to \G$ be a G-principal action. To check that \eqref{eq:X:action} defines a $G$-action observe first that that $xe = x$, where $e \in G$ denotes the identity element. We must also verify that $(xg)h = x(gh)$, but this is equivalent to
\[\t(1_{\t(1_xg)}h) = \t((1_xg)h).\]
Condition \eqref{eq:principal:action} now gives
\[\t((1_xg)h) = \t((1_{\t(1_xg)}\cdot1_xg)h) = \t((1_{\t(1_xg)}h)\cdot1_xg) = \t(1_{\t(1_xg)}h),\]
and therefore we have a $G$-action on $X$. Since the $G$-action on $\G$ is effective and locally free, it is now easy to check from \eqref{eq:action:morphism} that the map $\iota: X \rtimes G \to \G$ is an effective, locally injective, Lie groupoid morphism. 

%Let us show that it is a proper map. Assume that $K \subset \G$ is compact and denote by $\Psi: \G \times G \to \G \times \G$ the proper map $(\gamma, g) \mapsto (\gamma, \gamma g)$. Then 
%\[\iota^{-1}(K) = (\mathbf{u}\times \mathrm{Id}_G)^{-1}(\Psi^{-1}(\mathbf{u}(\s(K))\times K)),\]
%where $\mathbf{u}: X \to \G$ denotes the identity embedding. It follows that $\iota^{-1}(K)$ is compact in $X \times G$. 

Conversely, assume that $G$ acts on $X$ on the right, $\iota: X \rtimes G \to \G$ is an effective, locally injective, Lie groupoid morphism, and define $\gamma g = \iota(t(\gamma),g)\cdot \gamma$. It is straightforward to verify that this is indeed an effective, locally free, action along the $\s$-fibers of $\G$ satisfying $(\gamma \cdot \gamma')g = (\gamma g)\cdot \gamma'$. 

%To see that the action is proper, assume that $(\gamma_n, g_n)$ is a sequence such that $(\gamma_n, \iota(t(\gamma_n), g_n)\cdot \gamma_n)$ converges to $(\gamma, \gamma')$. It follows that $\gamma_n$ converges to $\gamma$, and $\iota(t(\gamma_n), g_n)$ converges to $\gamma' \cdot \gamma^{-1}$. Then, since $\iota$ is proper it follows that $g_n$ converges and therefore the action of $G$ on $\Gamma$ is proper. \textcolor{blue}{THIS LAST PART ONLY PROVES THE CLAIM WHEN $\G$ IS HAUSDORFF.}
\end{proof}

\begin{remark}
In general, the $G$-action on $X$ is neither free, nor effective.  Note also that all actions involved are proper since we are assuming that $G$ is compact. When $G$ is not compact, one needs to assume that the action of $G$ on the source fibers is proper. When $\Gamma$ is a Hausdorff groupoid, this is equivalent to the morphism $\iota: X \rtimes G \to \G$ being a proper map.

%\textcolor{blue}{nor proper. In most of our examples, however, the group $G$ will be compact and this will not be an issue. There are two places in this paper where the non-properness of this action will cause problems that lead to the necessity of extra hypothesis. The first is in the proof of local existence of solutions to a realization problem where a slice to the action is needed (Theorem \ref{thm:local:G:int}). The second is in the description of ``smaller models" for the moduli space of solutions (Section \ref{sec:morita})}. \textcolor{red}{It makes the source $\s:\G\to X$ a $G$-invariant map and the target $\t:\G\to X$ a $G$-equivariant map. THIS IS REPEATED}
\end{remark}

The previous proposition shows that we can define the $G$-principal action on $\G$ by specifying first a $G$-action on $X$ and then a groupoid morphism $\iota: X\rtimes G\to \G$. We will often use this alternative point of view and we call $\iota$ the {\bf action morphism} of the $G$-principal groupoid.  For example, a morphism of $G$-principal groupoids can be alternatively characterized as a morphism of groupoids
\[
\xymatrix@R=20pt{
\G_1\ar[r]^{\Phi}\ar@<0.2pc>[d] \ar@<-0.2pc>[d]  & \G_2\ar@<0.2pc>[d] \ar@<-0.2pc>[d]  \\
X_1\ar[r]_{\phi} & X_2
}
\]
which intertwines the actions morphisms,
\[ \Phi\circ \iota_1=\iota_2\circ (\phi\times I). \]
% One can consider more general morphisms between principal groupoids with different structure groups, but we will not need them in this work.

The action morphism $\iota: X\rtimes G\to \G$ also allows us to define a different action of $G$ on $\Gamma$ this time {\bf an action by inner automorphisms}
\begin{equation}
\label{eq:inner:G:action}
\G\times G\to\G, \quad \gamma \odot g:=\iota(\t(\gamma),g)\cdot \gamma \cdot \iota(\s(\gamma), g)^{-1}.
\end{equation}
In general, the inner action of $G$ on $\Gamma$ \emph{does not determine} the original $G$-action on $\Gamma$. One can have, for example, a non-trivial action morphism whose associated inner action is trivial.  

\begin{example}
\label{ex:principal:bundle:grpd}
Any $G$-principal bundle $\pi:P\to M$ can be viewed as a $G$-principal groupoid. Indeed, a $G$-principal action on the manifold $P$ is the same thing as
a $G$-principal action on the pair groupoid $\G=P\times P\tto P$. The two actions are related by
%Consider a $G$-principal bundle $\pi:P\to M$. The pair groupoid $\G=P\times P\tto P$ has a natural $G$-principal action given by:
\[  \G\times G\to \G,\quad (p_1,p_2)\, g:=(p_1g,p_2). \]
and condition \eqref{eq:principal:action} holds since we have
\[ ((p_1,p_2)\cdot (p_2,p_3))\, g=(p_1,p_3)\, g=(p_1\, g,p_3)=((p_1,p_2)\, g)\cdot (p_2,p_3). \]
The action morphism of this $G$-principal groupoid is given by
\[ \iota: P\rtimes G\to P\times P,\quad \iota(p,g)=(p\, g,p),  \]
and the associated inner action by groupoid automorphisms is given by
\[ (p_1,p_2)\odot g=(p_1 g,p_2g). \]
In this case all source fibers $\s^{-1}(x)$ are diffeomorphic to $P$ and $\s^{-1}(x)/G=M$, so this groupoid parameterizes a trivial family of $G$-principal bundles with fiber $P$ and base $M$.
\end{example}

\begin{example}
\label{ex:morphism:s:fibers:principal:grpd}
As we observed above, for a $G$-principal groupoid $\G\tto X$, each source fiber is a $G$-principal bundle $\s^{-1}(x)\to \s^{-1}(x)/G$, hence it determines (previous example) a $G$-principal groupoid $\s^{-1}(x)\times \s^{-1}(x)\tto \s^{-1}(x)$. These two $G$-principal groupoids are related by the morphism of $G$-principal groupoids covering the target map
\[
\vcenter{\vbox{ 
\xymatrix@R=20pt{
\s^{-1}(x)\times \s^{-1}(x) \ar[r]\ar@<0.2pc>[d] \ar@<-0.2pc>[d]  & \G\ar@<0.2pc>[d] \ar@<-0.2pc>[d]  \\
\s^{-1}(x)\ar[r]_{t} & X
}}}
\qquad (\gamma_1,\gamma_2)\mapsto \gamma_1\cdot \gamma_2^{-1}.
\]
\end{example}

\begin{example}
\label{ex:action:grpd}
Let $H$ be a Lie group that acts (on the left) on some manifold $X$. Then we have the associated action groupoid 
\[ \G=H\times X\tto X. \]
If $G\subset H$ is a closed subgroup, we can define a $G$-principal action on $\G$ by setting
\[ (h,x)\, g:=(g^{-1}h,x). \]
One checks immediately that the condition \eqref{eq:principal:action} holds. 
The action morphism of this $G$-principal groupoid is given by
\[ \iota: X\rtimes G\to H\times X,\quad \iota(x,g)=(g^{-1}, x),  \]
and has an associated inner action by groupoid automorphhisms
\[ (h,x)\odot g=(g^{-1}hg,x). \]
This groupoid parameterizes a trivial family of $G$-principal bundles with fiber $H$ and base $H/G$.
\end{example}

\begin{example}
\label{ex:bundle:groups}
There exists a (non-Hausdorff) bundle of Lie groups over the real line $\G\tto \Rr$ with fibers
\[
\s^{-1}(x)=\t^{-1}(x)\simeq
\left\{
\begin{array}{lr}
\SO(3,\Rr),\qquad &\text{ if }x>0,\\
\SO(2,\Rr)\ltimes \Rr^2,&\text{ if }x=0,\\
\SL(2,\Rr),\qquad &\text{ if }x<0.
\end{array}
\right.
\]
One can embed $\SO(2,\Rr)$ as a subgroup of each of these groups in such a way that one obtains a smooth groupoid morphism $\iota:\Rr\rtimes \SO(2,\Rr)\to \Gamma$, where
$\SO(2,\Rr)$ acts trivially on the real line. Then $\iota$ is an action morphism so that $\Gamma\tto\Rr$ is a $\SO(2,\Rr)$-principal groupoid. This groupoid parameterizes a non-trivial family of $\SO(2,\Rr)$-principal bundles with base manifolds $\Ss^2$ $(x>0)$, $\Rr^2$ $(x=0)$ or $ \mathbb{H}^2$ $(x<0)$.
\end{example}

%%%%%%%%%%%%%%%%%%%%%%%%%%%%%%%
\subsection{Connections on $G$-principal groupoids}
%%%%%%%%%%%%%%%%%%%%%%%%%%%%%%%

Virtually all the constructions that one does with ordinary $G$-principal bundles, such as connections, associated bundles, etc., extend to $G$-principal groupoids. In this paper connections play a crucial role.

\begin{definition}
A {\bf connection} on a $G$-principal groupoid $\G\tto X$ is a right-invariant distribution $\H\subset T^{\s} \G$ 
 \[ \H_{\tau\cdot\gamma}=(R_\gamma)_*\H_\tau,\quad \forall (\tau,\gamma)\in \G\timesst\G, \]
satisfying:
\begin{enumerate}[(i)]
\item $\H$ is horizontal:
\[  T^{\s}_\gamma \G=\H_\gamma\oplus T_\gamma(\gamma G),\quad \forall \gamma\in\G; \]
\item $\H$ is $G$-invariant:
\[ \H_{\gamma\, g}=g_*(\H_\gamma),\quad \forall g\in G. \]
\end{enumerate}
A {\bf morphism} of $G$-principal groupoids with connection is a $G$-principal groupoid morphism $\Phi:\G_1\to\G_2$ such that $\Phi_*(\H_1)\subset\H_2$.
\end{definition}

Hence, a connection $\H$ on $G$-principal groupoid consists of a smooth family of (usual) connections $\{\H_x:x\in X\}$ on the $G$-principal bundle $\s^{-1}(x)\to \s^{-1}(x)/G$.
\medskip

We can also define connections via connections 1-forms. For that, recall that a right-invariant $k$-form on a Lie groupoid $\G\tto X$ with values in a vector space $V$ is an s-foliated, $V$-valued, differential form $\Omega\in\G(\wedge^kT^\s\G,V)$ satisfying
\[ (\d R_\gamma)^*\Omega_\tau=\Omega_{\tau\gamma}, \quad \forall (\tau,\gamma)\in \G\timesst\G. \]
We write $\Omega^k_R(\G,V)$ to denote the set of $V$-valued, right-invariant $k$-forms. Exactly as for ordinary principal bundle connections, one has the following result.

\begin{prop}
Given a connection $\H$ on a $G$-principal groupoid $\G\tto X$ there is a unique right-invariant $\gg$-valued 1-form $\Omega\in\Omega^1_R(\G;\gg)$ with $\Ker\Omega=\H$ and satisfying:
\begin{enumerate}[(i)]
\item $\Omega$ is vertical:
\[ \Omega(\al_\G)=\al,\quad \forall \al\in\gg. \]
\item $\Omega$ is $G$-equivariant:
\[ g^*\Omega=\Ad_{g^{-1}}\cdot \Omega,\quad \forall g\in G. \]
\end{enumerate}
Conversely, given a right-invariant 1-form $\Omega\in\Omega^1_R(\G;\gg)$ satisfying (i) and (ii), its kernel $\H=\Ker\Omega$ defines a connection.
\end{prop}

We call  $\Omega\in\Omega^1_R(\G;\gg)$ the {\bf connection 1-form} of $\H$. The restriction of the connection 1-form  $\Omega\in\Omega^1_R(\G;\gg)$ to the source fiber $\s^{-1}(x)$ gives the (usual) connection 1-form $\Omega_x$ of the connection $\H_x$ on the $G$-principal bundle $\s^{-1}(x)\to \s^{-1}(x)/G$.

\begin{example}
\label{ex:principal:bundle:connection:grpd}
Under the correspondence between $G$-principal bundles $P\to M$ and $G$-principal groupoids $P\times P\tto P$ discussed in Example \ref{ex:principal:bundle:grpd},
specifying a connection $H$ on the $G$-principal bundle $\pi:P\to M$ is the same thing as specifying a connection $\H$ on the pair groupoid $P\times P\tto P$. The two notions of connection are related by
\[ \H:=\{(v,0)\in T(P\times P): v\in H\}. \]
If $\omega\in\Omega^1(P;\gg)$ is the connection 1-form of the connection on $P$, then the connection 1-form of the groupoid $\G$ is the $\gg$-valued right-invariant form defined by
\[ \Omega(v,0)=\omega(v). \]
\end{example}

Henceforth, we will prefer the form point of view, so we will specify a connection by giving the connection 1-form $\Omega\in\Omega^1_R(\G;\gg)$. Morphisms of $G$-principal groupoids with connection are also easily characterized in terms of connections 1-forms. Indeed, a morphism of Lie groupoids $\Phi:\G_1\to\G_2$ maps source-fibers to source-fibers, and so induces a pullback map of $V$-valued right-invariant forms 
\[ \Phi^*:\Omega^k_R(\G_2,V)\to \Omega^k_R(\G_1,V). \]
In the case of a morphism of $G$-principal groupoids we find the following result whose (easy) proof is left to the reader.

\begin{prop} 
\label{prop:maps:connect:grpd}
Given two $G$-principal groupoids with connection $\G_1$ and $\G_2$, a morphism of groupoids $\Phi:\G_1\to\G_2$ is a morphism of $G$-principal groupoids with connections if and only if it is $G$-equivariant and preserves the connection 1-forms: $\Phi^*\Omega_2=\Omega_1$.
\end{prop}

%\begin{proof}
%Since a morphism of $G$-principal groupoids maps $G$-orbits isomorphically to $G$-orbits, we have that $\Phi_*(\al_{\G_1})=\al_{\G_2}$. Therefore: 
%\[ \Phi^*\Omega_2(\al_{\G_1})=\Omega_2(\Phi_*(\al_{\G_1}))=\Omega_2(\al_{\G_2})=\al. \]
%On the other hand, we have:
%\[ \Phi^*\Omega_2(v)=0 \quad \Leftrightarrow \quad \Omega_2(\Phi_*(v))=0 \quad \Leftrightarrow \quad \Phi_*(v)\in \H_2. \]
%Hence:
%\begin{enumerate}[(a)]
%\item if $\Phi_*(\H_1)\subset \H_2$, since the $G$-actions are locally free, we see that $\H_1=\Ker(\Phi^*\Omega_2)$ and $\Phi^*\Omega_2(\al_{\G_1})=\al$ for all $\al\in\gg$, so we must have $\Omega_1=\Phi^*\Omega_2$;
%\item if  $\Omega_1=\Phi^*\Omega_2$, then $v\in \Ker\Omega_1=\H_1$ implies that $\Phi_*(v)\in \H_2$, so $\Phi$  preserves the connections.
%\end{enumerate}
%\end{proof}

The {\bf curvature 2-form} of a connection $\Omega$ is the right-invariant, $\gg$-valued, 2-form defined by
\[ \Curv(\Omega)\in \Omega^2_R(\G;\gg),\quad \Curv(v,w):=\d \Omega(h(v),h(w)). \]
Here $h:T^{\s} \G\to \H$ denotes the projection and $\d:\Omega^k_R(\G,V)\to \Omega^{k+1}_R(\G,V)$ denotes the $\s$-foliated differential (de Rham differential along the s-fibers). The curvature form measures the failure of the horizontal distribution $\H$ in being integrable. Of course, the restriction of $\Curv(\Omega)$ to an s-fiber $\s^{-1}(x)$ is the usual curvature 2-form of the induced connection on $\s^{-1}(x)\to \s^{-1}(x)/G$. This leads immediately to

\begin{prop}
A connection $\Omega$ on a $G$-principal groupoid $\G\tto X$ satisfies:
\begin{enumerate}[(i)]
\item {\bf 1st structure equation}:
\[ \d\Omega=-\Omega\wedge \Omega +\Curv(\Omega); \]
\item {\bf 1st Bianchi's identity}:
\[ \d\Curv(\Omega)|_{\H}=0. \]
\end{enumerate}
\end{prop}

\begin{example}
\label{ex:morphism:s:fibers:principal:grpd:connection}
For an arbitrary $G$-principal groupoid $\G\tto X$ with connection $\Omega$, each source fiber $\s^{-1}(x)\to \s^{-1}(x)/G$ is a $G$-principal bundle with connection. As we saw in Example \ref{ex:principal:bundle:connection:grpd}, the pair groupoid $\s^{-1}(x)\times \s^{-1}(x) \tto \s^{-1}(x)$ is then a $G$-principal groupoid with connection, and we have a morphism of $G$-principal bundles with connections covering the target map (see Example \ref{ex:morphism:s:fibers:principal:grpd})
\[
\vcenter{\vbox{ 
\xymatrix@R=20pt{
\s^{-1}(x)\times \s^{-1}(x) \ar[r]\ar@<0.2pc>[d] \ar@<-0.2pc>[d]  & \G\ar@<0.2pc>[d] \ar@<-0.2pc>[d]  \\
\s^{-1}(x)\ar[r]_{t} & X
}}}
\qquad (\gamma_1,\gamma_2)\mapsto \gamma_1\cdot \gamma_2^{-1}.
\]
\end{example}

%%%%%%%%%%%%%%%%%%%%%%%%%%%%%%%
\subsection{$G$-structure groupoids}
%%%%%%%%%%%%%%%%%%%%%%%%%%%%%%%

Recall our standing assumption that $G$ is a compact group. If not, in the next definition we should require $G\subset \GL(n,\Rr)$ to be a closed subgroup.

\begin{definition}
Given $G\subset \GL(n,\Rr)$, a {\bf $G$-structure groupoid} consists of a $G$-principal groupoid $\G\tto X$ equipped with a pointwise surjective, right-invariant, $\Rr^n$-valued 1-form $\Theta\in\Omega^1_R(\G;\Rr^n)$ satisfying:
\begin{enumerate}[(i)]
\item $\Theta$ is strongly horizontal:
\[ \Theta_\gamma(v)=0 \quad\text{iff}\quad v=(\al_\G)|_\gamma,\text{ for some }\al\in\gg. \]
\item $\Theta$ is $G$-equivariant:
\[ g^*\Theta=g^{-1}\cdot\Theta,\quad \forall g\in G. \]
\end{enumerate}
We call $\Theta$ the {\bf tautological form} of the $G$-structure groupoid.

A {\bf morphism} of $G$-structure groupoids $\Phi:\G_1\to\G_2$ is a morphism of $G$-principal groupoids which preserves the tautological forms: $\Phi^*\Theta_2=\Theta_1$.
\end{definition}

The tautological form $\Theta$  restricts on each source fiber of $\G$ to a 1-form 
\[ \theta_x=\Theta|_{\s^{-1}(x)} \in\Omega^1(\s^{-1}(x);\Rr^n). \]
By properties (i) and (ii) in the definition, this 1-form is strongly horizontal and $G$-equivariant for the $G$-principal action on the fiber $\s^{-1}(x)$. Hence, we obtain a $G$-structure $\s^{-1}(x)\to \s^{-1}(x)/G$ over an orbifold (cf.~Proposition \ref{prop:G-structures}). Therefore, the source fibers of a $G$-structure groupoid $\G$ form a smooth family of $G$-structures.

\begin{remark}
\label{rem:condition i}
Since $\Theta$ is assumed to be right-invariant, condition (i) in the definition holds if and only if it holds along the identities. Using the action morphism $\iota:X\rtimes G\to\G$, we conclude that (i) is equivalent to 
\[ \Theta_{1_x}(v)=0 \quad\text{iff}\quad v=\d_{(x,e)} \iota(0,\al),\text{ for some }\al\in\gg. \]
\end{remark}

The definition of a $G$-structure groupoid $\G\tto X$ gives
\[ \dim( \s^{-1}(x))=n+\dim G,\quad \forall x\in X. \]
On the other hand, there is no restriction on the dimension of the base $X$, which can be arbitrary. A morphism of $G$-structure groupoids $\Phi:\G_1\to\G_2$ restricts to an \'etale map between the source-fibers.
\medskip

\begin{example}
\label{ex:morphism:s:fibers:G:grpd}
Similar to Example \ref{ex:principal:bundle:grpd}, a $G$-structure groupoid on the pair groupoid $P\times P\tto P$ is the same thing as an ordinary $G$-structure on $P\to P/G=M$. The tautological 1-forms $\theta\in\Omega^1(P;\Rr^n)$ and $\Theta\in\Omega^1_R(\G;\Rr^n)$ are related by
\[ \Theta_{(p_1,p_2)}(v_1,0)=\theta_{p_1}(v). \]

Moreover, as in Example \ref{ex:morphism:s:fibers:principal:grpd}, for any $G$-structure groupoid $\G\tto X$, each source fiber $\s^{-1}(x)\to \s^{-1}(x)/G$ has a $G$-structure, and we have a morphism of $G$-structure groupoids covering the target map
\[ 
\s^{-1}(x)\times \s^{-1}(x) \to \G, \qquad (\gamma_1,\gamma_2)\mapsto \gamma_1\cdot \gamma_2^{-1}.
\]
%
%\vcenter{\vbox{ 
%\xymatrix@R=20pt{
%\s^{-1}(x)\times \s^{-1}(x) \ar[r]\ar@<0.2pc>[d] \ar@<-0.2pc>[d]  & \G\ar@<0.2pc>[d] \ar@<-0.2pc>[d]  \\
%\s^{-1}(x)\ar[r]_{t} & X
%}}}
%\qquad (\gamma_1,\gamma_2)\mapsto \gamma_1\cdot \gamma_2^{-1}.
%\]
\end{example}

\begin{example}[The trivial $G$-structure Lie group]
Given $G\subset \GL(n.\Rr)$ we can form the semi-direct product Lie group
\[ \G:=\Rr^n\rtimes G, \]
which we view as a Lie groupoid over a singleton $\{*\}$. The $G$-action 
\[ \G\times G\to \G, \quad (v,h,g)=(v, hg), \]
is a $G$-principal action (in the sense of Definition \ref{def:principal:G:action}) and carries the tautological  form $\Theta\in\Omega^1_R(\G;\Rr^n)$ defined by
\[ \Theta_{(u,g)}(v,\al)=g^{-1}v. \]
The resulting $G$-structure on $\s^{-1}(*)=\Rr^n\rtimes G\to \s^{-1}(*)/G=\Rr^n$ is the trivial $G$-structure over $\Rr^n$. So we call this the trivial $G$-structure Lie group.
\end{example}

Later we will see much less trivial examples.

%%%%%%%%%%%%%%%%%%%%%%%%%%%%%%%
\subsection{$G$-structure groupoids with connection}
%%%%%%%%%%%%%%%%%%%%%%%%%%%%%%%
Let $\G\tto X$ be a $G$-structure groupoid with tautological form $\Theta$. If, additionally, $\Gamma$ is equipped with a connection $\Omega$, we define the {\bf torsion of the connection} to be the right-invariant 2-form given by
\[ \Tors(\Omega)\in\Omega^2_R(\G;\Rr^n),\quad \Tors(\Omega)(v,w):=\d\Theta(h(v),h(w)). \]
The restriction of $\Tors(\Omega)$ to the source fiber $\s^{-1}(x)$ is the (usual) torsion 2-form of the induced connection on $\s^{-1}(x)\to \s^{-1}(x)/G$ and we find

\begin{prop}
If $\H$ is a connection on a $G$-structure groupoid $\G\tto X$, its tautological 1-form $\Theta$, connection 1-form $\Omega$ and torsion $\Tors$, satisfy:
\begin{enumerate}[(i)]
\item {\bf 2nd structure equation}:
\[ \d\Theta=-\Omega\wedge\Theta +\Tors(\Omega); \]
\item {\bf 2nd Bianchi's identity}:
\[ \d\Tors(\Omega)|_{\H}=(\Curv(\Omega)\wedge \Theta)|_{\H}. \]
\end{enumerate}
\end{prop}

The following result is also an immediate consequence of Proposition \ref{prop:maps:connect:grpd}.

\begin{prop}
Given two $G$-structure groupoids with connection $\G_1$ and $\G_2$, a morphism of groupoids $\Phi:\G_1\to\G_2$ is a $G$-structure groupoid morphism if and only if it is $G$-equivariant and preserves the tautological and connection 1-forms
\[ \Phi^*\Theta_2=\Theta_1,\quad \Phi^*\Omega_2=\Omega_1. \]
\end{prop}

%%%%%%%%%%%%%%%%%%%%%%%%%%%%%%%%%%%%%%%%%%%%%%%%%%%%%%%%%%%%%%%%%%
%%%%%%%%%%%%%%%%%%%%%%%%%%%%%%%%%%%%%%%%%%%%%%%%%%%%%%%%%%%%%%%%%%
\section{$G$-structure algebroids and connections}
\label{sec:G-structure:algbd}
%%%%%%%%%%%%%%%%%%%%%%%%%%%%%%%%%%%%%%%%%%%%%%%%%%%%%%%%%%%%%%%%%%
%%%%%%%%%%%%%%%%%%%%%%%%%%%%%%%%%%%%%%%%%%%%%%%%%%%%%%%%%%%%%%%%%%

We now turn to the infinitesimal versions of  $G$-principal groupoids and $G$-structure groupoids, namely, $G$-principal algebroids and $G$-structure algebroids. We will see that that specifying a Cartan's realization problem is equivalent to specifying a $G$-structure algebroid with connection and this will be the starting point for our method to solve the problem.
\medskip

%%%%%%%%%%%%%%%%%%%%%%%%%%%%%%%
\subsection{$G$-principal algebroids}
%%%%%%%%%%%%%%%%%%%%%%%%%%%%%%%

We will denote by $A\to X$ a Lie algebroid with anchor $\rho:A\to TX$ and Lie bracket $[\cdot,\cdot]$. Recall that if a Lie group $G$ acts on $A$ by Lie algebroid automorphisms, it has an associated infinitesimal action, namely the Lie algebra morphism $\widehat{\psi}:\gg\to \Der(A)$ into the Lie algebra of derivations of $A$, given by
\[ \widehat{\psi}(\al)(s):=\left.\frac{\d }{\d t}\right|_{t=0} (\exp(t\al))^* s, \quad \al\in\gg,\, s\in\Gamma(A), \]
where the right hand side of the expression above denotes the action of $G$ on $\Gamma(A)$ induced by the action of $G$ on $A$.
The $G$-action on $A$ covers a $G$-action on $X$ which also has an associated infinitesimal action $\psi:\gg\to\X(X)$
\[ \psi(\al)(f)(x):=\left.\frac{\d }{\d t}\right|_{t=0} f(x \exp(t\al)), \quad \ f\in C^\infty(X).\] 
The two infinitesimal actions are related: the derivation $\widehat{\psi}(\al)\in\Der(A)$ has symbol the vector field $\psi(\al)\in\X(X)$, so that
\[ \widehat{\psi}(\al)(f s)=f\,\widehat{\psi}(\al)(s)+\psi(\al)(f) s,\quad \al\in\gg,\ s\in\Gamma(A),\ f\in C^\infty(X).\] 

\begin{definition}
A $G$-action on a Lie algeboid $A$ by Lie algebroid automorphisms with infinitesimal action $\widehat{\psi}:\gg\to\Der(A)$ is called a {\bf $G$-principal action} if there is an injective algebroid morphism $i:X\rtimes \gg\to A$ such that
\[ \widehat{\psi}(\al)=[i(\al),-]. \]
We call $A$ a {\bf $G$-principal algebroid} and $i:X\rtimes \gg \to A$ the {\bf action morphism}.
\smallskip

A {\bf morphism of $G$-principal Lie algebroids} is a Lie algebroid morphism
 \[
 \xymatrix@R=20pt{
 A_1\ar[d]\ar[r]^{\Phi} & A_2\ar[d] \\
 X_1\ar[r]_{\phi} & X_2
 }
 \]
 which is $G$-equivariant and which intertwines the action morphisms,
 \[ \Phi\circ i_1=i_2\circ(\phi\times I). \]

\end{definition}

The action morphism $i:X\rtimes \gg \to A$ is part of the data defining a $G$-principal Lie algebroid: there could be more that one such morphism determining the same  infinitesimal action $\widehat{\psi}:\gg\to\Der(A)$. When $G$ is connected, this morphism determines the $G$-action. 
% One should keep in mind that not every infinitesimal action $\widehat{\psi}:\gg\to\Der(A)$ integrates to a $G$-action.

\begin{prop}
\label{prop:principal:G:action:integrable}
Let $\G\tto X$ be a $G$-principal Lie groupoid with action morphism $\iota:X\rtimes G\to \G$. Then its Lie algebroid $A\to X$ is a $G$-principal algebroid with action morphism $i=\iota_*:X\rtimes \gg \to A$. 

Conversely, given a $G$-principal algebroid $A\to X$, a Lie groupoid $\G\tto X$ integrating $A$ admits a (unique) $G$-principal action inducing the $G$-action on $A$ provided the action morphism $i:X\rtimes \gg\to A$ integrates to an effective Lie groupoid morphism
\[ \iota: X\rtimes G\to \G. \]
\end{prop}

\begin{proof}
Assume first that we have a $G$-principal action $\G\times G\to \G$. Differentiating the associated inner $G$-action on $\G$, given by \eqref{eq:inner:G:action}, we obtain a $G$-action on the Lie algebroid $A$ of $\G$ by Lie algebroid automorphisms:
\[ A\times G\to A,\quad \xi \odot g:=\left.\frac{\d}{\d t}\right|_{t=0} \gamma(t)\odot g,\quad (\xi\in A_x, g\in G),\]
where $\gamma(t)$ is a curve in the source fiber $\s^{-1}(x)$ with $\gamma(0)=1_x$ and $\dot\gamma(0)=\xi$. 

The $G$-action on $A$ is proper and effective (but it may fail to be free). It covers the $G$-action on $X$ and it has an associated infinitesimal action $\widehat{\psi}:\gg\to \Der(A)$. The action morphism $\iota:X\rtimes G\to \G$ differentiates to a Lie algebroid morphism $i:X\rtimes\gg\to A$ and we find that
\[ \widehat{\psi}(\al)(s)=[i(\al),s].\]
This proves the first half of the theorem.

Conversely, given a $G$-principal action on $A\to X$ and a Lie groupoid $\G\tto X$ integrating $A$, if the  morphism $i:X\rtimes \gg\to A$ integrates to a Lie groupoid morphism $\iota: X\rtimes G\to \G$, the latter must be locally injective, since the former is injective by definition. Then, as in Proposition \ref{prop:principal:G:action}, formula \eqref{eq:principal:action:formula} defines a $G$-action on $\G$ which is principal provided it is effective.
\end{proof}

\begin{remark}[Effectiveness]
\label{rem:effective action}
Concerning the assumption in the statement of the previous proposition, notice that if the $G$-action on $A$ is effective, the action morphism $i:X\rtimes \gg\to A$ integrates to an effective Lie groupoid morphism $\iota: X\rtimes G\to \G$. Indeed, assume that that this was not the case, so there exists $e\not=g_0\in G$ such that $i(x,g_0)=1_x$ for all $x\in X$. It follows that for the associated inner action of $G$ on $\G$, given by \eqref{eq:inner:G:action}, the element $g_0$ also acts trivially. Hence, the element $g_0$ acts trivially on $A$, contradicting the assumption that $G$ acts effectively on $A$. This covers many applications.
\end{remark}

When it comes to morphisms, observe that a morphisms of $G$-principal groupoids $\Phi:\G_1\to \G_2$ induces a morphism $\Phi_*:A_1\to A_2$ of the associated $G$-principal algebroids. The converse, in general, fails unless $\G_1$ is the canonical $G$-integration of $A_1$ to be discussed later in Section \ref{sec:integrability}.

The proposition above shows that a $G$-principal algebroid parameterizes infinitesimally families of $G$-principal bundles. This is illustrated by the following examples.
% shows that this can be seen already in the case of a single principal bundle.

\begin{example}
\label{ex:principal:bundle:algbrd}
A $G$-principal action on the  tangent Lie algebroid $A=TP\to P$ covering an effective $G$-action on $P$, is the same thing as a $G$-principal action on $P$. Indeed, the infinitesimal action $\psi:\gg\mapsto \X(P)$ can be viewed as a Lie algebroid morphism
\[ i:P\rtimes \gg\to TP,\quad (x,\al)\mapsto (\al_P)|_p. \]
This map is just the action morphism of the $G$-action on $A=TP$, since one checks easily that its associated infinitesimal action is given by
\[ \widehat{\psi}:\gg\to \Der(TP), \quad \al\mapsto [i(\al),-]. \]
The condition that $i:P\rtimes \gg\to TP$ is injective, together with the assumption that the $G$-action on $P$ is effective, guarantees that we have a $G$-principal action. The $G$-action on $A=TP$ is the tangent lift of the $G$-action on $P$,
\[ \xi\odot g=\left.\frac{\d}{\d t}\right|_{t=0} p(t)\,g, \]
where $p(t)$ is any curve in $P$ with $\dot{p}(0)=\xi$.
\end{example}

\begin{example}
\label{ex:morphism:s:fibers:principal:algbrd}
Given a $G$-principal groupoid $\G\tto X$, it follows from Example \ref{ex:principal:bundle:algbrd} that $T(\s^{-1}(x))\to \s^{-1}(x)$ is a $G$-principal algebroid. If $A\to X$ is the Lie algebroid of $\G$, then we have a morphism of $G$-principal Lie algebroids covering the target
\[
\vcenter{\vbox{ 
 \xymatrix@R=20pt{
 T(\s^{-1}(x))\ar[d]\ar[r]& A\ar[d] \\
 \s^{-1}(x)\ar[r]_{t} & X
 }}}
 \qquad v_\gamma\mapsto \d R_{\gamma^{-1}}(v_\gamma).
 \]
This morphism is just the infinitesimal version of the morphism of $G$-principal groupoids given in Example \ref{ex:morphism:s:fibers:principal:grpd}.
\end{example}

\begin{example}
\label{ex:action:algbrd}
The $G$-principal algebroid corresponding to the $G$-principal action of Example \ref{ex:action:grpd} is the action algebroid $A=X\rtimes \hh\to X$ defined by the infinitesimal action $\psi:\hh\to\X(X)$ associated with the $H$-action on $X$. The corresponding action morphism is just the inclusion
\[ i:X\rtimes \gg\hookrightarrow X\rtimes \hh. \]
The $G$-action on $A$ is 
\[ (x,\al)\odot g=(g^{-1}x,\Ad_{g^{-1}} \al), \quad (x,\al)\in X\rtimes\hh, \]
where $\Ad$ denotes the restriction of the adjoint action of $H$ to $G$.
\end{example}

\begin{example}
\label{ex:bundle:algebras}
The $\SO(2,\Rr)$-principal algebroid corresponding to the $\SO(2,\Rr)$-principal action of Example \ref{ex:bundle:groups} is the bundle of Lie algebras $A=\Rr\times \Rr^3\to\Rr$ with Lie bracket given by
\[ [e_1,e_2]=x e_3, \quad [e_2,e_3]=e_1, \quad [e_3,e_1]=e_2, \]
where we have denoted by $\{e_1,e_2,e_3\}$ the basis of constant sections of $A$ associated with the canonical basis of $\Rr^3$. The action morphism of this Lie algebroid is
\[ i:\Rr\rtimes \mathfrak{so}(2,\Rr)\to A,\quad (x,\lambda)\mapsto \lambda e_3. \]
The $\SO(2,\Rr)$-action on $A$ is given by
\[ (x,v)\odot e^{-i\theta}=(x,R_{\theta}v), \quad (x,v)\in \Rr\times \Rr^3\, \]
where $R_\theta:\Rr^3\to\Rr^3$ denotes a rotation by an angle $\theta$ around the 3rd axis. 
\end{example}

%Consider a $G$-principal action on a Lie algebroid $A\to X$ with associated morphism $i:\gg\ltimes X\to A$. Assume that $\G\tto X$ is a Lie groupoid integrating $A\to X$ and that the morphism $i:X\rtimes \gg\to A$ integrates to a Lie groupoid morphism:
%\[ \iota: X\rtimes G\to \G. \]
%Then we define a right $G$-action on $\G$ preserving the s-fibers by:
%\[ \G\times G\to \G,\quad \gamma\, g:= \iota(\t(\gamma),g)\cdot \gamma. \]
%Now observe that:
%\begin{enumerate}[(a)]
%\item The action $\G\times G\to \G$ is proper: This follows because it covers a proper action. In fact, the target map $\t:\G\to X$ is $G$-equivariant  and the $G$-action on $X$ is proper;
%\item The action  $\G\times G\to \G$ is locally free: This follows because the morphism $i:X\rtimes \gg\to A$ is injective, and the infinitesimal action determined by the right $G$-action on $\G$ is given by:
%\[ \gg\to \X(\G), \quad \al\mapsto i_*(\alpha), \]
%where $i_*(\alpha)$ denotes the right-invariant vector field corresponding to the section of $A\to X$ given by $x\mapsto i(x,\al)$.
%\end{enumerate}
%Therefore, the right $G$-action on each fiber $\s^{-1}(x)$ is proper and locally free, so $P=\s^{-1}(x)\to M=\s^{-1}(x)/G$ is a $G$-principal bundle over an effective orbifold. Obviously, if the groupoid morphism $ \iota: X\rtimes G\to \G$ is injective, then the action is free and $M=\s^{-1}(x)/G$ is a smooth manifold.
%\end{example}

%%%%%%%%%%%%%%%%%%%%%%%%%%%%%%%
\subsection{Connections on $G$-principal algebroids}
%%%%%%%%%%%%%%%%%%%%%%%%%%%%%%%

\begin{definition}
A {\bf connection} on a $G$-principal algebroid $A\to X$ with action morphism $i:X\rtimes\gg\to A$, is a subbundle $H\subset A$ satisfying:
\begin{enumerate}[(i)]
\item $H$ is horizontal relative to $i:X\rtimes\gg\to A$:
\[ A=H\oplus \Im i; \]
\item $H$ is $G$-invariant:
\[ H_{x\,g}=g_*(H_x). \]
\end{enumerate}

A {\bf morphism} of $G$-principal algebroids with connection is a $G$-principal algebroid morphism $\Phi:A_1\to\A_2$ such that $\Phi_*(H_1)\subset H_2$.
\end{definition}

A connection $\H$ on a $G$-principal groupoid $\G\tto X$, being a right-invariant distribution, defines a subbundle $H=\H|_X$ of the Lie algebroid $A\to X$ of $\G$. This subbundle is complementary to the image of the morphism $i:X\rtimes\gg\to A$ and it is immediate to check that it satisfies the properties in the definition. This gives a 1:1 correspondence
\begin{equation}
\label{eq:correspondence:connections}
\left\{\txt{connections $\H$ on \\ $\G\tto X$ \\ \,}\right\}
\tilde{\longleftrightarrow}
\left\{\txt{connections $H$ on \\ $A\to X$\\ \,} \right\}.
\end{equation}

At the algebroid level one can also characterize connections using differential $A$-forms, since we have the following easy proposition.

\begin{prop}
Given a connection $H$ on a $G$-principal algebroid $A\to X$ there is a unique $\gg$-valued $A$-form $\omega\in\Omega^1(A;\gg)$ with $\Ker\omega=H$ and satisfying:
\begin{enumerate}[(i)]
\item It is vertical relative to the morphism $i:X\rtimes\gg\to A$:
\[ \omega(i(x,\al))=\al,\quad \forall \al\in\gg,\ x\in M; \]
\item It is $G$-equivariant for the $G$-action on $A$ by automorphisms:
\[ \omega_{xg}(\xi \odot g)=\Ad_{g^{-1}}\cdot \omega_x(\xi),\quad \forall \xi\in A_x,\ g\in G.\]
\end{enumerate}
Conversely, given a $\gg$-valued $A$-form $\omega\in\Omega^1(A;\gg)$ satisfying (i) and (ii), its kernel $H=\Ker\omega$ defines a connection.
\end{prop}

The $A$-form $\omega\in\Omega^1(A;\gg)$  given in the previous proposition ought to be called the {\bf connection 1-form} of the connection $H$. Under the 1:1 correspondence \eqref{eq:correspondence:connections}, the connection 1-form on the groupoid $\Omega\in\Omega^1_R(\G;\gg)$ and the connection 1-form on the algebroid $\omega\in\Omega^1(A;\gg)$ are related via restriction along the identity section
\[ \Omega^1_R(\G;\gg)\ni \Omega \quad \longleftrightarrow\quad \omega\in\Omega^1(A;\gg), \quad \text{with}\quad \omega_x:=\Omega_{1_x}. \]

On the other hand, a morphism of Lie algebroids $\Phi:A_1\to A_2$ induces a pullback of $V$-valued $A$-forms
\[ \Phi:\Omega^k(A_2,V)\to \Omega^k(A_1,V). \]
This leads to the following characterization of morphisms of $G$-principal algebroids with connection. The proof is immediate.

\begin{prop} 
\label{prop:maps:connect:algbrd}
Let $A_1$ and $A_2$ be $G$-principal algebroids with connection. A Lie algebroid morphism $\Phi:A_1\to A_2$ is a morphism of $G$-principal algebroids with connection if and only if $\Phi$ is $G$-equivariant, intertwines the action morphisms and preserves the connection 1-forms, $\Phi^*\omega_2=\omega_1$.
\end{prop}

Henceforth, we will specify connections on a $G$-principal algebroid by specifying the connection 1-form  $\omega\in\Omega^1(A;\gg)$.

\begin{remark}[Existence of connections]
Under our standing assumption that $G$ is compact, any $G$-principal algebroid admits a connection. In fact, if we consider the short exact sequence
\[0 \longrightarrow X \rtimes \gg \stackrel{i}{\longrightarrow} A \longrightarrow A/\mathrm{Im}i \longrightarrow 0,\]
then a connection is just a $G$-equivariant splitting of the map $i$. To construct such a splitting we may take an arbitrary splitting and use an averaging argument to obtain a connection. We note also that since there is a 1:1 correspondence between connections on $G$-principal groupoids and their $G$-principal algebroids it follows that every $G$-principal groupoid admits a connection. This argument remains valid, more generally, if we replace the compactness assumption on $G$ by the assumption that the $G$-action on $A$ is proper.
\end{remark}

\begin{example}
\label{ex:principal:bundle:connection:algbrd}
Under the 1:1 correspondence between $G$-principal bundles $P\to M$ and $G$-principal algebroids $TP\to P$ (see Example \ref{ex:principal:bundle:algbrd}), a principal connection $H$ on $P\to M$ is the same thing as a connection on the algebroid $TP\to P$. The two notions of connection 1-form, in the principal bundle and principal algebroid senses, coincide. Moreover, the connection $H$ is the infinitesimal version of the connection $\H$ on the pair groupoid $P\times P\tto P$ (see Example \ref{ex:principal:bundle:connection:grpd}).
\end{example}

\begin{example}
\label{ex:morphism:s:fibers:principal:algbrd:connection}
Given a $G$-principal groupoid $\G\tto X$ with connection $\H$, each source fiber $\s^{-1}(x)\to \s^{-1}(x)/G$ is a $G$-principal bundle with connection $\H_x$, and so $T(\s^{-1}(x))\to \s^{-1}(x)$ is a $G$-principal algebroid with connection $\H_x$ (see Example \ref{ex:principal:bundle:connection:algbrd}). The Lie algebroid $A\to X$ of $\G$ is a $G$-principal Lie algebroid with connection $H$ and we have a morphism of $G$-principal Lie algebroids with connection
\[  
T(\s^{-1}(x)) \to A,  \qquad v_\gamma\mapsto \d R_{\gamma^{-1}}(v_\gamma).
 \]
%\vcenter{\vbox{ 
% \xymatrix@R=20pt{
% T(\s^{-1}(x))\ar[d]\ar[r]& A\ar[d] \\
% \s^{-1}(x)\ar[r]_{t} & X
% }}}
% \qquad v_\gamma\mapsto \d R_{\gamma^{-1}}(v_\gamma).
% \]
This is, of course, the infinitesimal version of Example \ref{ex:morphism:s:fibers:principal:grpd:connection}.
% The statement in the previous proposition applied to this morphism, is just a consequence of the fact that the connections 1-form on $A$ and on $\G$ are related by right-translation. 
\end{example}

Naturally, we define the {\bf curvature 2-form} of a connection $\omega$ to be the $A$-form $\Curv(\omega)\in\Omega^2(A;\gg)$ given by
\[ \Curv(\omega)(\xi,\zeta):=\d_A\omega(h(\xi),h(\zeta)), \quad \xi,\zeta\in A, \]
where $h:A\to H$ denotes the projection and $\d_A:\Omega^k(A,V)\to \Omega^{k+1}(A,V)$ the Lie algebroid differential. We leave the proof of the following easy proposition to the reader.

\begin{prop}
The connection $\omega\in\Omega^1(A;\gg)$ on a $G$-principal algebroid $A$ satisfies:
\begin{enumerate}[(i)]
\item {\bf 1st structure equation}:
\[ \d_A\omega=-\omega\wedge\omega +\Curv(\omega); \]
\item {\bf 1st Bianchi's identity}:
\[ \d_A\Curv(\omega)|_H=0. \]
\end{enumerate}
Moreover, $\Curv(\omega)=0$ if and only if $H\subset A$ is a Lie subalgebroid.
\end{prop}

\subsection{$G$-structure algebroids}
%%%%%%%%%%%%%%%%%%%%%%%%%%%%%%%

We are now ready to discuss the infinitesimal version of a $G$-structure groupoid.

% As we will see later, these are the objects that appear naturally in Cartan's realization problem.

%For an action $\Psi:G\times M\to M$ we will denote by $\psi:\gg\to\X(M)$ the corresponding infinitesimal action and by $\gg\ltimes M\to M$ the associated action Lie algebroid. When

\begin{definition}
If $G\subset \GL(n,\Rr)$, a {\bf $G$-structure algebroid} consists of a $G$-principal algebroid $A\to X$ with action morphism $i:X\rtimes \gg\to A$, equipped with a fiberwise surjective, $\Rr^n$-valued, $A$-form $\theta\in\Omega^1(A;\Rr^n)$ satisfying:
\begin{enumerate}[(i)]
\item strong horizontality:
\[ \theta_x(\xi)=0 \quad\text{iff}\quad \xi=i(x,\al),\text{ for some }\al\in\gg. \]
\item $G$-equivariance:
\[ \theta_{x\cdot g}(\xi\odot g)=g^{-1}\cdot\theta_x(\xi),\quad \forall g\in G. \]
\end{enumerate}
We call $\theta$ the {\bf tautological form} of the $G$-structure algebroid.

A {\bf morphism} of $G$-structure algebroids $\Phi:A_1\to A_2$ is a morphism of $G$-principal algebroids which preserves the tautological forms: $\Phi^*\theta_2=\theta_1$.
\end{definition}

The definition of a $G$-structure algebroid $A\to X$ implies that:
\[ \rank A=n+\dim G. \]
There is no restriction on the dimension of the base $X$, which can be arbitrary. A morphism of $G$-structure algebroids $\Phi:A_1\to A_2$ is necessarily a fiberwise isomorphism.

Since restriction along the unit section establishes a 1:1 correspondence between right-invariant forms on the Lie groupoid $\G\tto X$ and $A$-forms on its Lie algebroid $A\to X$, the following result is an immediate consequence of Proposition \ref{prop:principal:G:action:integrable} and Remark \ref{rem:condition i}

\begin{prop}
\label{prop:G:structures:integrable}
If $\G\tto X$ is a $G$-structure groupoid with tautological form $\Theta$, then its Lie algebroid $A\to X$ is a $G$-structure algebroid with tautological form $\theta=\Theta|_X$. 

Conversely, given a $G$-structure algebroid $A\to M$ with tautological form $\theta$ and a Lie groupoid $\G\tto X$ integrating $A$, if the action morphism $i:X\rtimes \gg\to A$ integrates to an effective Lie groupoid morphism 
\[ \iota: X\rtimes G\to \G,\] 
then there is a unique $G$-structure groupoid structure on $\G$ with tautological 1-form the unique right-invariant, $\Rr^n$-valued, 1-form $\Theta$ such that $\Theta|_X=\theta$.
\end{prop}

The previous proposition shows that $G$-structure algebroids parameterize, infinitesimally, families of $G$-structures. %We illustrate this in the case of a single $G$-structure.

\begin{example}
\label{ex:morphism:s:fibers:G:algbrd}
A $G$-structure algebroid on $A=TP\to P$ is the same thing as a $G$-structure on $P\to P/G=M$. In fact, we saw in Example \ref{ex:principal:bundle:algbrd} that a $G$-principal action on $TP$ is the same thing as a $G$-principal bundle. Moreover, the form $\theta$ in the definition of a $G$-structure algebroid is an ordinary form $\theta\in\Omega^1(P;\Rr^n)$ which is strongly horizontal and equivariant in the usual sense. 

Now, given a $G$-structure groupoid $\G\tto X$, each source fiber $\s^{-1}(x)\to \s^{-1}(x)/G$ has a $G$-structure and so $T(\s^{-1}(x))\to \s^{-1}(x)$ is a $G$-structure algebroid. The Lie algebroid $A$ of $\G$ is also a $G$-structure algebroid and we have a morphism of $G$-structure algebroids:
\[
\vcenter{\vbox{ 
 \xymatrix@R=20pt{
 T(\s^{-1}(x))\ar[d]\ar[r]& A\ar[d] \\
 \s^{-1}(x)\ar[r]_{t} & X
 }}}
 \qquad v_\gamma\mapsto \d R_{\gamma^{-1}}(v_\gamma).
 \]
This is, of course, the infinitesimal version of Example \ref{ex:morphism:s:fibers:G:grpd}.
%The statement in the previous proposition applied to this morphism, is just a consequence of the fact that the connections 1-form on $A$ and on $\G$ are related by right-translation. 
\end{example}

%%%%%%%%%%%%%%%%%%%%%%%%%%%%%%%
\subsection{$G$-structure algebroids with connection}
%%%%%%%%%%%%%%%%%%%%%%%%%%%%%%%
Given a $G$-structure algebroid $A\to X$ with connection form $\omega\in\Omega^1(A;\gg)$ we define the {\bf torsion of the connection} to be the $\Rr^n$-valued $A$-form 
\[ \Tors(\omega)\in\Omega^2(A;\Rr^n),\quad \Tors(\omega)(\xi,\zeta):=\d_A\theta(h(\xi),h(\zeta)), \]
where $\theta\in\Omega^1(A;\Rr^n)$ denotes the tautological 1-form and $h:A\to H$ the bundle projection. 

Under the correspondence between $G$-structure groupoids and $G$-structure algebroids given by Proposition \ref{prop:G:structures:integrable}, if a connection $\Omega$ on $\G$ corresponds to a connection $\omega$ on $A$, then the associated torsions $\Tors(\Omega)$ and $\Tors(\omega)$ correspond to each other:
\[ \Tors(\omega)_x=\Tors(\Omega)_{1_x}, \quad \forall x\in X. \]
Moreover, we have the infinitesimal versions of the 2nd structure equation and Bianchi's identity. We leave the (easy) proof to the reader.

\begin{prop}
Let $A\to X$ be a $G$-structure algebroid with connection $H$.  Then the tautological 1-form $\theta$, the connection 1-form $\omega$ and its torsion satisfy:
\begin{enumerate}[(i)]
\item {\bf 2nd structure equation}:
\[ \d_A\theta=-\omega\wedge\theta +\Tors(\omega); \]
\item {\bf 2nd Bianchi's identity}: 
\[ \d_A\Tors(\omega)|_H=(\Curv(\omega)\wedge \theta)|_H. \]
\end{enumerate}
\end{prop}

%If we let $\{e_1,\dots,e_n\}$ be the standard basis of $\Rr^n$ and we let $\{E_i^j:1\le i,j,\le n\}$ be the corresponding basis of $\mathfrak{gl}(n,\Rr)$, then we can write the tautological and connection 1-forms in components:
%\[ \theta=\sum_{i=1}^n \theta^i e_i,\quad \omega=\sum_{i,j=1}^n \omega^i_j E^j_i.,\]
%and also the torsion and curvature 2-forms:
%\[ \Tors(\omega)=\sum_{i=1}^n \Tors(\omega)^i e_i,\quad \Curv(\omega)=\sum_{i,j=1}^n \Curv(\omega)^i_j E^j_i.,\]
%The structure equations become:
%\begin{align*}
%\d_A \theta^i &=-\sum_{j=1}^n \omega^i_j \wedge \theta^j+ \Tors(\omega)^i,\quad (i=1,\dots,n)\\
%\d_A \omega^i_j &=-\sum_{k=1}^n \omega^i_j \wedge \omega^k_j+ \Curv(\omega)^i_j,\quad (i,j=1,\dots,n).
%\end{align*}
%which we abbreviate to:
%\begin{align*}
%\d_A \theta &=-\omega \wedge \theta + \Tors(\omega), \\
%\d_A \omega &=-\omega \wedge \omega + \Curv(\omega).
%\end{align*}

Also, the following characterization of morphisms of $G$-structure algebroids with connection is an immediate consequence of Proposition \ref{prop:maps:connect:algbrd}.

\begin{prop}
Let $A_1$ and $A_2$ be $G$-structure algebroids with connection. A morphism of algebroids $\Phi:A_1\to A_2$ is a morphism of $G$-structure algebroids with connection if and only if it is $G$-equivariant, intertwines the action morphisms, and preserves the tautological and connection 1-forms:
\[ \Phi^*\theta_2=\theta_1,\quad \Phi^*\omega_2=\omega_1. \]
\end{prop}

\begin{example}
\label{ex:G:algbrd:connect}
As we saw in Example \ref{ex:principal:bundle:connection:algbrd}, a $G$-structure algebroid with connection on $TP\to P$ is the same thing as a $G$-structure with connection on $P\to P/G=M$. Hence, if one is given $G$-structures with connection $(P_i,\theta_i,\omega_i)$, and a local diffeomorphism $\varphi:M_1\to M_2$ the lift 
\[ \tilde{\varphi}:P_1\to P_2 \] 
is a map of $G$-structures with connections if and only if the bundle map
\[ \d \tilde{\varphi}:TP_1\to TP_2 \]
is a morphism of $G$-structure algebroids with connection.

%Similarly, combining Examples \ref{ex:morphism:s:fibers:principal:algbrd:connection} and \ref{ex:morphism:s:fibers:G:algbrd}, if we are given a $G$-structure algebroid $A\to X$ with connection $H$ and we assume that there is a $G$-structure groupoid $\G\tto X$ with connection $\H$ integrating it, then each source fiber $\s^{-1}(x)\to \s^{-1}(x)/G$ is a $G$-structure with connection $\H_x$, and we have a morphism of $G$-structure algebroids with connection:
%\[
%\vcenter{\vbox{ 
% \xymatrix@R=20pt{
% T(\s^{-1}(x))\ar[d]\ar[r]& A\ar[d] \\
% \s^{-1}(x)\ar[r]_{t} & X
% }}}
% \qquad v_\gamma\mapsto \d R_{\gamma^{-1}}(v_\gamma).
% \]
%\end{example}
%
%\begin{example}
\end{example}

%%%%%%%%%%%%%%%%%%%%%%%%%%%%%%%
\subsection{Canonical form of a $G$-structure algebroid with connection}
%%%%%%%%%%%%%%%%%%%%%%%%%%%%%%%

We are now in condition to justify why $G$-structure algebroids with connection give the appropriate language to deal with Cartan's realization problem.

\begin{definition}
A $G$-structure algebroid $A\to X$ with connection is said to be in {\bf canonical form} if 
\begin{itemize}
\item $A=X\times(\Rr^n\oplus\gg)$ is the trivial vector bundle;
\item the $G$-action on $A$ takes the form: $(x, u, \al)\, g = (x\, g, g^{-1}\, u, \Ad_{g^{-1}}\cdot \al)$;
\item the action morphism is given by $i:X\rtimes \gg\to A$, $(x,\al))\mapsto (x,0,\al)$;
\item the tautological form is given by the projection $\theta:X\times(\Rr^n\oplus\gg)\to \Rr^n$;
\item the connection is given by $H=X\times(\Rr^n\oplus\{0\})\subset A$ or, equivalently, the connection 1-form is the 
projection $\omega:X\times(\Rr^n\oplus\gg)\to \gg$.
\end{itemize}
\end{definition}

For a $G$-structure algebroid with connection in canonical form we can re-express the torsion and curvature 1-forms as $G$-equivariant maps:
\begin{itemize}
\item {\bf torsion}: $c: X \to \hom(\wedge^2\Rr^n, \Rr^n)$:
\[ c(x)(v,w)=\Tors(\omega)_x(v,w); \]
\item {\bf curvature}: $R : X\to \hom(\wedge^2\Rr^n,\gg)$:
\[ R(x)(v,w)=\Curv(\omega)_x(v,w); \]
\end{itemize}
where we don't distinguish between the vector $v\in\Rr^n$ and $(x,0,v)\in A_x$.  Using these expressions, we can deduce that the Lie bracket and the anchor coincide with the ones deduced in Section \ref{sec:Cartan:Realization} in connection with Cartan's realization problem.

\begin{prop}
\label{prop:canonical:form}
For a $G$-structure algebroid with connection which is in canonical form $A=X\times(\Rr^n\oplus\gg)\to X$  the Lie bracket is given on constant sections by:
\begin{equation}
\label{eq:canonical:Lie:bracket}
[(u,\al), (v, \be)] = (\al\cdot v - \be\cdot u - c(u,v), [\al,\be]_\gg - R(u,v)),
\end{equation}
while the anchor $\rho:A\to TX$ takes the form:
\begin{equation}
\label{eq:canonical:anchor} 
\rho(u,\al)=F(u) + \psi(\al), \quad (u,\al)\in\Rr^n\oplus\gg,
\end{equation}
where $F:X\times\Rr^n\to TX$ is a $G$-equivariant bundle map and $\psi:\gg\to \X(X)$ denotes the infinitesimal $G$-action on $X$.
\end{prop}

\begin{proof}
If $s_1$ and $s_2$  are sections such that $\theta(s_i)$ and $\omega(s_i)$ are constant, then the definition of the differential shows that:
\[ \d_A\theta(s_1,s_2)=-\theta([s_1,s_2]),\quad \d_A\omega(s_1,s_2)=-\omega([s_1,s_2]). \]
So all we have to show is that:
\begin{align*}
-\d_A\theta((u,\al),(v,\be))&=\al\cdot v - \be\cdot u - c(u,v),\\
-\d_A\omega((u,\al),(v,\be))&= [\al,\be]_\gg - R(u,v)).
\end{align*}
This follows easily from the structure equations and the definition of $c$ and $R$.

The form of the anchor follows since the composition of the action morphism $i:X\rtimes \gg\to A$ with the anchor $\rho$ is the infinitesimal action $\psi$.

\end{proof}

For any $G$-structure algebroid with connection, the tautological and connection 1-forms give a coframe $(\theta,\omega)$ for the vector bundle $A\to X$.  This allows to put any $G$-structure Lie algebroid with connection in canonical form and in a natural way, i.e., independent of choices:

\begin{prop}
\label{prop:normal:form}
Any $G$-structure algebroid with connection $A\to X$ is naturally isomorphic to one in canonical form.
\end{prop}

\begin{proof}
The tautological and the connection 1-form give a coframe $(\theta,\omega)$ for $A$ defining an isomorphism
\[ 
A\stackrel{\cong}{\longrightarrow} X\times (\Rr^n\oplus\gg), \quad \xi_x\mapsto (x,\theta(\xi),\omega(\xi)). 
\]
One checks immediately that this puts $A$ into canonical form.
\end{proof}

For a Lie algebroid in canonical form, using expression \eqref{eq:canonical:Lie:bracket} for the Lie bracket, we deduce that the Jacobi identity is equivalent to the following set of equations:
\begin{enumerate}
\item the Jacobi identity for the bracket on $\gg$ 
\[ [[\al,\be]_\gg,\gamma]_\gg+[[\be,\gamma]_\gg,\al]_\gg+[[\gamma,\al]_\gg,\be]_\gg=0, \]
(by considering constant sections of the form $(0,\al)$, $(0,\be)$ and $(0,\gamma)$);
\item the  defining representation $\gg\subset \gl(n,\Rr)$:
\[ [\al,\be]_\gg(u)=\al(\be(u))-\be(\al(u)), \]
(by considering constant sections of the form $(0,\al)$, $(0,\be)$ and $(u,0)$);
\item the $\gg$-equivariance of $R$ and $c$: 
\begin{align*}
\psi(\al)(R(u,v))&=R(\al(u),v)+R(u,\al(v)),\\ 
\psi(\al)(c(u,v))&=c(\al(u),v)+c(u,\al(v)),
\end{align*}
(by considering constant sections of the form $(0,\al)$, $(u,0)$ and $(v,0)$);
\item The 1st and 2nd Bianchi identities:
\begin{align*}
\bigodot_{u,v,w} \left\{F(u)(c(v,w))+c(c(u,v),w)\right\}&=\bigodot_{u,v,w} R(u,v)w,\\
\bigodot_{u,v,w} \left\{F(u)(R(v,w))+ R(c(u,v),w)\right\}&=0,
\end{align*}
(by considering constant sections of the form $(u,0)$, $(v,0)$ and $(w,0)$). Here, the symbol $\bigodot$ means sum over the cyclic permutations.
\end{enumerate}
We conclude that:

\begin{corol}
\label{cor:Jacobi:identity}
Let $(G, X,c,R,F)$ be a Cartan Data satisfying additionally:
\begin{align*}
\bigodot_{u,v,w} \left\{F(u)(c(v,w))+c(c(u,v),w)\right\}&=\bigodot_{u,v,w} R(u,v)w,\\
\bigodot_{u,v,w} \left\{F(u)(R(v,w))+ R(c(u,v),w)\right\}&=0,
\end{align*}
then the Lie bracket \eqref{eq:canonical:Lie:bracket} and the anchor \eqref{eq:canonical:anchor} define a $G$-structure algebroid with connection (in canonical form).
\end{corol}

Let us give two simple examples of $G$-structure algebroids in canonical form.

\begin{example}[$G$-structure Lie algebras]
\label{ex:G:algebras}
Let $A\to \{*\}$ be a $G$-structure algebroid over a singleton. This means that $A=\Rr^n\oplus\gg$ is a Lie algebra with Lie bracket 
\[[(u,\al), (v, \be)] = (\al\cdot v - \be\cdot u - c(u,v), [\al,\be]_\gg - R(u,v)),\]
where $c\in \hom(\wedge^2\Rr^n, \Rr^n)$ and $R\in \hom(\wedge^2\Rr^n,\gg)$ satisfy:
\[ c(g\cdot u,g\cdot v)=g\cdot c(u,v), \quad R(g\cdot u,g\cdot v)=\Ad_g R(u,v),\quad (g\in G). \]
If we assume that $G$ is connected, these are equivalent to:
\[ c(\al\cdot u,v)+c(u,\al\cdot v)=\al\cdot c(u,v), \quad R(\al\cdot u,\cdot v)+R(u,\al\cdot v)=[\al,R(u,v)]_\gg, \quad (\al\in\gg).\]
The bracket satisfies the Jacobi identity iff additionally:
\[ \bigodot_{u,v,w} R(u,v)w-c(c(u,v),w)=0,\qquad \bigodot_{u,v,w} R(c(u,v),w)=0.\]
When $R=c=0$ we call $A=\Rr^n\oplus \gg$ the {\bf trivial $G$-structure Lie algebra}. The corresponding Cartan Problem describes flat spaces. 
\end{example}

\begin{example}[Constant scalar curvature $G$-structure algebroid]
\label{ex:constant:curvature}
Let $G=\SO(n,\Rr)$ act trivially on $X=\Rr$. Let $A=\Rr^n\times(\Rr^n\oplus\so(n,\Rr))\to \Rr$ be the $G$-structure algebroid with zero torsion ($c\equiv 0$), zero anchor ($F\equiv 0$) and curvature given by:
\[ R(x)(u,v) w=x \left(\langle w,v\rangle u-\langle w,u\rangle v\right). \]
One checks that 
\[ \bigodot_{u,v,w} R(u,v)w=0, \]
so the Jacobi identity holds. The corresponding Cartan Problem describes spaces of constant scalar curvature, i.e., space forms. 
\end{example}

%%%%%%%%%%%%%%%%%%%%%%%%%%%%%%%%%%%
\subsection{$G$-realizations of a $G$-structure algebroid with connection}
\label{sec:G-realizations}
%%%%%%%%%%%%%%%%%%%%%%%%%%%%%%%%%%%

The following notion will be relevant to discuss solutions of Cartan's Realization Problem:

\begin{definition}
A {\bf $G$-realization} of a $G$-structure algebroid $A\to X$ (with connection) is a $G$-structure $P\to M$ (with connection) together with a morphism of $G$-structure algebroids (with connection):
\[
\vcenter{\vbox{ 
 \xymatrix@R=20pt{
 TP\ar[d]\ar[r]^{\Phi}& A\ar[d] \\
 P\ar[r]_{\phi} & X
 }}}
 \]
 One calls $\phi:P\to X$ the {\bf classifying map} of the $G$-realization. 
 
A {\bf morphism of $G$-realizations} of a $G$-structure algebroid $A\to X$ (with connection) is a map $\varphi:P_1\to P_2$, commuting with the classifying maps $\phi_i:P_i\to M$ and yielding a commutative diagram of morphisms of $G$-structure algebroids (with connection):
\[ 
\xymatrix{
TP_1\ar[dr]_{\Phi_1}\ar[rr]^{{\d\varphi}} & &TP_2\ar[dl]^{\Phi_2} \\
 & A
}
\]
\end{definition}

It follows from this definition that a morphism $\varphi:P_1\to P_2$ of $G$-realizations is a local equivalence of $G$-structures with connection. There are more general notions of morphisms, but for us this notion will suffice.

In order to see how to construct realizations, let us recall that if $\G\tto X$ is a Lie groupoid with Lie algebroid $A$, one defines its {\bf Maurer-Cartan form} to be the right-invariant, $A$-valued, 1-form $\wmc\in\Omega^1(T^{\s} \G,\t^*A)$ given by
\[ \wmc(X):=\d R_{\gamma^{-1}} X, \quad \text{if }X\in T^{\s}_\gamma\G. \]
Equivalently, we can view $\wmc$ as a bundle map
\[
\xymatrix@C=30pt{
T^{\s} \G\ar[d]\ar[r]^{\wmc} & A\ar[d] \\
\G \ar[r]_{\t} & X
}
\]
The $1$-form $\wmc$ satisfies the Maurer-Cartan equation which can be equivalently stated as saying that this bundle map is a Lie algebroid morphism (see, e.g., \cite{FernandesStruchiner1}).

\begin{example}
Let $A\to X$  be a $G$-structure algebroid with connection $\omega$ and assume that $\G\tto X$ is a $G$-structure groupoid with connection $\Omega$ integrating it. Combining Examples \ref{ex:morphism:s:fibers:principal:algbrd:connection} and \ref{ex:morphism:s:fibers:G:algbrd}, each source fiber $\s^{-1}(x)\to \s^{-1}(x)/G$ is a $G$-structure with connection $\Omega|_{\s^{-1}(x)}$. The restriction of the Maurer-Cartan form gives a morphism of $G$-structure algebroids with connection:
\[
 \xymatrix@R=20pt{
 T(\s^{-1}(x))\ar[d]\ar[r]^---{\wmc} & A\ar[d] \\
 \s^{-1}(x)\ar[r]_{t} & X
 }
 \]
Hence, $(\s^{-1}(x),\wmc|_{\s^{-1}(x)},\t)$ is a $G$-realization of $A$.
\end{example}

There exist $G$-realizations which are not source fibers of groupoids integrating $A$. For example, one can consider open $G$-invariant subsets of the source fibers $\s^{-1}(x)$ of any $G$-structure groupoid $\G\tto X$ integrating $A\to X$. Later we will discuss the existence and classification problems for $G$-realizations, which is intimately related with the existence and classification problem of solutions of Cartan's Realization Problem.

Note that in the case where the $G$-realization consists of a source fiber $\s^{-1}(x)$ of a Lie groupoid integrating $A$, its classifying map (the target map) is a submersion onto the leaf containing $x$. For an arbitrary $G$-realization we have:

\begin{lemma}
\label{lem:classifying:map}
The classifying map $\phi: P \to X$ of a $G$-realization $\Phi:TP\to A$ is a submersion onto an open $G$-saturated subset of a leaf of $A$.
\end{lemma}

\begin{proof}
A $G$-realization $\Phi:TP\to A$, being a morphism of $G$-structure algebroids, is a fiberwise isomorphism.  Its base map $\phi:P\to X$ is $G$-equivariant, so its image is a $G$-saturated set. Also, composing $\Phi$ with the anchor, we conclude that for any $p\in P$ the differential of the classifying map:
\[ \d\phi|_p=(\rho\circ\Phi)|_p: T_p P\to T_{\phi(p)}L \]
is surjective, where $L$ is the leaf containing $\Phi(p)$. It follows that the image of $\phi$  must be an open saturated set $U$ contained in $L$ and that $\phi:P\to U$ is a submersion.
\end{proof}

\section{Construction of solutions}
\label{sec:construction}
%%%%%%%%%%%%%%%%%%%%%%%%%%%%%%%%%%%%%%%%%%%%%%%%%%%%%%%%%%%%%%%%%%
%%%%%%%%%%%%%%%%%%%%%%%%%%%%%%%%%%%%%%%%%%%%%%%%%%%%%%%%%%%%%%%%%%

We have now an appropriate language to deal with Cartan's realization problem. In this section we start by establishing the dictionary that gives the correspondence between Cartan's realization problem and its solutions with $G$-structure algebroids with connection and their $G$-realizations. Then we start exploring it to solve the existence problem. 

%%%%%%%%%%%%%%%%%%%%%%%%%%%%%%
\subsection{The dictionary I}
%%%%%%%%%%%%%%%%%%%%%%%%%%%%%%

The first piece of the dictionary concerns Cartan Data. The following result is essentially due to Bryant  \cite[Appendix A.4]{Bryant}, but he ignores the presence of the group which, as we shall see, is a crucial aspect: 

\begin{theorem}
\label{thm:Cartan:G:alg:A}
Let $(G, X,c,R,F)$ be a Cartan Data. If for each $x \in X$ the associated Cartan Problem has a solution $(\B_G(M), (\theta,\omega), h)$ with $x \in \Im h$, then the Cartan Data determines a $G$-structure algebroid $A\to X$ with connection. Conversely, every $G$-structure algebroid with connection determines Cartan Data. 
\end{theorem}

\begin{proof}
Given Cartan Data $(G, X,c,R,F)$ we construct a $G$-structure algebroid on $A=X\times(\Rr^n\oplus\gg)\to X$, as in Proposition \ref{prop:canonical:form}. The only issue is if the Jacobi identity holds or, equivalently, if  $\d_A:\Omega^\bullet(A)\to\Omega^{\bullet+1}(A)$ satisfies $\d_A^2=0$. Let $(\B_G(M), (\theta,\omega), h)$ be a solution with $x \in \Im h$. Then we have a bundle map 
\[ \Phi:T\B_G(M)\to A, \qquad (p,v)\mapsto (h(p),\theta_p(v),\omega_p(v)). \]
The structure equations for$(\theta,\omega,h)$ and the definition of $\d_A$ shows that:
\[ \Phi^*\d_A=\d\Phi^*. \]
Given that the de Rham differential satisfies $\d^2=0$ and $\Phi$ is fiberwise isomorphism, we must have $\d_A^2=0$ at $x$. Hence, if there is one such solution for each $x\in X$, we conclude that $\d_A^2=0$ so $A$ is a Lie algebroid.

%Corollary \ref{cor:Jacobi:identity} shows that Cartan Data $(G, X,c,R,F)$ determines a  $G$-structure algebroid with connection in canonical form provided it satisfies:
%\begin{align*}
%\bigodot_{u,v,w} \left\{F(u)(c(v,w))+c(c(u,v),w)\right\}&=\bigodot_{u,v,w} R(u,v)w,\\
%\bigodot_{u,v,w} \left\{F(u)(R(v,w))+ R(c(u,v),w)\right\}&=0.
%\end{align*}
%So we need to check that these equations hold if for every $x\in X$ the associated Cartan Problem has a solution $(\B_G(M), (\theta,\omega), h)$ with $x \in \Im h$. For this, observe that this solution $(\B_G(M), (\theta,\omega), h)$ must satisfy the structure equations:
%\[
%\left\{
%\begin{array}{l}
%\d \theta=c (h)( \theta \wedge \theta) - \omega \wedge \theta\\
%\d \omega= R (h)( \theta \wedge \theta) 
%- \omega\wedge\omega\\
%\d h = F(h,\theta) + \psi(h,\omega)
%\end{array}
%\right.
%\]
%Now one checks easily that the equations above for $F$, $R$ and $c$, follow by differentiating these structure equations and using that $\d^2=0$ (this computation is done in local coordinates in \cite[Appendix A.4]{Bryant} or \cite[Prop. 3.2]{FernandesStruchiner1}).

Conversely, by Proposition \ref{prop:normal:form} a Lie algebroid with connection is naturally isomorphic to one in canonical form, and by Proposition \ref{prop:canonical:form} it determines Cartan Data.
\end{proof}

The previous theorem, together with Corollary \ref{cor:local:existence:2} asserting that there always exist local solutions of a Cartan Problem associated with a $G$-structure algebroid with connection, establishes a 1:1 correspondence:
\[
 \left\{\txt{Cartan Data\\ admitting solutions}\right\} 
\quad
\stackrel{1-1}{\longleftrightarrow}\
\quad
\left\{\txt{$G$-structure algebroid $A$\\ with connection}\right\}
\]

\begin{remark}
One half of Theorem \ref{thm:Cartan:G:alg:A} is a special case of the following more general statement whose proof is identical. Let $A \to X$ be an almost Lie algebroid (i.e., $A$ has a bracket and an anchor map which maps brackets of sections of $A$ to brackets of vector fields and for which the Leibniz identity holds, but the Jacobi identity for the bracket may fail). If for every $x \in X$ there exists a fiberwise surjective morphism of almost Lie algebroids
\[\xymatrix{TP \ar[r] \ar[d] & A\ar[d]\\
P \ar[r]_h & X}\]
with $x \in \Im h$, then $A$ is a Lie algebroid.
\end{remark}
\medskip

%%%%%%%%%%%%%%%%%%%%%%%%%%%%%%
\subsection{The dictionary II}
%%%%%%%%%%%%%%%%%%%%%%%%%%%%%%

The second piece of the dictionary concerns solutions of Cartan Problem. 

We saw that Cartan Data $(G, X,c,R,F)$ determines a  $G$-structure algebroid with connection. One the other hand, by Examples \ref{ex:principal:bundle:connection:algbrd} and \ref{ex:G:algbrd:connect}, a solution $(\B_G(M), (\theta,\omega), h)$ determines also a  $G$-structure algebroid with connection, namely $T\B_G(M)\to \B_G(M)$. Moreover, as in the proof of Theorem \ref{thm:Cartan:G:alg:A}, we have a Lie algebroid morphism 
\[
\xymatrix{
T(\B_G(M)) \ar[r]^--{(\theta,\omega)} \ar[d] & A\ar[d]\\
\B_G(M)\ar[r]_{h} & X}
\]
which is actually a morphisms of  $G$-structure algebroids with connection. In fact, we have:

\begin{theorem}
\label{thm:Cartan:G:alg:B}
Let $(G, X,c,R,F)$ be a Cartan Data defining a $G$-structure algebroid with connection $A\to X$. The solutions of the associated problem are in one to one correspondence with $G$-realizations of $A$:
\[
 \left\{\txt{Solutions of Cartan's\\ realization problem}\right\} 
\quad
\stackrel{1-1}{\longleftrightarrow}\
\quad
\left\{\txt{$G$-realizations of the\\ algebroid $A$}\right\}\qquad
\]
\end{theorem}

\begin{proof}
It is proved in \cite[Prop. 3.5]{FernandesStruchiner1} that in the case $G=\{e\}$ solutions are in 1:1 correspondence with bundle maps which are Lie algebroid morphisms and fiberwise isomorphisms. Applying this to our case, we observe that:
\begin{enumerate}[(a)]
\item Given a solution $(\B_G(M),(\theta,\omega),h)$, the bundle $T(\B_G(M))\to \B_G(M)$ is a $G$-structure algebroid with connection (Example \ref{ex:G:algbrd:connect}). The bundle map 
 \[ \Phi:T\B_G(M)\to A, \qquad (p,v)\mapsto (h(p),\theta_p(v),\omega_p(v)), \]
is a Lie algebroid morphism. Furthermore, due the canonical form of $A$, this morphism is $G$-equivariant and we have a commutative diagram:
\[
\xymatrix{
\B_G(M)\rtimes \gg\ar[r]^i\ar[d]& T(\B_G(M))\ar[d]\\
X\rtimes \gg\ar[r]_i&  A}
\]
Hence, $(\theta,\omega):T(\B_G(M))\to A$ is a $G$-realization.

\item A $G$-realization of $A$ is a  a $G$-structure algebroid morphism $\Phi:TP\to A$, where $TP\to P$ is a a $G$-structure algebroid with connection covering a map $h:P\to X$. Then $P\to P/G$ is a $G$-structure with connection (Example \ref{ex:G:algbrd:connect}). Recall that $\Phi$ is necessarily a fiberwise isomorphism, and since $A$ is in canonical form, it must take the form $\Phi(v_x)=(x,\theta(v),\omega(v))$, where $\theta$ and $\omega$ are the tautological and connection 1-forms in $P$.  Hence, $\theta$ and $\omega$ are the $h$-pullback of the tautological and connection 1-forms of $A$. Therefore, since the later satisfy \eqref{eq:structure:2}, the former satisfy the structure equations:
\begin{equation*}
\label{eq:structure} \left\{
\begin{array}{l}
\d \theta=c (h)( \theta \wedge \theta) - \omega \wedge \theta\\
\d \omega= R (h)( \theta \wedge \theta) 
- \omega\wedge\omega
\end{array}
\right.
\end{equation*}
Finally, the fact that $\d h:TP\to TX$ permutes the anchors, means that:
\[
\d h = F(h,\theta) + \psi(h,\omega)
\]
This shows that $(P,(\theta,\omega),h)$ is a solution of the Cartan's realization problem determined by the data $(G, X,c,R,F)$.
\end{enumerate}
\end{proof}

Ones say that $(\B_G(M_i),(\theta_i,\omega_i),h_i)$, $i=1,2$, are (locally) equivalent solutions of Cartan's Realization Problem if there exists a (locall defined) diffeomorphism  $\varphi:M_1\to M_2$ whose lift to the frame bundles maps one $G$-structure to the other, preserves the connection 1-forms and commutes with the classifying maps:
\[ 
\xymatrix{
\B_G(M_1)\ar[dr]_{h_1}\ar[rr]^{\tilde{\varphi}} & & \B_G(M_2)\ar[dl]^{h_2} \\
 & X
}
\]
Clearly, under our dictionary, we have:
% (local) equivalent solutions correspond to (local) equivalent $G$-realizations.
\[\quad
 \left\{\txt{(locally) equivalent solutions of\\ Cartan's realization problem}\right\} 
\ 
\stackrel{1-1}{\longleftrightarrow}
\ 
\left\{\txt{(locally) isomorphic $G$-realizations\\of the algebroid $A$}\right\}\qquad
\]

%%%%%%%%%%%%%%%%%%%%%%%%%%%%%%
\subsection{Existence of solutions I}

\begin{definition}
A $G$-structure algebroid with connection $A\to X$ is {\bf$G$-integrable} if there exists a $G$-structure groupoid with connection $\G\tto X$ integrating it. We  call $\G$ a {\bf $G$-integration} of $A$. 
\end{definition}

As we saw in Section \ref{sec:G-realizations}, when $A\to X$ is $G$-integrable, a source fiber of a $G$-integration $\G$ gives rise to a $G$-realization $(\s^{-1}(x_0),\wmc,\t)$. Hence, as a consequence of Theorem \ref{thm:Cartan:G:alg:B}, we conclude that:

\begin{corol}
\label{cor:local:existence}
Let $(G, X,c,R,F)$ be Cartan Data and assume that the associated Lie algebroid with connection $A\to X$ is $G$-integrable. Then each source fiber $(\s^{-1}(x_0),\wmc,\t)$ of a $G$-integration $\G\tto X$ yields a solution of the corresponding Cartan's Realization Problem.
\end{corol}

Under the assumptions of this corollary, since $A$ has fiber $\Rr^n\oplus\gg$, the Maurer-Cartan form $\wmc$ gives rise to a pair of 1-forms on the source fiber:
\begin{align*} 
\theta:&=\pr_{\Rr^n}\wmc|_{\s^{-1}(x_0)}\in\Omega^1(\s^{-1}(x_0),\Rr^n),\\ 
\omega:&=\pr_{\gg}\wmc|_{\s^{-1}(x_0)}\in\Omega^1(\s^{-1}(x_0),\gg).
\end{align*}
These are the tautological 1-form and the connection 1-form of the $G$-structure $\s^{-1}(x_0)\to \s^{-1}(x_0)/G$ solving the Cartan's realization problem. For this reason, we will often denote a $G$-realization of $A\to X$ by $(P,(\theta,\omega),h)$.

A Cartan's realization problem, in general, has many other solutions besides the ones arising as source fibers of $G$-integrations. Moreover, solutions can exist even when the $G$-structure algebroid with connection associated with the problem is not $G$-integrable. For example, Lemma \ref{lem:classifying:map} shows that the image of the classifying map of any $G$-realization is an open $G$-saturated set of a leaf. The restriction of a $G$-structure algebroid $A\to X$ to an open $G$-saturated subset $U\subset L$ of a leaf is still a $G$-structure algebroid. Hence, by the previous Corollary, given $x_0\in X$ we can produce a solution $(P,(\theta,\omega),h)$ of Cartan's realization problem with $x_0\in \Im h$ if we can show that $A|_U$ is $G$-integrable, where $U$ is some $G$-saturated open neighborhood $x_0\in U\subset L$ in the leaf $L$ containing $x_0$. We will discuss $G$-integrability in the next sections and we will show that such a neighborhood $U$ always exists even when $A\to X$ is not $G$-integrable (Theorem \ref{thm:local:G:int}).

%%%%%%%%%%%%%%%%%%%%%%%%%%%%%%
\subsection{Universality of solutions}
%%%%%%%%%%%%%%%%%%%%%%%%%%%%%%

When $G$ is a connected Lie group, solutions arising from source fibers of $G$-integrations, as in the previous paragraph, are universal in the following sense: 

\begin{theorem}[Local Universal Property of $\wmc$]
\label{thm:MC:local}
Let $G\subset \GL(n,\Rr)$ be connected, and let $\G\tto X$ be a $G$-structure groupoid integrating a $G$-structure algebroid with connection $A\to X$. Given a $G$-realization $(P,(\theta,\omega),h)$ of $A$, choose $p_0\in P$ and let $x_0=h(p_0)$. Then there is a neighborhood $U\subset M=P/G$ of $m_0=\pi(p_0)$ and a unique embedding $\phi:P|_U\to \s^{-1}(x_0)$ of $G$-realizations with $\phi(p_0)=1_{h(p_0)}$, making the following diagram commute:
\[
\xymatrix@R=7 pt{ T(P|_U) \ar@{-->}[rr]^>>>>>>>>{\phi_{\ast}} \ar[dd]  \ar[dr]_{(\theta,\omega)} & & T(\s^{-1}(x_0)) \ar[dd] \ar[dl]^{\wmc} \\
& A \ar[dd] & \\
P|_U \ar@{-->}[rr]^>>>>>>>>{\phi} \ar[dr]_h & & \s^{-1}(x_0)\ar[dl]^{\t} \\
& X &}
\]
% If $\G=\Sigma_G(A)$ is the canonical integration and $P/G$ has trivial orbifold fundamental group, then $\phi$ is globally defined (i.e., we can take $U=P$).
\end{theorem}

\begin{proof}
We can assume that $A$ is in canonical form. It follows from \cite[Thm 4.6]{FernandesStruchiner1} that  there is a neighborhood $V\subset P$ of $p_0$ and a unique embedding $\phi:V\to \s^{-1}(x_0)$ making the diagram above commutative and satisfying $\phi(p_0)=1_{x_0}$. Since
\[ \phi^*\wmc=(\theta,\omega). \]
it follows that $\phi$ must be $\gg$-equivariant. Since $G$ is connected and acts properly on $P$, we can choose $V$ is such that
\begin{equation}
\label{eq:local:equivariance}
\phi(pg)=\phi(p)g,\quad\text{whenever } g\in G,\ p,pg\in V.
\end{equation}
Hence, the result will follow provided that we show that we can choose $V$ to be $G$-saturated, for then we can set $U:=V/G$.

If $V$ is not saturated let $\bar{V}$ be the saturation of $V$ and extend $\phi$ to an embedding $\bar{\phi}:\bar{V}\to \s^{-1}(x_0)$ as follows: if $\bar{p}\in \bar{V}$ then there exists $p\in V$ and $g\in G$ such that $\bar{p}=pg$ and we set:
\[ \bar{\phi}(\bar{p}):=\phi(p)g. \]
The map $\bar{\phi}$ is well-defined: if $\bar{p}=pg=p'g'$ with $p,p'\in V$ and $g,g'\in G$ then $p=p'g'g^{-1}$, and we have by \eqref{eq:local:equivariance}:
\[ \phi(p)=\phi(p'g'g^{-1})=\phi(p')g'g^{-1} \quad \Rightarrow \quad \phi(p)g=\phi(p')g. \]
Since $\bar{\phi}$ is $G$-equivariant, every orbit in $\bar{V}$ intersects the open $V$, and the restriction $\bar{\phi}|_V=\phi$ is an embedding, it follows that $\bar{\phi}:\bar{V}\to \s^{-1}(x_0)$  is a $G$-equivariant embedding. One checks easily that it still makes the diagram above commute and hence is a morphism of $G$-realizations.
\end{proof}

%%%%%%%%%%%%%%%%%%%%%%%%%%%%%%
\subsection{An elementary example: metrics of constant curvature}
\label{sec:constant:curvature}
%%%%%%%%%%%%%%%%%%%%%%%%%%%%%%
To illustrate how the dictionary and the integration procedure above produce solutions of a Cartan's realization problem, we consider what is perhaps the simplest example: metrics of constant curvature. 

Given an oriented Riemannian manifold $(M,g)$ with a metric of constant scalar curvature $K$, pass to the bundle of (oriented) orthogonal frames $\B_{\SO(n,\Rr)}(M)$.  The curvature $R$ of the Levi-Civita connection can be re-expressed in terms of the scalar curvature as
\[R(u,v) w=K \left(\langle w,v\rangle u-\langle w,u\rangle v\right).\]
It follows that the Levi-Civita connection gives a connection 1-form $\omega$ which together with the tautological 1-form $\theta$ satisfy the structure equations:
%\footnote{\textcolor{blue}{In the formula bellow I changed the sign of the second equation and added the transpose to $\theta$ in the second equation. What I mean by the expression $d \omega= - K\, \theta \wedge \theta^T$ is that $d \omega$ is a $2$-form with values in $\mathfrak{so}(n, \Rr)$ such that 
%\[\d \omega(\xi_u, \xi_v) (w)= -K(\theta(\xi_u)(\theta(\xi_v)^Tw)-\theta(\xi_v)(\theta(\xi_u)^Tw))= -K\left(\langle w,v\rangle u-\langle w,u\rangle v\right).\]
%Maybe there is a better form of expressing this.}}:
\[
\left\{
\begin{array}{l}
\d \theta=- \omega \wedge \theta\\
\d \omega=  R(K)\, \theta \wedge \theta - \omega\wedge\omega
\end{array}
\right.
\]
and the condition that the scalar curvature is constant reads:
\[
\d K=0.
\]
These equations defines a Cartan's Realization Problem.  Using $x=K$ as a coordinate on $\Rr$, we see that the associated Lie $\SO(n,\Rr)$-algebroid with connection is precisely the constant scalar curvature Lie algebroid of Example \ref{ex:constant:curvature}:
\[ A=\Rr\times(\Rr^n\oplus\so(n,\Rr))\to \Rr. \] 
This is a bundle of Lie $\SO(n,\Rr)$-algebras with fibers:
\begin{itemize}
\item for $x>0$, $A_x\simeq \so(n+1,\Rr)$;
\item for $x=0$, $A_0\simeq \mathfrak{so}(n,\Rr)\ltimes\Rr^n$;
\item for $x<0$, $A_x\simeq \mathfrak{so}(n,1)$.
\end{itemize}
We will see later that $A$ integrates to a Lie $\SO(n,\Rr)$-groupoid $\G\tto \Rr$ which is a bundle of Lie groups  (so $\s=\t$) with fibers:
\begin{itemize}
\item for $x>0$, $\s^{-1}(x)\simeq \SO(n+1,\Rr)$;
\item for $x=0$, $\s^{-1}(0)\simeq \SO(n,\Rr)\ltimes \Rr^n$;
\item for $x<0$, $\s^{-1}(x)\simeq \SO(n,1)^+$.
\end{itemize}

The group $\SO(n,\Rr)$ acts trivially on the base $\Rr$ of the groupoid/algebroid. The action morphism $\iota: \Rr\rtimes \SO(n,\Rr)\to \G$ is, for each $x\in\Rr$, the obvious inclusion of $\SO(n,\Rr)$ in each fiber. In particular, we find that this $G$-integration produces solutions the $\SO(n,\Rr)$-structures $\s^{-1}(x)\to \s^{-1}(x)/\SO(n,\Rr)$:
\begin{itemize}
\item for $x>0$, $\SO(n+1,\Rr)\to \Ss^n$;
\item for $x=0$, $\SO(n,\Rr)\ltimes \Rr^n\to \Rr^n$;
\item for $x<0$, $ \SO(n,1)^+\to \mathbb{H}^n$.
\end{itemize}
These are, of course, the bundle of (oriented) orthogonal frames of the round metric on the sphere, the flat metric on the plane and the hyperbolic metric of constant scalar curvature $x$.

We will see later that in this example there is only one $\SO(n,\Rr)$-integration.  Of course there are many more manifolds with metrics of constant curvature besides these (even complete ones, such as the flat torus). Since the leaves are just the points of $\Rr$, restricting to opens in the leaves does not add solutions. One can consider $\SO(n)$-saturated opens in the $\s$-fibers: these just give open sets of the solutions above. We will explain later how additional solutions can be obtained, after we discuss $G$-integrability.

\section{$G$-integrability}
\label{sec:integrability}
%%%%%%%%%%%%%%%%%%%%%%%%%%%%%%
%%%%%%%%%%%%%%%%%%%%%%%%%%%%%%

According to the results in the previous section, in order to find solutions to Cartan's realization problem it is important to understand the following problem:
\begin{itemize}
\item When is a $G$-structure algebroid with connection $G$-integrable?
\end{itemize}
In this section we will give a complete answer to this question. In particular, we will see that $G$-integrability is stronger than (ordinary) integrability of an algebroid. From now on, to simplify the discussion, we will assume that $G$ is connected.%\footnote{\textcolor{blue}{Maybe add the case where $G$ is a semi-direct product of a connected Lie group with a discrete group.}}
 
 A first observation is that the connection and tautological form do not play any role here. On the one hand, the correspondence \eqref{eq:correspondence:connections} shows that if $A\to X$ is a Lie algebroid with connection $H$ and $\G\tto X$ is a $G$-structure groupoid integrating it, then there is a unique connection $\H$ on $\G$ with $\H_x=H_x$, for all $x\in M$. On the other hand, by Proposition \ref{prop:G:structures:integrable}, a similar argument applies to the tautological forms: if $A$  is a $G$-structure algebroid with tautological form $\theta$, and $\G$ is a $G$-principal groupoid integrating $A$, then $\G$ has a unique tautological form $\Theta$ with $\Theta|_M=\theta$.

Therefore, in this section we will assume that $A\to X$ is a $G$-principal algebroid and we will look for a $G$-principal groupoid $\G\tto X$ integrating it. If the latter exists we will say that $A$ is \emph{$G$-integrable} and call $\G$ a \emph{$G$-integration} of $A$. 

%All results in this section will be stated for $G$-structure algebroids and $G$-structure groupoids, but in the proofs we need not argue about the tautological forms. On the other hand, all results in this section remain true if in the statements $G$-structure algebroids and $G$-structure groupoids are replaced, respectively, by $G$-principal algebroids and $G$-principal groupoids.

%%%%%%%%%%%%%%%%%%%%%%%%%%%%%%
\subsection{The canonical $G$-integration}
%%%%%%%%%%%%%%%%%%%%%%%%%%%%%%

A Lie algebroid $A\to X$ may fail to be integrable, but when it is integrable it may have many integrations (see \cite{CrainicFernandes,CrainicFernandes:lectures}). Still, there is a unique (up to isomorphism) integration $\Sigma(A)\tto X$ which is characterized by having 1-connected source fibers. Alternatively, one can characterize the groupoid $\Sigma(A)$ as the \emph{maximal integration} among all $\s$-connected integrations of $A$: for any $\s$-connected integration $\G\tto X$ of $A$ there exists a unique \'etale, surjective, groupoid morphism $\Sigma(A)\to\G$. 

For a $G$-principal algebroid $A$ which is $G$-integrable one can also look for a maximal $G$-integration. The following simple examples show that, if such an integration exists, it may be different from $\Sigma(A)$.

\begin{example}
\label{ex:trivial:1}
Consider the trivial $G$-structure Lie algebra $A=\Rr^n\oplus\gg$ (see Example \ref{ex:G:algebras}). It always admit a $G$-integration, namely the group semi-direct product $\G=\Rr^n\rtimes G$: the action morphism $i:\gg\to \Rr^n\oplus\gg$ integrates to the inclusion $\iota: G\to \Rr^n\rtimes G$, $g\mapsto(0,g)$. 

When $G$ is not $1$-connected, the 1-connected integration $\Sigma(A)=\Rr^n\rtimes \tilde{G}$ fails to be a $G$-integration: there is no action morphism $\iota:G\to \Rr^n\rtimes \tilde{G}$. In fact, it is not hard to see that in this case $\Rr^n\rtimes G$ is the only $G$-integration. 
\end{example}

\begin{example}
\label{ex:curvature:1}
Consider the constant positive curvature $\SO(n,\Rr)$-structure Lie algebra $A=\Rr^n\oplus\so(n,\Rr)$ with zero torsion (see Example \ref{ex:constant:curvature}):
\[ R(u,v)\cdot w=\left(\langle w,v\rangle u-\langle w,u\rangle v\right). \]
It is easy to check that $A\simeq \so(n+1,\Rr)$, as Lie algebras. Then we have:
\begin{enumerate}[(i)]
    \item If $n$ is even, there is one $\SO(n,\Rr)$-integration, namely 
    \[ \G=\SO(n+1,\Rr),\] 
    and we find $\s^{-1}(x)/\SO(n,\Rr)\simeq \Ss^n$.
    \item If $n$ is odd, there are two $\SO(n,\Rr)$-integrations, namely 
    \[ \G=\SO(n+1,\Rr)\quad\text{ and }\quad \G'=\mathrm{PSO}(n+1,\Rr).\] 
    In the first case, we find that $\s^{-1}(x)/\SO(n,\Rr)\simeq \Ss^n$ while in the second case we find that $\s^{-1}(x)/\SO(n,\Rr)\simeq \mathbb{RP}^n$.
\end{enumerate}
Note that if $n$ is even then $\mathbb{RP}^n$ is not orientable, so it does not admit a $\SO(n,\Rr)$-structure. 
\end{example}

\begin{theorem}
\label{thm:canonical:G:integration}
Let $A\to X$ be a $G$-integrable, $G$-principal algebroid. Then there exists a unique (up to isomorphism) $G$-principal groupoid $\Sigma_G(A)\tto X$ which is characterized by either of the following:
\begin{enumerate}[(a)]
\item $\Sigma_G(A)$ is an $\s$-connected $G$-integration and the orbifold fundamental groups of $\s^{-1}(x)/G$ are trivial;
\item $\Sigma_G(A)$ is maximal among all $\s$-connected $G$-integrations of $A$: for any $\s$-connected $G$-integration $\G$ there exists a unique \'etale, surjective, morphism of $G$-principal groupoid $\Sigma_G(A)\to\G$. 
\end{enumerate}
\end{theorem}

\begin{proof}
We start by defining $\Sigma_G(A)$ so that (b) holds. Let $\tilde{G}$ be the 1-connected Lie group integrating $\gg$ and let $\G\tto X$ be some $\s$-connected $G$-integration of $A$. Applying Lie's 2nd Theorem (\cite{CrainicFernandes:lectures}) we obtain a commutative diagram of groupoid morphisms
\[
\xymatrix{X\rtimes \tilde{G} \ar[r]^{\tilde{\iota}} \ar[d] & \Sigma(A) \ar[d]\\
X\rtimes G \ar[r]_{\iota}  & \G} 
\]
where $\tilde{\iota}$ and $\iota$ are the morphisms integrating $i:X\rtimes \gg\to A$. Since $\Sigma(A)\to \G$ is a covering of Lie groupoids, it follows that
\[ \widetilde{\cN}^G:=\tilde{\iota}(X\rtimes \pi_1(G))\subset \Sigma(A), \]
is a discrete, bundle of Lie groups in the center of the isotropy. It is also a normal sub-bundle of groups as will be shown in Proposition \ref{prop:normal}. Hence, the quotient
\[
\Sigma_G(A):=\Sigma(A)/\widetilde{\cN}^G
\]
is a $\s$-connected Lie groupoid integrating $A$ and the diagram above factors:
\[
\xymatrix@R=20pt{
X\rtimes \tilde{G} \ar[r]^{\tilde{\iota}} \ar[d] & \Sigma(A) \ar[d]\\
X\rtimes G \ar[r] \ar[d] & \Sigma_G(A) \ar[d]\\
X\rtimes G \ar[r]_{\iota}  & \G} 
\]
This shows that $\Sigma_G(A)\to\G$ is an \'etale, surjective, morphism of $G$-principal groupoids. Hence, $\Sigma_G(A)$ is a $\s$-connected integration which is maximal among all $\s$-connected integrations. Uniqueness follows from the uniqueness in Lie's 2nd Theorem.

Now to prove (a), we first observe that the groupoid $\Sigma_G(A)$ we have just constructed has this property. It is clearly $\s$-connected. Now denote by $\s^{-1}(x)$ and $\tilde{\s}^{-1}(x)$ the source fibers of $\Sigma_G(A)$ and $\Sigma(A)$, respectively. The quotient morphism $\Sigma(A)\to \Sigma_G(A)$ and the groupoid morphisms $\tilde{\iota}:X\rtimes \tilde{G} \to \Sigma(A)$, $\iota: X\rtimes G\to \Sigma_G(A)$ yield the commutative diagram
\[
\xymatrix@R=20pt{
\pi_1(G)\ar[r]\ar[d] & \widetilde{\cN}^G\ar[d] \\
\tilde{G}\ar[r]\ar[d] & \tilde{s}^{-1}(x) \ar[r]\ar[d] &  \tilde{s}^{-1}(x)/\tilde{G}\ar[d] \\
G\ar[r]& \s^{-1}(x) \ar[r] &  \s^{-1}(x)/G
}
\]
Here the two first columns and the two bottom rows are stack fibrations (see \cite[Example 4.5]{Noohi14}) and the top row is a surjective group morphism. Passing to the corresponding long exact sequences in homotopy (\cite[Theorem 5.2]{Noohi14}), using that $\tilde{s}^{-1}(x)$ and $\tilde{G}$ are 1-connected, we obtain the commutative diagram of stacky homotopy groups
\[
\xymatrix@R=20pt{
 & 1\ar@{=}[r] \ar[d] & 1\ar[d] \\
\pi_2(\s^{-1}(x)/G)\ar[r] & \pi_1(G)\ar[r]\ar[d] & \pi_1(\s^{-1}(x)) \ar[r]\ar[d] & \pi_1(\s^{-1}(x)/G)\ar[r] & 1 \\
& \pi_1(G) \ar@{->>}[r]\ar[d] & \widetilde{\cN}^G\ar[d] \\
& 1\ar@{=}[r] & 1
}
\]
We conclude that the map $\pi_1(G)\to \pi_1(\s^{-1}(x))$ is surjective and so we must have $\pi_1(\s^{-1}(x)/G)=1$. 

Finally, let $\G\tto X$ be a $G$-integration of $A$ with $\s$-connected fibers and assume that $\pi_1(\s_\G^{-1}(x)/G)=1$ for all $x\in X$, where $\s_\G: \G \to X$ denotes the source map of $\G$. We claim that the groupoid morphism $\Sigma_G(A)\to \G$ given in part (b) must be an isomorphism.

In order to show that the projection $p: \Sigma_G(A)\to \G$ is an isomorphism, it suffices to show that the restriction of this projection to each source fiber is injective. We denote, as above, the source map of $\Sigma_G(A)$ by $\s: \Sigma_G(A) \to X$. Then for each $x \in X$ we have a map of fibrations
\[\xymatrix{G \ar@{=}[d] \ar[r] & \s^{-1}(x) \ar[d]^p \ar[r] &\s^{-1}(x)/G \ar[d]\\
G  \ar[r] & \s_\G^{-1}(x)  \ar[r] &\s_\G^{-1}(x)/G,}\]
which induces the commutative diagram with exact lines at the level of fundamental groups
\[\xymatrix{\pi_1(G) \ar@{=}[d] \ar[r] & \pi_1(\s^{-1}(x)) \ar[d]^{p_*} \ar[r] &1 \ar@{=}[d]\\
\pi_1(G)  \ar[r] & \pi_1(\s_\G^{-1}(x))  \ar[r] &1.}\]
Hence, $p_*: \pi_1(\s^{-1}(x)) \to\pi_1(\s_\G^{-1}(x))$ is surjective, and $p: \s^{-1}(x) \to \s_\G^{-1}(x)$ being a connected covering, it follows that $p$ is injective. This concludes the proof.
\end{proof}

\begin{example} 
\label{ex:canonical:integ:principal:bundle}
Given a principal bundle $\pi: P \to M$ the canonical integration of the $G$-principal algebroid $TP \to P$ can be obtained as follows. Let $q: \tilde{M} \to M$ be the universal covering space of $M$ and consider the pullback diagram
\[\xymatrix{q^*P \ar[d]_\pi \ar[r]^{\hat{q}} & P\ar[d]^\pi\\
\tilde{M} \ar[r]_q & M.}\]
On the one hand $q^*P$ is a principal $G$-bundle over $\tilde{M}$ and therefore caries a right principal $G$-action. On the other hand, $q^*P$ is a $\pi_1(M)$-covering of $P$. The canonical $G$-integration of $TP$ is the gauge groupoid corresponding to the $\pi_1(M)$-principal bundle $\hat{q}: q^*P \to P$:
\[ \Sigma_G(TP)=(q^*P\times q^*P)/\pi_1(M)\tto P, \]
with the $G$-action $[(p_1,p_2)]\, g:=[(p_1g,p_2)]$.
\end{example}

Theorem \ref{thm:canonical:G:integration} shows that any $\s$-connected $G$-integration $\G$ of a $G$-principal algebroid $A$ is a quotient of the canonical $G$-integration. The kernel of the \'etale, surjective, morphism $\Sigma_G(A)\to\G$ is a discrete bundle of subgroups lying in the the center of the isotropy groups of $\Sigma_G(A)$, whose intersection with the image of the morphism $\iota:X\rtimes G\to \Sigma_G(A)$ is the identity section. 

Conversely, given a discrete bundle of subgroups $\Delta\subset \Sigma_G(A)$, contained in the the center of the isotropy groups and intersecting the image of the morphism $\iota:X\rtimes G\to \Sigma_G(A)$ in the identity section, one has the $\s$-connected $G$-integration 
\[ \G:=\Sigma_G(A)/\Delta\tto X. \]

\begin{example}
\label{ex:curvature:2}
We can now better understand what happens with with the $\SO(n,\Rr)$-structure Lie algebra $A=\Rr^n\oplus\so(n,\Rr)$ with zero torsion and constant positive curvature of Example \ref{ex:curvature:1}. The canonical  $G$-integration is $\Sigma_G(A)=\SO(n+1,\Rr)$. Note that the center of $\SO(n+1,\Rr)$ is
\[
Z(\SO(n+1,\Rr))=
\left\{
\begin{array}{l}
\quad \{ I \}, \text{ if $n$ is even,}\\
\\
\{I,-I\}, \text{ if $n$ is odd.}
\end{array}
\right.
\]
So if $n$ is even the only connected $\SO(n,\Rr)$-integration is $\SO(n+1,\Rr)$, while if $n$ is odd there is another one, namely 
\[ \G=\SO(n+1,\Rr)/\{I,-I\}=\mathrm{PSO}(n+1,\Rr). \]
\end{example}

%\begin{proof}
%COMPLETE THIS PROOF
%\end{proof}

%%%%%%%%%%%%%%%%%%%%%%%%%%%%%%
\subsection{Existence of solutions II}
%%%%%%%%%%%%%%%%%%%%%%%%%%%%%%

The $G$-realizations of a $G$-structure algebroid $A$ with connection which have trivial orbifold fundamental group can be embedded into the canonical $G$-integration $\Sigma_G(A)$. This will follow as a consequence of a version of Lie's 2nd Theorem (integration of morphisms) in the context of $G$-principal algebroids. We recall that throughout this section we are assuming that $G$ is connected.

Theorem \ref{thm:canonical:G:integration} can be interpreted as a generalization of Lie's 1st Theorem \cite{CrainicFernandes} to the category of $G$-principal algebroids. Analogously, we have the following version of Lie's 2nd Theorem.

\begin{theorem}\label{thm:integration:G:morphisms} Let $G$ be a connected Lie group, and let $A \to X$ and $B \to Y$ be $G$-integrable $G$-principal algebroids. If $\phi: A \to B$ is a morphism of $G$-principal algebroids, then there exists a unique morphism of $G$-principal groupoids
\[ \Phi: \Sigma_G(A) \to \Sigma_G(B) \] 
which integrates $\phi$, i.e., such that $\Phi_* = \phi$.
\end{theorem}

\begin{proof}
Given a morphism of $G$-principal algebroids 
\[\xymatrix{A \ar[d] \ar[r]^\phi & B\ar[d]\\
X \ar[r]_\varphi & Y,}\]
by definition $\phi$ intertwines the action morphisms, so we obtain a commutative diagram of Lie algebroid morphisms
\[
\xymatrix{A \ar[r]^\phi & B\\
X \rtimes \gg \ar[r]_{\varphi \times I} \ar[u]^i& Y \rtimes \gg \ar[u]_i.}
\]

Using (the usual version of) Lie's 2nd Theorem \cite{CrainicFernandes} we obtain a commutative diagram of Lie groupoid morphisms
\[
\xymatrix{
\Sigma(A) \ar[r]^{\tilde{\Phi}} & \Sigma(B)\\
X \rtimes \tilde{G} \ar[r]_{\varphi \times I} \ar[u]^{\tilde{\iota}}& Y \rtimes \tilde{G} \ar[u]_{\tilde{\iota}}.}
\] 
It follows that $\tilde{\Phi}: \Sigma(A) \to \Sigma(B)$ maps the bundle of subgroups $\widetilde{\cN}^G(A)$ of $\Sigma(A)$ defined in Theorem \ref{thm:canonical:G:integration} to the bundle of subgroups $\widetilde{\cN}^G(B)$  of $\Sigma(B)$. Therefore it induces a morphism of $G$-principal groupoid morphism 
\[ \Phi: \Sigma_G(A) \to \Sigma_G(B)\] 
which fits into a commutative diagram: \[\xymatrix{\Sigma(A) \ar[d] \ar[r]^{\tilde{\Phi}} & \Sigma(B) \ar[d]\\
\Sigma_G(A) \ar[r]_\Phi & \Sigma_G(A).}\]
Moreover, since the vertical arrows are \'etale morphisms of Lie groupoids, we must have $\Phi_* = \phi$.
\end{proof}

One can apply the previous theorem to deduce the following global version of Theorem \ref{thm:MC:local}.

\begin{theorem}[Global Universal Property of $\wmc$]
\label{thm:MC:global}
Let $A\to X$ be a $G$-integrable $G$-structure algebroid with connection and let $\Sigma_G(A)\tto X$ be its canonical $G$-integration.  Given a $G$-realization $(P,(\theta,\omega),h)$ of $A$ such that $M=P/G$ has trivial orbifold fundamental group, choose $p_0\in P$ and let $x_0=h(p_0)$. Then there is a unique embedding $\phi$ of $G$-realizations with $\phi(p_0)=1_{x_0}$, making the following diagram commute:
\[
\xymatrix@R=7 pt{ T P \ar@{-->}[rr]^>>>>>>>>{\phi_{\ast}} \ar[dd]  \ar[dr]_{(\theta,\omega)} & & T(\s^{-1}(x_0)) \ar[dd] \ar[dl]^{\wmc} \\
& A \ar[dd] & \\
P \ar@{-->}[rr]^>>>>>>>>{\phi} \ar[dr]_h & & \s^{-1}(x_0)\ar[dl]^{\t} \\
& X &}
\]
%If $\G=\Sigma_G(A)$ is the canonical integration and $P/G$ has trivial orbifold fundamental group, then $\phi$ is globally defined (i.e., we can take $U=P$).
\end{theorem}

\begin{proof}
Since we are assuming that $P/G$ has trivial orbifold fundamental group, it follows from the characterization given in Theorem \ref{thm:canonical:G:integration} that the canonical integration of $TP$ is canonically isomorphic to the pair groupoid $P \times P$ (see Example \ref{ex:canonical:integ:principal:bundle}). Thus,  by Theorem \ref{thm:integration:G:morphisms}, the $G$-principal algebroid morphism $(\theta, \omega): TP \to A$ integrates to a $G$-principal groupoid morphism $\Phi: P \times P \to \Sigma_G(A)$. After fixing $p_0 \in P$, we obtain
\[\phi: P \longrightarrow \s^{-1}(h(p_0)), \quad \phi(p) = \Phi(p, p_0).\]
Clearly, $\phi_*: TP \to T\s^{-1}(h(p_0))$ is a morphism of $G$-structure algebroids with connection which satisfies $\phi(p_0) = 1_{h(p_0)}$. It follows that $\phi^*\wmc = (\theta, \eta)$. Since $\phi$ is an open immersion, it is an open embedding.

To prove uniqueness, observe that if $\phi':P\to \s^{-1}(x_0)$ is an embedding of $G$-realizations with $\phi(p_0)=1_{x_0}$, then 
\[ \Phi':P \times P \to \Sigma_G(A), \quad (p_1,p_2)\mapsto \phi'(p_1)\phi'(p_2)^{-1},\] 
is a groupoid morphism integrating $(\theta, \omega): TP \to A$. By uniqueness in Theorem \ref{thm:integration:G:morphisms}, we must have $\Phi'=\Phi$ and we conclude that $\phi'=\phi$.
\end{proof}

As a corollary, in the integrable case we obtain a description of all $G$-realizations (and hence all solutions of Cartan's realization problem) up to cover:

\begin{corol}
\label{cor:integration:realizations}
Let $A\to X$ be a $G$-structure al\-ge\-broid with connection which is $G$-integrable and let $\Sigma_G(A)\tto X$ be the canonical $G$-integration.  
Any $G$-realization of $A$ has a cover isomorphic to a $G$-invariant, open subset of a $G$-realization of the form $(\s^{-1}(x),\wmc,\t)$.
\end{corol}

\begin{proof}
Let $(P, (\theta,\omega),h)$ be a $G$-realization of $A$. Then, setting $M = P/G$, we have a $G$-structure $P \to M$. If we denote by $q: \tilde{M} \to M$ the orbifold universal cover of $M$, then $(q^*P, (q^*\theta, q^*\omega), q^*h)$ is a $G$-realization of $A$ with simply connected base manifold. The result then follows by applying Theorem \ref{thm:MC:global}.
\end{proof}

%%%%%%%%%%%%%%%%%%%%%%%%%%%%%%
\subsection{Obstructions to $G$-integrability}
%%%%%%%%%%%%%%%%%%%%%%%%%%%%%%

Let $A\to X$ be a $G$-principal algebroid. Obviously, for $A$ to be $G$-integrable it must be integrable. So as first obstructions to $G$-integrability we have the well-known integrability obstructions for a Lie algebroid (\cite{CrainicFernandes}). It will be useful to recall them.

First, given a Lie algebroid $A\to X$, for each $x\in X$ there is a group morphism, called the {\bf monodromy morphism} at $x$:
\[ \partial :\pi_2(L,x)\to \cG(\ker\rho_{x}), \]
where $L$ is the leaf containing $x$ and  $\cG(\ker\rho_{x})$(\footnote{It would be more coherent with previous notation to write $\Sigma(\ker\rho_x)$ instead of $\cG(\ker\rho_x)$, but we prefer the later notation to emphasize that it is a Lie group.}) is the 1-connected Lie group integrating the isotropy Lie algebra. The image of this morphism is called the {\bf extended monodromy group} $\widetilde{\cN}_{x}$ and it always lies in the center $Z(\cG(\ker\rho_{x}))$. The relevance of this group comes from the fact that, whenever $A$ is integrable, the connected component of the identity of the isotropy group at $x$ of the canonical integration $\Sigma(A)$ is:
\[ \Sigma(A)_x^0=\cG(\ker\rho_{x})/\widetilde{\cN}_{x}. \]
Hence, the discreteness of $\widetilde{\cN}_{x}$ is a necessary condition for integrability.

One defines also the (restricted) {\bf monodromy group} $\cN_{x}$ to be the additive subgroup of the center $Z(\ker\rho_{x})$ such that:
\[ \exp(\cN_{x})=\widetilde{\cN}_{x}\cap Z(\cG(\ker\rho_{x}))^0. \]
It is proved in \cite{CrainicFernandes} that the following statements are equivalent:
\begin{enumerate}[(i)]
\item The restriction $A|_L$ is integrable;
\item The extended monodromy group $\widetilde{\cN}_{x}$ is a discrete subgroup of $\cG(\ker\rho_{x})$ for some (and hence all) $x\in L$;
\item The monodromy group $\cN_{x}$ is a discrete subgroup of $A_{x}$ for some (and hence all) $x\in L$.
\end{enumerate}
Moreover, the main result of \cite{CrainicFernandes} shows that a Lie algebroid $A\to X$ is integrable iff the monodromy groups are uniformly discrete, i.e., iff there is a open neighborhood $U\subset A$ of the zero section such that:
\[ \cN_x\cap U=\{0_x\}, \quad \forall x\in X. \]

Let us now turn to $G$-integrability. Assuming that $A\to X$ is an integrable $G$-principal algebroid, we consider its 1-connected integration $\Sigma(A)$. Let $\tilde{G}$ be the 1-connected Lie group with Lie algebra $\gg$ so that $\pi_1(G,1)\subset \tilde{G}$ is a discrete subgroup of $Z(\tilde{G})$. The Lie algebroid morphism  $i:X\rtimes \gg \to A$ integrates to a Lie group morphism $\tilde{\iota}: X\rtimes \tilde{G} \to \Sigma(A)$ and we set:

\begin{definition}
Let $A\to X$ be an integrable $G$-principal algebroid. The  {\bf $G$-monodromy morphism} at $x\in X$ is the map
\[ \partial^G_x:\pi_1(G)\to \Sigma(A)_x,\quad g\mapsto \tilde{\iota}(x,g). \]
We call $\widetilde{\cN}^G_{x}:=\Im \partial^G_x$ the \textbf{extended $G$-monodromy group} at $x$.
\end{definition}

Note that these groups already appeared in the proof of Theorem \ref{thm:canonical:G:integration}: there it is shown that under the assumption of $G$-integrability, the extended $G$-monodromy groups form a discrete normal sub-bundle of the center of the isotropy and the canonical integration is precisely:
\[ \Sigma_G(A)=\Sigma(A)/\widetilde{\cN}^G. \]
More generally, without assuming $G$-integrability, we always have:

\begin{prop}\label{prop:normal}
If $A$ be an integrable $G$-principal Lie algebroid, then $\widetilde{\cN}^G\subset\Sigma(A)$ is a normal sub-bundle of groups contained in the center of the isotropy.
\end{prop}

\begin{proof}
The $G$-action on $A$ by algebroid automorphisms integrates to the inner $G$-action on $\Sigma(A)$ by inner automorphisms given by:
\[
\gamma\odot g=\tilde{\iota}(\t(\gamma),\tilde{g})\cdot \gamma\cdot \tilde{\iota}(\s(\gamma),\tilde{g})^{-1},
\]
where $\tilde{g}\in \tilde{G}$ is any element which covers $g\in G$. In particular, if $\tilde{g}\in\pi_1(G)$ then it covers the identity $e\in G$, so we have for each  arrow $\gamma\in\Sigma(A)$:
\[ 
\gamma=\gamma\odot e=\tilde{\iota}(\t(\gamma),\tilde{g})\cdot \gamma\cdot \tilde{\iota}(\s(\gamma),\tilde{g})^{-1},\quad \forall \tilde{g}\in\pi_1(G).
\]
It follows that:
\[ \gamma\cdot  \tilde{\iota}(\s(\gamma),\tilde{g})=\tilde{\iota}(\t(\gamma),\tilde{g})\cdot \gamma,\quad \forall \tilde{g}\in\pi_1(G), \]
so that $\widetilde{\cN}^G\subset \Sigma(A)$ is a normal sub-bundle of groups contained in the center of the isotropy.
\end{proof}

Using this proposition, we define:

\begin{definition}
Let $A\to X$ be an integrable $G$-principal algebroid. The {\bf $G$-monodromy group} at $x$ is the additive subgroup $\cN^G_x\subset Z(\ker\rho|_x)$ such that
\[ \exp(\cN^G_{x})=\widetilde{\cN}^G_{x}\cap Z(\Sigma(A)_x)^0. \]
\end{definition}

Notice that the $G$-monodromy $\cN^G_x$ contains the kernel of the exponential map $\exp:Z(\ker\rho|L)\to Z(\Sigma(A)_x)^0$, i.e., the usual monodromy $\cN_x$.
Our main result concerning $G$-integrability is the following theorem.

\begin{theorem}
\label{thm:G-integrability:obstructions}
Let $A\to X$ be a $G$-principal algebroid, fix a leaf $L\subset X$ and assume that $A|_L$ is integrable. Then the following statements are equivalent:
\begin{enumerate}[(i)]
\item The restriction $A|_L$ is $G$-integrable;
\item The extended $G$-monodromy group $\widetilde{\cN}^G_{x}$ is a discrete subgroup of $\Sigma(A)_x$ for some (and hence all) $x\in L$;
\item The monodromy $G$-group $\cN_{x}$ is a discrete subgroup of $A_{x}$ for some (and hence all) $x\in L$.
\end{enumerate}
\end{theorem}

\begin{proof}
The equivalence between (ii) and (iii) is clear, since the exponential map is a local diffeomorphism around the origin and $\widetilde{\cN}^G_{x}\subset Z(\Sigma(A)_x)$ is a discrete subgroup of $\Sigma(A)_x$ if and only if $\widetilde{\cN}^G_{x}\cap Z(\Sigma(A)_x)^0$ is a discrete subgroup of $\Sigma(A)_x$. 

If the restriction $A|_L$ is $G$-integrable then the proof of Theorem \ref{thm:canonical:G:integration} shows that the extended $G$-monodromy groups $\widetilde{\cN}^G_{x}$ are discrete subgroups of $\Sigma(A)_x$, for all $x\in L$, so (i) implies (ii). 

Conversely, if (ii) holds, then we can define
\[ \Sigma_G(A)|_L:=\Sigma(A)|_L/\widetilde{\cN}^G.\]
This is a Lie groupoid integrating the Lie algebroid $A|_L$ and the groupoid morphism
 \[ \tilde{\iota}: L\rtimes \tilde{G}  \to \Sigma(A)|_L \]
descends to a groupoid morphism
 \[ \iota: L\rtimes G \to \Sigma_G(A)|_L, \]
which still integrates $i:L\rtimes \gg \to A|_L$. Hence, $\Sigma_G(A)|_L$  is a smooth $G$-principal groupoid integrating the $G$-principal algebroid $A|_L$, so (i) holds.
\end{proof}

For the purpose of finding solutions to Cartan's realization problem, the previous result is enough, since we only need $G$-integrations of $A|_L$. However, to understand e.g. the smoothness of the moduli space of solutions it is useful to know if $A$ is $G$-integrable. The theorem has the following easy corollary

\begin{corol}
\label{cor:G:integrability}
Let $A\to X$ be an integrable $G$-principal algebroid. Then $A$ is $G$-integrable iff the $G$-monodromy groups are uniformly discrete, i.e., there is a open neighborhood $U\subset A$ of the zero section such that
\[ \cN^G_x\cap U=\{0_x\}, \quad \forall x\in X. \]
\end{corol}

\begin{proof}
If $A\to X$ is $G$-integrable, then the proof of Theorem \ref{thm:canonical:G:integration} shows that the extended $G$-monodromy groups $\widetilde{\cN}^G$ form an embedded bundle of discrete subgroups of $\Sigma(A)$, so they are uniformly discrete. Conversely, assuming that $\widetilde{\cN}^G$ are uniformly discrete, we can set
\[ \Sigma_G(A):=\Sigma(A)/\widetilde{\cN}^G.\]
This is a $G$-integration, since the theorem shows that its restriction to every leaf is a $G$-integration.
\end{proof}

\begin{remark}
\label{rmk:smooth:vs:orbifold}
As we pointed out before, if $A$ is $G$-integrable and we let $\Sigma_G(A)$ be the canonical $G$-integration, in general, the quotients $\s^{-1}(x)/G$ are orbifolds. The previous discussion shows that this quotient is a smooth manifold if and only if we have $\tilde{\iota}^{-1}(\widetilde{\cN}^G_x)=\pi_1(G)$. Indeed, this is exactly the condition that the induced morphism $\iota: L\rtimes G \to \Sigma_G(A)|_L$ is injective, so the $G$-action on $\s^{-1}(x)$ is free.
\end{remark}

\begin{remark}
It is shown in  \cite{CrainicFernandes} that the canonical integration $\Sigma(A)$ always exists as a topological groupoid, even when $A$ is non-integrable. In this case, the morphism $\tilde{\iota}:X\rtimes\tilde{G}\to \Sigma(A)$ still exists, so one can define the extended and restricted $G$-monodromy groups, even when $A$ is not an integrable Lie algebroid. Since one has always $\cN_x\subset \cN_x^G$, the integrability assumption in Corollary \ref{cor:G:integrability} can be dropped.
\end{remark}

%%%%%%%%%%%%%%%%%%%%%%%%%%%%%%
\subsection{Computing the obstructions to $G$-integrability}
%%%%%%%%%%%%%%%%%%%%%%%%%%%%%%

For concrete applications it is useful to have some tool to compute the $G$-monodromy groups from the infinitesimal data, i.e., the $G$-principal algebroid. We discuss now such a tool, which is similar to the method given in \cite{CrainicFernandes} to compute the usual monodromy groups. We start by recalling it. 

Given a Lie algebroid $A\to X$, consider the short exact sequence of Lie algebroids corresponding to a leaf $L$ of $A$
\begin{equation}
\label{eq:short:seq}
\xymatrix{
0 \ar[r] & \mathrm{Ker}\rho|_L \ar[r] & A|_L \ar[r]^{\rho}& TL \ar[r] & 0.}
\end{equation}
Given a splitting $\sigma: TL \to A|_L$ of this sequence, its curvature is the $2$-form $\Omega_\sigma\in\Omega^2(L,\ker\rho|_L)$ defined by
\[\Omega_{\sigma}(X,Y) = \sigma([X,Y]_L) - [\sigma(X), \sigma(Y)]. \]
The splitting also defines a connection $\nabla^\sigma$ on the the bundle $\ker\rho|_L\to L$ by
\[ \nabla^\sigma_X \xi=[\sigma(X),\xi],\]
and its curvature is related to the curvature 2-form of the splitting by
\[ R_{\nabla^\sigma}(X,Y)\xi=[\Omega(X,Y),\xi]. \]
Hence, if the curvature of the splitting takes values in the center $Z(\ker\rho|_x)$, this connection is flat, and we can integrate forms with values in $\ker\rho|_x$.

\begin{prop}[Lemma 3.6 in \cite{CrainicFernandes}]
\label{prop:monodromy:comput}
Let $A\to X$ be a Lie algebroid and fix a leaf $L\subset X$. If there exists a splitting $\sigma: TL \to A|_L$ whose curvature $2$-form $\Omega_\sigma$ takes values in the center $Z(\ker\rho|_L)$, then $\exp(\cN_x)=\widetilde{\cN}_x$ and
\[ \cN_x=\left\{\int_\gamma \Omega_\sigma:[\gamma]\in\pi_2(L,x)\right\}\subset Z(\ker\rho|_x), \]
\end{prop}

The assumption in the proposition is verified in many cases and it is generic since it holds for regular leaves of $A$.

Now, let $A$ be a $G$-principal algebroid, fix $x\in X$ and consider again the short exact sequence of Lie algebroids \eqref{eq:short:seq} associated with the leaf $L$ through $x$.

\begin{definition}
\label{def:G:splitting}
A splitting $\sigma: TL \to A|_L$ of \eqref{eq:short:seq} is called a {\bf $G$-splitting} if:
\begin{enumerate}[(a)]
\item its curvature $2$-form is center-valued
\[\Omega_{\sigma}(X,Y) \in Z(\ker\rho|_L). \]
\item along the $G$-orbit through $x$, the splitting restricts to a splitting of the action algebroid
\[\sigma \circ \psi (y,\al) = i(y,\al),\quad \al\in\gg,\, y \in x \cdot G.\]
\end{enumerate}
\end{definition}

Again, existence of a $G$-splitting is a generic condition since the regular leaves of a  $G$-principal algebroid always admit $G$-splittings.  

%The elements in the $G$-monodromy group $\cN^G_x$ are obtained by integrating the curvature of a $G$-splitting over 2-disks lying in the leaf $L$ through $x$ and whose boundary lies in the $G$-orbit through $x$

\begin{prop}
\label{prop:G:monodromy:comput}
Let $A\to X$ be a $G$-principal algebroid and fix $x\in X$. If the action is locally free at $x$ and the leaf $L\subset X$ admits a $G$-splitting $\sigma: TL \to A|_L$ then
\[ \cN^G_x=\left\{\int_\gamma \Omega_\sigma~|~ \gamma:D^2\to L\text{ with } \gamma|_{\partial D^2}\subset  x\cdot G\right\}\subset Z(\Ker\rho|_x). \]
\end{prop}

\begin{proof}
The proof follows the same pattern as the proof of Proposition \ref{prop:monodromy:comput} given in \cite{CrainicFernandes}. The splitting allows to identify the vector bundle $A|_L$ with $TL\oplus \mathrm{Ker} \rho|_L$, so that the anchor $\rho$ becomes projection in $TL$ and the Lie bracket is given by
\[[(X,\xi),(Y, \zeta)] = ([X,Y], [\xi,\zeta] +\nabla^\sigma_X\zeta - \nabla^\sigma_Y\xi-\Omega_\sigma(X,Y)).\]
Moreover, if we choose a connection $\nabla^L$ on $TL$ we obtain a connection $\nabla = (\nabla^L, \nabla^\sigma)$ on $A|_L$ whose torsion is
\[T_\nabla((X,\xi), (Y,\zeta)) = (T_{\nabla^L}(X,Y), \Omega_{\sigma}((X,Y)) - [\xi, \zeta]).\]
Then any $A$-homotopy in $L$ takes the form
\[a(\epsilon, t)\d t + b(\epsilon, t)\d\epsilon \text{ with } a = (\frac{\d\gamma}{\d t}, \phi), b = (\frac{\d\gamma}{\d\epsilon}, \eta),\]
where $\phi$ and $\eta$ are variations of paths in $\mathrm{Ker}\rho|_L$ satisfying
\[\partial_t\eta - \partial_\epsilon\phi = \Omega_\sigma\left(\frac{\d\gamma}{\d t},\frac{\d\gamma}{\d\epsilon}\right) - [\phi, \eta]. \]

Now denote by $\tilde{\iota}:X\rtimes\tilde{G}\to \Sigma(A)$ be the integration of $i:X\rtimes\gg\to A$. Let $\tilde{g}:I\to G$ be a loop defining an element $[\tilde{g}]\in\pi_1(G)$. We claim that if the loop $\gamma_0(t)=x\tilde{g}(t)$ bounds a disk $\gamma:D^2\to L$ then
\[ \exp\left(\int_\gamma \Omega_\sigma\right)=\tilde{\iota}(x,[\tilde{g}]). \]

To prove this note that, since the action is locally free at $x$, condition (b) in Definition \ref{def:G:splitting}, shows that the $A$-loop $a_0(t):=(\frac{\d }{\d t}\gamma_0,0)$ represents the element 
\[ [a_0]=\tilde{\iota}(x,[\tilde{g}])\in\Sigma(A)_x.\]

On the the other hand, we can define an $A$-homotopy $a \d t + b\d\epsilon$ over $\gamma$ by setting $\eta=0$ and defining
\[ \phi=-\int_0^\eps \Omega_\sigma\left(\frac{\d\gamma}{\d t},\frac{\d\gamma}{\d\epsilon}\right). \]
This gives an $A$-homotopy starting at $a_0$ and ending at the $A$-loop
\[ a_1(t)=\left(0,t \int_\gamma \Omega_\sigma\right). \]
so the claim follows.

To finish the proof of the proposition we observe that if $[\tilde{g}]\in\pi_1(G)$ is such that $\tilde{\iota}(x,[\tilde{g}])\in Z(\Sigma(A)_x)^0$ then the loop $\gamma_0(t)=x\tilde{g}(t)$ bounds some disk $\gamma:D^2\to L$. This follows because $\tilde{\iota}(x,[\tilde{g}])$ is represented by the $A$-loop $a_0(t):=(\frac{\d }{\d t}\gamma_0,0)$ with base path $\gamma_0$, and the elements in $\Sigma(A)_x^0$ are precisely the $A$-loops whose base path is a loop based at $x$ contractible in $L$.
\end{proof}

\begin{remark}
Under the conditions of the proposition, if one considers two disks $\gamma_i:D^2\to L$, $i=1,2$, with the same boundary $\partial \gamma_1=\partial \gamma_2\subset x\cdot G$, it follows from Proposition \ref{prop:monodromy:comput} that
\[ \int_{\gamma_2} \Omega_\sigma-\int_{\gamma_1} \Omega_\sigma\in \cN_x. \]
\end{remark}

\begin{remark}
There are more general situations, where the action is not locally free at $x$ and one can still compute the $G$-monodromy as in the proposition. For example, if $G$ is compact then $\gg=[\gg,\gg]\oplus Z(\gg)$ and $\pi_1(G)/\pi_1(Z(G)^0)$ is finite. Hence, the $G$-monodromy at $x$ will be discrete if and only if the subgroup
\[ \tilde{\iota}(x,\pi_1(Z(G)^0))\subset \Sigma(A)_x \] 
is discrete. Applying the proposition to the restriction of the $G$-action to the $Z(G)^0$-action, we conclude that if the restriction $\psi|_{\{x\}\times Z(\gg)}$ is injective then the $G$-monodromy is discrete if and only if the subgroup
\[ \left\{\int_\gamma \Omega_\sigma~|~ \gamma:D^2\to L\text{ with } \gamma|_{\partial D^2}\subset  Z(G)^0\cdot x\right\}\subset Z(\Ker\rho|_x), \]
is a discrete subgroup.
\end{remark}

%%%%%%%%%%%%%%%%%%%%%%%%%%%%%%%%
%%%%%%%%%%%%%%%%%%%%%%%%%%%%%%%%
\section{Solutions to Cartan's Realization Problem}
\label{sec:global:solutions}
%%%%%%%%%%%%%%%%%%%%%%%%%%%%%%%%
%%%%%%%%%%%%%%%%%%%%%%%%%%%%%%%%

We will now give a full account of the problem of existence of solutions of a Cartan's realization problem. We first discuss the existence of local solutions and then complete solutions, in a sense that will be made precise.

%%%%%%%%%%%%%%%%%%%%%%%%%%%%%%%%
%%%%%%%%%%%%%%%%%%%%%%%%%%%%%%%%
\subsection{Local solutions}
\label{sec:local:solutions}
%%%%%%%%%%%%%%%%%%%%%%%%%%%%%%%%
%%%%%%%%%%%%%%%%%%%%%%%%%%%%%%%%

We have seen that $G$-integrations of a $G$-structure algebroid give rise to solutions of Cartan's realization problem. However, not every $G$-structure algebroid is $G$-integrable so, a priori, it is not even clear if local solutions of a Cartan's realization problem exist. In this section, we will prove the following result.

\begin{theorem}\label{thm:local:G:int}
Let $A \to X$ be a transitive $G$-principal algebroid and let $x \in X$. Then there exists a $G$-saturated open neighborhood $U \subset X$ of $x$ such that $A|_U$ is $G$-integrable.
\end{theorem}

As a corollary, we obtain that local solutions always exists, i.e., we have the following converse to Theorem \ref{thm:Cartan:G:alg}.

\begin{corol}\label{cor:local:existence:2}
Let $A \to X$ be a $G$-structure algebroid with connection and $x \in X$. Then there exists a $G$-realization $(P, (\theta, \omega), h)$ of $A$ such that $x$ belongs to the image of $h$.
\end{corol}

\begin{proof}
The theorem shows that we can choose a $G$-saturated open neighborhood $x\in U \subset L$ in the leaf of $A$ through $x$, for which $A|_U$ is $G$-integrable. If $\G \tto U$ is a $G$-integration of $A|_U$ then, by Corollary \ref{cor:local:existence}, $(\s^{-1}(x), \wmc|_{\s^{-1}(x)}, \t)$ is a $G$-realization of $A|_U$, and therefore also of $A$, with $x =\t(1_x)$.
\end{proof}

The rest of this section is dedicated to the proof of Theorem \ref{thm:local:G:int}. For that reason, we will assume that $A\to X$ is a \emph{transitive} Lie algebroid.

Recall that a Lie subalgebroid $B\to Y$ of a Lie algebroid $A\to X$ is called \emph{wide} if $Y=X$. The following result should be well-known, but we could not find a proof in the literature, so we give one such proof here.

\begin{prop}
\label{prop:integrable:wide:subalg} 
Let $A \to X$ be a transitive Lie algebroid and let $B \subset A$ be a wide transitive Lie subalgebroid. Then $A$ is integrable if and only if $B$ is integrable.
\end{prop}

\begin{proof}
Let $i:B\hookrightarrow A$ be a wide transitive Lie subalgebroid. Then a splitting of $\rho^B:B\to TM$ determines a splitting of $\rho^A:A\to TM$, so it follows from the definition of the monodromy homomorphism (see \cite{CrainicFernandes}) that the monodromy maps of $A$ and $B$ fit into the commutative diagram
\[\xymatrix{& \mathcal{G}(\ker \rho^B_x) \ar[dd]^{i_*}\\
\pi_2(X,x)\ar[ur]^{\partial^B_x} \ar[dr]_{\partial^A_x}&\\
&\mathcal{G}(\ker \rho^A_x).}\]
Since the (extended) monodromy groups are defined precisely as the images of the maps $\partial_x$ we see that $\tilde{\mathcal{N}}^A_x =i_*(\tilde{\mathcal{N}}^B_x)$. In general, the morphism $i_*$ does not map centers to centers and it does not have closed image. Still, observing that the exponential maps fit into the commutative diagram
\[\xymatrix{
\mathcal{G}(\ker \rho^B_x) \ar[r]^{i_*} & \mathcal{G}(\ker \rho^A_x) \\
\ker \rho^B_x\ar[u]^\exp \ar[r]_{i} & \ker \rho^A_x\ar[u]_\exp} \]
we obtain that the (restricted) monodromy groups are also related via the inclusion in some neighborhood $U\subset \ker \rho^A_x$ of zero
\[ \mathcal{N}^A_x\cap U =i(\mathcal{N}^B_x)\cap U.\] 
This shows that one is discrete iff the other is, therefore $A$ is integrable if and only if $B$ is integrable.

\end{proof}

For the statement of the next proposition, observe that a transitive Lie algebroid $A\to X$ can always be restricted to a submanifold $Y\subset X$.

\begin{prop}\label{prop:G-orbit:int}
Let $A \to X$ be a $G$-principal algebroid and let $x\in X$. Then the restriction $A|_{x\cdot G} \to x \cdot G$ of $A$ to the $G$-orbit through $x$ is $G$-integrable. 
\end{prop}

%\begin{remark}
%In general, the Lie algebroid $A|_{x \cdot G}$ will \emph{not} be a $G$-structure algebroid since its fibers may fail to be isomorphic to $\Rr^n\oplus \gg$. In any case,we still have the inclusion map $i: \gg \ltimes (x \cdot G) \to A|_{x \cdot G}$ , which is a Lie algebroid morphism determining an action of $G$ on $A|_{x \cdot G}$ by inner automorphisms. Hence, the definition and obstructions to $G$-integrability extend in the obvious way to this setting. This is the meaning of ``$G$-integrable'' in the statement of the proposition above. 
%\end{remark}

\begin{proof}
To simplify the notation, throughout this proof we assume that $X = x \cdot G$ and denote $A|_{x\cdot G}$ simply by $A$. Applying Proposition \ref{prop:integrable:wide:subalg} to the wide subalgebroid 
\[i: (x\cdot G)\rtimes \gg \to A\]
we obtain that $A$ is integrable. All that is left to prove is that the $G$-monodromy group $\mathcal{N}^G_x(A)$ is discrete so that $A$ is $G$-integrable.

First, recall that the extended $G$-monodromy group is $\tilde{\mathcal{N}}^G_x(A)=\tilde{\iota}(x,\pi_1(G))$, the image of $\pi_1(G)$ under the morphisms $\tilde{\iota}:(x\cdot G)\rtimes \tilde{G}\to\Sigma(A)$ that integrates $i$. Moreover, the restricted monodromy group at $x$ is defined via the exponential map
\[
\exp(\cN^G_{x}(A))=\widetilde{\cN}^G_{x}(A)\cap Z(\Sigma(A)_x)^0,
\]
where $Z(\Sigma(A)_x)^0\subset \Sigma(A)^0_x$, the identity component of the isotropy group of $\Sigma(A)$ at $x$. Since $\Sigma(A)_x^0$ consists of the $A$-homotopy classes of $A$-loops whose base loop is contractible, we have
\[\tilde{\mathcal{N}}^G_x(A) \cap \Sigma(A)_x^0 = \{\tilde{\iota}(x,[\gamma]): [\gamma] \in \pi_1(G) \text{ with } x\cdot \gamma \text{ contractible}\}.\]
Hence, the long exact sequence of the fibration 
\[ \xymatrix{G_x \ar[r] & G \ar[r]& x\cdot G} \] 
yields
\[
\exp(\cN^G_{x}(A))=\widetilde{\cN}^G_{x}(A)\cap Z(\Sigma(A)_x)^0 = \tilde{\iota}(x,\pi_1(G_x)).
\]
Therefore, to prove that $\mathcal{N}^G_x(A)$ is discrete it is enough to prove that $\tilde{\iota}(x,\pi_1(G_x))\subset Z(\Sigma(A)_x)^0$ is a closed subgroup.

Recall now that the restricted monodromy group is precisely the subgroup of the center $\cN_x(A)\subset Z(\ker\rho_x)$ such that
\[ Z(\Sigma(A)_x)^0=Z(\cG(\ker\rho_x))^0/\exp(\cN_x(A)), \]
and similary for the action groupoid
\[ Z(\tilde{G}_x)^0=Z(\cG(\gg_x))^0/\exp(\cN_x((x\cdot G)\rtimes\gg)), \]
The morphism $i:(x\cdot G)\rtimes\gg\to A$ maps monodromy isomorphically to monodromy. However, in general, it does not map $Z(\gg_x)$ to $Z(\ker\rho_x)$ so its integration $\tilde{\iota}$ does not map $Z(\tilde{G}_x)^0$ to  $Z(\Sigma(A)_x)^0$. Still, the subalgebra
\[ \aa= \{\al\in\gg_x:i(x,\al)\in Z(\ker\rho_x)\}\subset Z(\gg_x). \]
contains the monodromy group $\cN_x(\gg\ltimes (x\cdot G))$ (since $i$ is injective and this group is mapped to $Z(\ker\rho_x)$). Therefore, setting
\[ \mathcal{A}:=\cG(\aa)/\exp(\cN_x((x\cdot G)\rtimes\gg)), \]
we can restrict $\tilde{\iota}$ to a group morphism
\[ 
\xymatrix{
Z(\tilde{G}_x)^0 \ar[r]^{\tilde{\iota}} & \Sigma(A)_x^0 \\
\mathcal{A}\ \ar@{^{(}->}[u] \ar[r]_---{\tilde{\iota}}  & Z(\Sigma(A)_x)^0\ar@{^{(}->}[u] } 
\]
which is an embedding $\mathcal{A}\hookrightarrow Z(\Sigma(A)_x)^0$. Now, observe that $\pi_1(G_x)\subset \mathcal{A}$, since its image under $\tilde{\iota}$ lies in the center $Z(\Sigma(A)_x^0)$. Hence, we conclude that $\tilde{\iota}(x,\pi_1(G_x))\subset Z(\Sigma(A)_x)^0$ is a closed subgroup.
\end{proof}
 
 We can now use the previous proposition and the slice theorem for proper group actions to obtain $G$-integrability on a $G$-saturated open neighbourhood of $x \cdot G$.

\begin{proof}[Proof of Theorem \ref{thm:local:G:int}]
By the slice theorem for proper actions (\cite[Theorem 2.3.3]{DuistKolk}) there exists a $G$-saturated open neighbourhood $U$ of $x \cdot G$ in $X$ such that $x \cdot G$ is a deformation retract of $U$. We note that the isotropy Lie algebras of $A$ coincide with those of $A|_U$ and $A|_{x \cdot G}$. The (usual) restricted monodromy of  $A|_{x\cdot G}$ at $x$ is the image of
\[\partial_x: \pi_2(x\cdot G,x) \to Z(\ker \rho_x),\]
and the restricted monodromy of $A|_U$ is given by
\[\partial_x: \pi_2(U,x) \to Z(\ker \rho_x).\]
If we denote by $r: U \to x \cdot G$ the deformation retraction, then
\[\xymatrix{\pi_2(U_x) \ar[dd]_{r_*} \ar[dr]^{\partial_x} &\\
& Z(\ker \rho_x)\\
\pi_2(x \cdot G)  \ar[ur]^{\partial_x} &}\]
commutes, and since $r_*$ is an isomorphism we conclude that the restricted monodromy groups of $A|_U$ and $A|_{g \cdot x}$ coincide. By Proposition \ref{prop:G-orbit:int}, it follows that $A|_U$ is integrable.

We proceed to show that $A|_U$ is $G$-integrable. We first note that at the level of Lie algebroids we have the commutative diagram of Lie algebroid morphisms
\[\xymatrix{(x\cdot G)\rtimes\gg \ar[d] \ar[r]^i & A|_{x \cdot G}\ar[d]\\
U\rtimes\gg \ar[r]_i & A|_U}\]
Therefore, integrating to the source $1$-connected Lie groupoids we get 
\[\xymatrix{(x\cdot G)\rtimes \tilde{G} \ar[d] \ar[r]^{\tilde{\iota}} & \Sigma(A|_{x \cdot G})\ar[d]\\
U\rtimes \tilde{G} \ar[r]_{\tilde{\iota}} & \Sigma(A|_U)}\]
It follows that at $x$ we have
\[\xymatrix{& \Sigma(A|_{x \cdot G})_x\ar[dd]\\
 \{x\} \times \pi_1(G)\ar[ur]^{\tilde{\iota}} \ar[dr]_{\tilde{\iota}} &\\
& \Sigma(A|_{U})_x.}\]
Note that the (usual) monodromy groups of $A|_{x \cdot G}$ and $A|_U$ coincide at $x$, so the map $\Sigma(A|_{x \cdot G}) \to \Sigma(A|_{U})$ is an isomorphism. Since $A|_{x \cdot G}$ is $G$-integrable, we conclude that the $G$-monodromy of $A|_U$ at $x$ is discrete, showing that $A|_U$ is $G$-integrable. 
\end{proof}

\begin{remark}
The compactness of $G$ was used in the previous proof only to guarantee the existence of $G$-invariant neighborhoods of an orbit retracting to the orbit. For more general groups one needs an assumption on the action of $G$ on $X$ (e.g., properness) that guarantees the existence of such neighborhood.
\end{remark}

%%%%%%%%%%%%%%%%%%%%%%%%%%%%%%%%
\subsection{Complete solutions and complete $G$-Realizations}
\label{sec:complete:solutions}
%%%%%%%%%%%%%%%%%%%%%%%%%%%%%%%%
%%%%%%%%%%%%%%%%%%%%%%%%%%%%%%%%

We have established that local solutions to Cartan's realization problem always exist. We would like now to understand the existence of global solutions. In this paragraph we will make sense of the notion of \emph{complete solution}. 

Let $A \to X$ be a $G$-structure algebroid with connection. Recall that the classifying map of a $G$-realization $(P, (\theta, \omega), h)$ of $A$ is a submersion $h:P\to U$ where $U$ is an open subset of a leaf $L$ of $A$ (see Lemma \ref{lem:classifying:map}).

\begin{definition}
\label{def:full}
A $G$-realization $(P,(\theta, \omega),h)$ of a $G$-structure algebroid with connection $A\to X$
is called a \textbf{full $G$-realization} if the classifying map is a surjective submersion onto a leaf $h: P \to L$. 
\end{definition}

A full $G$-realization $(P, (\theta, \omega), h)$ covering a leaf $L$ comes equipped with an infinitesimal action of $A_L$ on $h:P\to L$ defined by
\[\sigma: h^*A_L \to TP, \quad (\theta,\omega)_p\sigma(p, \xi) = \xi,\quad \forall \xi \in A_{h(p)}.\]
Recall that such an action is called {\bf complete} if the vector field $\sigma(\xi_t)\in\X(P)$ is complete whenever $\xi_t \in \Gamma(A_L)$ is a time-dependent section for which $\rho(\xi_t)\in\X(L)$ is a complete vector field.

\begin{definition}
\label{def:complete}
A $G$-realization $(P, (\theta, \omega), h)$ of a $G$-structure algebroid with connection $A\to X$
is called a \textbf{complete $G$-realization} if:
\begin{enumerate}[(i)]
    \item $(P, (\theta, \omega), h)$ is a full $G$-realization covering a leaf $L\subset X$, and
    \item the infinitesimal $A_L$-action on $P$ is complete.
\end{enumerate}
The corresponding solution of the associated Cartan's realization problem is called a \textbf{complete solution}.
\end{definition}

Notice that this definition of complete does not appeal to any metric notion. However, in Section \ref{sec:geom:structures} we will see that when $G \subset \OO(n,\Rr)$ is a closed subgroup, then any $G$-realization induces a metric on $M = P/G$. If this metric is complete then $(P, (\theta,\omega), h)$ is a complete $G$-realization in the sense of Definition \ref{def:complete}. So complete metric solutions are complete solutions in our sense and the converse also holds when $G$ is compact.

One important motivation to define complete $G$-realizations is that source fibers of $G$-integrations are complete $G$-realizations.

\begin{example}
Let $A$ be a $G$-integrable $G$-structure algebroid with connection and let $\G \tto L$ be a $G$-integration of $A_L$. Then the $A_L$-action on $(\s^{-1}(x), \wmc|_{\s^{-1}(x)}, \t)$ is given by
\[\sigma(\gamma, \xi) = \d_{\t(\gamma)}R_\gamma(\xi).\] 
This action is complete: if $\xi=\xi_t \in \Gamma(A_L)$ is a time-dependent section, then $\sigma(\xi)\in\X(\s^{-1}(x))$ has flow
\[ \phi^t_{\sigma(\xi)}(\gamma)=\exp(t\xi)(\t(\gamma))\cdot \gamma, \]
where $\exp$ denotes the exponential map taking sections of $\Gamma(A)$ to bisections of $\G$ (see \cite[Section 4.4]{CrainicFernandes:lectures}). This of course just says that the Lie algebroid action of $A_L$ on $(\s^{-1}(x), \wmc|_{\s^{-1}(x)}, \t)$ integrates to the $\G$-action on $(\s^{-1}(x), \wmc|_{\s^{-1}(x)}, \t)$ by left translations. Therefore $(\s^{-1}(x), \wmc|_{\s^{-1}(x)}, \t)$ is a complete $G$-realization. 
\end{example}

However, there exist complete $G$-realizations which do not arise as source fibers of a $G$-integration. For example, any $G$-realization whose total space is compact is complete. The following example shows that a compact $G$-realization may fail to be isomorphic to a source fiber.

\begin{example}
\label{ex:flat:torus}
Consider the Lie $SO(n,\Rr)$-algebra $A=\Rr^n\oplus\so(n,\Rr)\to \{*\}$. discussed in Example \ref{ex:trivial:1}. It admits only one $\SO(n,\Rr)$-integration, namely the group semi-direct product $\G=\Rr^n\rtimes \SO(n,\Rr)\tto \{*\}$. This integration gives rise to a complete solution with $M=P/G=\Rr^n$. Another complete solution arises from taking the (oriented) orthogonal frame bundle of the flat torus $M=\Tt^n$, which is not a source fiber of a $\SO(n,\Rr)$-integration.
\end{example}

In the next section we will give a characterization of complete solutions arising as source fibers. For now, we prove that existence of complete solutions already requires a Lie algebroid to be $G$-integrable.

\begin{theorem}\label{thm:complete:int}
Let $A \to X$ be a $G$-structure algebroid with connection and let $(P, (\theta,\omega), h)$ be a complete $G$-realization of $A$ covering a leaf $L \subset X$. Then $A_L$ is $G$-integrable. 
\end{theorem}

\begin{proof}
Given a $G$-realization $(P,(\theta,\omega),h)$, denoting by $\tilde{M}$ the universal (orbifold) cover of $M=P/G$, we have a pullback diagram
\[
\xymatrix{
\tilde{P}\ar[d]\ar[r]^q& P\ar[d]\\
\tilde{M}\ar[r] & M
}
\]
This yields a $G$-realization $(\tilde{P},(q^*\theta,q^*\omega),q^*h)$ with $\tilde{M}=\tilde{P}/G$ 1-connected. Moreover, if $(P,(\theta,\omega),h)$ covers a leaf $L$ so does  $(\tilde{P},(q^*\theta,q^*\omega),q^*h)$, and if $(P,(\theta,\omega),h)$ is complete so is $(\tilde{P},(q^*\theta,q^*\omega),q^*h)$. So we can assume we are given a complete $G$-realization $(P, (\theta,\omega), h)$ with $M=P/G$ a 1-connected orbifold.

The infinitesimal action $\sigma: h^*A_L \to TP$ determines a unique Lie algebroid structure on $h^*A_L\to P$: its anchor is $\sigma$ and its Lie bracket is uniquely determined by $[h^*\xi_1, h^*\xi_2] = h^*[\xi_1,\xi_2]$ so that Leibniz holds. Note that $\sigma: h^*A_L \to TP$ is actually a Lie algebroid isomorphism with inverse $(\theta,\omega):TP\to h^*A_L$. It follows from \cite[Corollary 7]{CrainicFernandes:Poisson} that
\begin{enumerate}[(a)]
\item $A_L$ is an integrable Lie algebroid;
\item the infinitesimal action of $A_L$ on $P$ integrates to an action of $\Sigma(A_L)$ on $P$ with moment map $h: P \to L$;
\item there is a Lie groupoid isomorphism
\[\Phi:P\rtimes \Sigma(A_L) \simeq\Sigma(h^*A_L)  \to \Pi_1(P).\]
\end{enumerate}

By the definition of a $G$-realization, we have the commutative diagram of morphism of algebroids
\[ 
\xymatrix{
P\rtimes\gg\ar[r] \ar[d]_{h\times \text{id}} & TP\ar[d]^{(\theta,\omega)} \\
L\rtimes\gg\ar[r] & A_L}
\]
inducing another commutative diagram of morphism of algebroids
\[ 
\xymatrix{
& P\rtimes\gg \ar[dl]\ar[dr] &  \\
h^*A_L \ar[rr]_{\sigma} & & TP}
\]
covering $L\rtimes\gg\to A_L$. This integrates to a commutative diagram of groupoid morphims
\[\xymatrix{&P\rtimes\tilde{G} \ar[dr] \ar[dl] &\\
P\rtimes \Sigma(A_L) \ar[rr]_\Phi& & \Pi_1(P).}\]
covering the morphism $\tilde{\iota}:L\rtimes\tilde{G}\to \Sigma(A_L)$.

Since $M=P/G$ is 1-connected, the long exact sequence of the (stack) fundamental groups of the fibration $G \to P \to M$, shows that $\pi_1(P,p) \simeq \pi_1(G)$ for all $p\in P$. Therefore, the previous triangular diagram restricts to a commutative diagram of group isomorphisms
\[
\xymatrix{& \{p \} \times \pi_1(G)\ar[dr] \ar[dl] &\\
 \{p \} \times \tilde{\cN}_x^G \ar[rr]_\Phi & & \pi_1(P,p)}
\]
Since $\Phi:P\rtimes \Sigma(A_L)\to \Pi_1(P)$ is an isomorphism and $\pi_1(P,p)$ is discrete in $\Pi_1(P)$ we conclude that $\{p\}\times \tilde{\cN}_x^G$ is discrete in $P\rtimes \Sigma(A_L)$, and therefore $\tilde{\cN}_x^G$ is discrete in $\Sigma(A_L)$, showing that $A_L$ is $G$-integrable.
\end{proof}

%%%%%%%%%%%%%%%%%%%%%%%%%
\subsection{Strongly complete realizations}
%%%%%%%%%%%%%%%%%%%%%%%%%
We saw in the previous paragraph that source fibers of $G$-integrations give complete solutions/$G$-realizations, but not every complete solution arises as a source fiber. So we need now to address the following question.
\begin{itemize}
\item Which complete $G$-realizations are source fibers of a $G$-integration? 
\end{itemize}
Examining the proof of Theorem \ref{thm:complete:int} one finds the following result (compare with Corollary \ref{cor:integration:realizations}).

\begin{prop}
\label{prop:1-connected:complete}
Let $(P,(\theta,\omega),h)$ be complete $G$-realization  of a $G$-structure algebroid with connection $A$, for which $M=P/G$ simply connected. Then $(P,(\theta,\omega),h)$ is isomorphic to a source fiber of the canonical $G$-integration $\Sigma_G(A_L)$, where $L$ is the leaf covered by the realization.
\end{prop}

\begin{proof}
The morphism $\Phi:\Sigma(A_L)\ltimes P\to \Pi_1(P)$ in the proof of Theorem \ref{thm:complete:int} is an isomorphism. The proof shows that the isotropy bundle of  $P\rtimes \Sigma(A_L)$ is $P\times_M \tilde{\cN}(A)$, which is mapped by $\Phi$ to $\bigcup_{p\in P}\pi_1(P,p)$, the isotropy bundle of $\Pi_1(P)$. Hence, $\Phi$ descends to a Lie groupoid isomorphism
\[ P\rtimes  \Sigma_G(A)\simeq P\times P. \]
Restricting both sides to source fibers gives the isomorphism $\s^{-1}(x)\simeq P$. The rest follows from the Universal Property of the Maurer-Cartan form (Theorem \ref{thm:MC:global}).
\end{proof}

As shown by Example \ref{ex:flat:torus} this proposition is false if we remove the assumption of 1-connectedness. In order to see what happens in the more general case, let $\G \tto X$ be \emph{some} $G$-integration of a $G$-structure algebroid $A$, fix $x\in M$ and consider the  $G$-realization $(\s^{-1}(x), \wmc|_{\s^{-1}(x)}, \t|_{\s^{-1}(x)})$. Then $\pi: \s^{-1}(x) \to \s^{-1}(x)/G$ is a $G$-structure. Given any local equivalence $\phi: U \to V$ between two open sets of $M=\s^{-1}(x)/G$, its lift $\tilde{\phi}: \pi^{-1}(U) \to \pi^{-1}(V)$ is a diffeomorphism satisfying $\tilde{\phi}^*\wmc = \wmc$. It follows from the Global Universal Property of the Maurer-Cartan form (Theorem \ref{thm:MC:global}) that $\tilde{\phi} = R_{\gamma}|_{\pi^{-1}(U)}$, for some arrow $\gamma$ in the isotropy Lie group $\G_x$. Thus the local symmetry $\phi$ extends to a \emph{global symmetry}.

%\begin{remark}\label{rmk:s:strongly:complete}
%Let $A\to X$ be a transitive $G$-structure algebroid and let $\G \tto X$ be a $G$-integration of $A$. Consider the $G$-realization $(s^{-1}(x), \wmc|_{s^{-1}(x)}, t|_{s^{-1}(x)})$ for some $x \in X$ and denote the quotient $s^{-1}(x)/G$ by $M$ so that $\pi: s^{-1}(x) \to M$ is identified with a $G$-structure on $M$. If $\phi: U \to V$ is a local equivalence between two open sets of $M$, then by lifting $\phi$ to $\tilde{\phi}: \pi^{-1}(U) \to \pi^{-1}(V)$ we obtain a diffeomorphism that satisfies $\tilde{\phi}^*\wmc = \wmc$. It then follows from the universal property of the Maurer-Cartan form that $\tilde{\phi} = R_{\gamma}$ where $\gamma$ is some element of the isotropy Lie group $\G_x$. Therefor, $\phi$ extends to a global symmetry of $M$.
%\end{remark} 
% 

This shows that $G$-realizations arising as source fibers of $G$-integrations are examples of \emph{strongly complete} $G$-realizations in the sense of the following definition.

\begin{definition}
A full $G$-realization $(P, (\theta, \omega), h)$ of a $G$-structure algebroid $A$ with connection is called \textbf{strongly complete} if any local symmetry of $(P, (\theta, \omega), h)$ extends to a global symmetry. The corresponding solution of the associated Cartan's realization problem is said to be a \textbf{strongly complete solution}.
\end{definition}

We will show that the property of being strongly complete indeed characterizes the $G$-realizations arising from source fibers of $G$-integrations. First we observe that such $G$-realizations have symmetry groups which are Lie groups.

\begin{prop}\label{prop:symmetry:Lie:grp}
Let $(P,(\theta,\omega),h)$ be a strongly complete $G$-realization of a $G$-structure algebroid $A\to X$ with connection and assume that $G$ is connected. Then the symmetry group of $(P,(\theta,\omega),h)$ is a Lie group $H$ with Lie algebra isomorphic to $\ker \rho_x$, where $x \in X$ is any point in the image of $h:P\to X$. Moreover, the $H$-action on $P$ is proper, free and transitive on the fibers of $h$, and commutes with the $G$-action.
\end{prop}

%%% Try to remove the connecteness assumption.

\begin{proof} 
As usual, we will write $M = P/G$ and denote by $\pi: P \to M$ the quotient map. By assumption, any local equivalence $\phi: U \to V$ lifts to a bundle map $\tilde{\phi}: \pi^{-1}(U) \to \pi^{-1}(V)$ preserving both the coframe and the classifying map
\[ \tilde{\phi}^*(\theta, \omega) = (\theta, \omega),\quad h \circ \tilde{\phi} = h.\] 
In particular, forgetting about the structure group $G$, $\tilde{\phi}$ is a local equivalence of the \emph{coframe}. Since $G$ is assumed to be connected this defines a one-to-one correspondence between (local) symmetries of the $G$-realization and (local) symmetries of the coframe which preserve the classifying map $h$. Choosing $x \in X$ any point in the image of $h:P\to X$,  by \cite[Proposition 6.7]{FernandesStruchiner1}, the symmetry group of $(P,(\theta,\omega),h)$ is a Lie group $H$ with Lie algebra $\ker \rho_x$, and $h: P \to L$ becomes a principal $H$-bundle over the leaf containing $x$. 

To conclude the proof of the proposition, note that the $H$ action on $P$ is given by the lifts of diffeomorphisms of $M$. Each such lift, being an automorphism of the $G$-structure, is a $G$-equivariant map. It follows that the $H$-action and the $G$-action on $P$ commute. 

\end{proof}

We can now prove that the strongly complete $G$-realizations are the ones arising from source fibers of $G$-integrations. In particular, they are also complete.

\begin{theorem}\label{thm:strongly:complete}
Let $A \to X$ be a $G$-structure algebroid with connection. A full $G$-realization $(P,(\theta,\omega),h)$ of $A$ covering a leaf $L\subset X$ is strongly complete if and only if it is isomorphic to a $G$-realization
\[ (\s^{-1}(x), \wmc|_{\s^{-1}(X)}, \t|_{\s^{-1}(x)}) \] 
associated with some $G$-integration $\G \tto L$ of $A_L$.
\end{theorem}

\begin{proof}
We have observed already that a $G$-realization $(\s^{-1}(x), \wmc|_{\s^{-1}(X)}$, $\t|_{\s^{-1}(x)})$ arising from some $G$-integration is strongly complete.

Conversely, assume that $(P,(\theta,\omega),h)$ is a strongly complete $G$-realization of $A$ covering a leaf $L$. By Proposition \ref{prop:symmetry:Lie:grp}, its symmetry group $H$ is a Lie group acting in a proper and free fashion on $P$. Therefore, we can form the gauge groupoid 
\[ \G = (P \times P)/H \tto L.\] 
This is a Lie groupoid whose Lie algebroid is isomorphic to $A_L$. Moreover, since the pair groupoid $P \times P \tto P$ is a $G$-structure groupoid integrating $TP$, and the $H$-action on $P$ commutes with the $G$-action, it follows that $\G$ is a $G$-structure groupoid integrating the Atiyah $G$-structure algebroid 
\[ A(\G): = TP/H\to L .\] 
To finish the proof we show that the identification of $A(\G)$ with $A_L$ is an isomorphism of $G$-structure algebroids. This follows since this isomorphism is induced from the vector bundle map
\[\xymatrix{TP \ar[d] \ar[r]^-{(\theta,\omega)} & A_L\ar[d]\\
P \ar[r]_h & L}\]
which is a morphism of $G$-structure algebroids (by the definition of a $G$-realization). Therefore, $\G \tto L$ is a $G$-integration of $A_L$ and $(P,(\theta,\omega),h)$ is isomorphic to $(\s^{-1}(x), \wmc|_{\s^{-1}(X)}, \t|_{\s^{-1}(x)})$ for some (and hence any) $x \in L$.

\end{proof}

\begin{example}
We saw in Example \ref{ex:flat:torus} that the flat torus $\Tt^n$ gives a complete $G$-realization of the trivial $\mathrm{SO}(n,\Rr)$-Lie algebra $A = \Rr^n\rtimes \mathfrak{so}(n,\Rr)$ which does not arise from a source fiber of some $\SO(n,\Rr)$-integration of $A$. Hence, according to Theorem \ref{thm:strongly:complete}, this $G$-realization is not strongly complete. 

One can also see directly that this $G$-realization is not strongly complete: thinking of $\Tt^n$ as the quotient of $\Rr^n$ by the lattice $\mathbb{Z}^n$, one sees that the global symmetries of $\Tt^n$ are induced by isometries of $\Rr^n$ preserving the lattice $\mathbb{Z}^n$. However, there are many more local symmetries -- any isometry of $\Rr^n$ after restriction to a small enough open subset, yields a local isometry of $\Rr^n$ which descends to a local isometry of the flat torus.
\end{example}

We can summarize the previous results on global solutions of Cartan's realization problem as follows:
\begin{enumerate}[(i)]
\item A solution is strongly complete if and only if it covers a leaf $L$ and every local symmetry extends to a global symmetry;
\item A solution is complete if and only it is covered by a strongly complete solution;
\item Complete solutions covering a leaf $L$ exist if and only if $A|_L$ is $G$-integrable;
\item Strongly complete solutions arise as the source fibers of $G$-integrations.
\end{enumerate}
Theorem \ref{thm:strongly:complete} reduces the classification of strongly complete $G$-realizations to the classification of all $G$-integrations, but this may not be an easy task. 

% Except for very special cases (see, e.g., Example \ref{ex:curvature:2}) this is a very hard or even impossible task.

%%%%%%%%%%%%%%%%%%%%%%%%%%%%%%%%
%%%%%%%%%%%%%%%%%%%%%%%%%%%%%%%%
\section{Symmetries and moduli space of solutions}
\label{sec:moduli}
%%%%%%%%%%%%%%%%%%%%%%%%%%%%%%%%
%%%%%%%%%%%%%%%%%%%%%%%%%%%%%%%%

We turn to the problem of determining the symmetries of solutions and, more generally, describing the moduli space of solutions to a Cartan's realization problem. For both of these problems it is important to distinguish between germs of solutions, which are local solution with marked points, and complete solutions. The answers will differ drastically in each case. 

%%%%%%%%%%%%%%%%%%%%%%%%%%%%%%%%
\subsection{Moduli space of germs of solutions}
%%%%%%%%%%%%%%%%%%%%%%%%%%%%%%%%

Let $(G, X, c, R, F)$ be the Cartan Data of a Realization Problem. We have seen that for every $x \in X$ there exists a solution $(F_G(M), (\theta, \omega),h)$ with $x \in \Im h$ if and only if the Cartan Data determines a $G$-structure algebroid with connection $A \to X$. Assuming this is the case, we consider \textbf{marked $G$-realizations} of $A$, i.e, quadruples $(P,(\theta,\omega), h, m_0)$ where $(P,(\theta,\omega), h)$ is a $G$-realization of $A$, and $m_0 \in M = P/G$ is a marked point.

\begin{definition}
Let $A\to X$ be a $G$-structure algebroid with connection. Two marked $G$-realizations $(P_1,(\theta_1,\omega_1), h_1, m_1)$ and $(P_2,(\theta_2,\omega_2), h_2, m_2)$  of $A$ are said to \textbf{represent the same germ of $G$-realizations} if there exist open subsets $m_1 \in V_1 \subset M_1$, $m_2 \in V_2 \subset M_2$ and an isomorphism of $G$-realizations
\[\xymatrix{(P_1)|_{V_1} \ar[r]^{\tilde{\phi}} \ar[d] & (P_2)|_{V_2}\ar[d]\\
V_1 \ar[r]_\phi & V_2}\]
such that $\phi(m_1) = m_2$.
\end{definition}

The condition of representing the same germ determines an equivalence relation on the set of marked $G$-realizations of $A$ and an equivalence class is called a {\bf germ of $G$-realization}. Let us denote the set of such germs by
\[\mathfrak{G}(A) = \{\text{germs of $G$-realizations of $A$}\} \] 

The {\bf symmetry group of a germ} is defined as follows. One chooses a marked $G$-realization $(P,(\theta,\omega), h, m_0)$ representing the germ and considers all germs of diffeomorphisms $\phi:M\to M$ fixing $m_0$, and which are local symmetries of the $G$-realization. The set of such germs of diffeomorphisms furnished with composition of diffeomorphisms forms a group and, clearly, its isomorphism class does not depend on the choice of representative of the germ.

The following result describes the moduli stack of germs of $G$-realizations of a $G$-structure algebroid with connection as an orbit space, and in particular allows one to speak of nearby germs, families of germs, etc.

\begin{theorem}\label{thm:moduli:germs}
Let $A \to X$ be a $G$-structure algebroid with connection. Then:
\begin{enumerate}[(i)]
\item For each $x \in X$ there corresponds a germ of $G$-realization of $A$;
\item This correspondence induces a bijection between $X/G$ and $\mathfrak{G}(A)$;
\item The symmetry group of the germ corresponding to $[x] \in X/G$  is isomorphic to the isotropy group $G_x$.
\end{enumerate}
In particular, the action groupoid $X\rtimes G\tto X$ presents the moduli stack of germs of solutions.
\end{theorem}

\begin{proof}
Let $x\in X$ and denote by $L\subset X$ the leaf of $A$ containing $x$. By Theorem \ref{thm:local:G:int} we can choose a $G$-saturated open neighborhood $U_x \subset L$ of $x$ such that $A_{U_x}$ is $G$-integrable. The corresponding canonical $G$-integration $\Sigma_G(A_{U_x})$ gives us a marked $G$-realization 
\[  (\s^{-1}(x), \wmc|_{\s^{-1}(x)}, \t|_{\s^{-1}(x)}, [1_x]) \]
where $[1_x]\in \s^{-1}(x)/G$ is the projection of the identity arrow $1_x$. This gives a well-defined germ of $G$-realization at $x$, establishing a map
\[ \varphi:X\to \mathfrak{G}(A). \]
This proves (i). This map seems to depend on the choice of $G$-saturated open neighborhood $U_x$, but the rest of the proof will show that this is not the case.

Let $(P, (\theta, \omega), h, m_0)$ be a marked $G$-realization and choose $p_0 \in P$ such that $[p_0]  = m_0$. Also, let $x_0=h(p_0) $. By choosing a small enough open neighborhood $V \subset M = P/G$ of $m_0$ we may assume that $h|{P_V}$ takes values in $U_{x_0}$. Then $(P_V, \theta|_{P_V}, \omega|_{P_V}, h|_{P_V})$ is a $G$-realization of $A_{U_{x_0}}$ and by the Universal Property of the Maurer-Cartan form we obtain (possibly after restricting to even smaller open sets) an equivalence  
\[\phi: (P_V, \theta|_{P_V}, \omega|_{P_V}, h|_{P_V}) \to (\s^{-1}(x_0), \wmc|_{\s^{-1}(x_0)}, \t|_{\s^{-1}(x_0)})\]
which maps $p_0$ to $1_{x_0}$. This shows that $\varphi(x_0) = [(P, (\theta, \omega), h, m_0)]$,
so the map $\varphi$ is surjective. Incidentally, it also show that $\varphi(x)$ is independent of the choice of $G$-saturated open neighborhood $U_x$.

Next, assume that $\varphi(x) = \varphi(y)$. This means that after restricting to $G$-saturated open neighborhoods of $1_x \in \s^{-1}(x)$ and $1_y \in \s^{-1}(y)$ we obtain an equivalence of marked $G$-realizations
\[\phi: (\s^{-1}(x), \wmc|_{\s^{-1}(x)}, \t|_{\s^{-1}(x)}, [1_x]) \to (\s^{-1}(y), \wmc|_{\s^{-1}(y)}, \t|_{\s^{-1}(y)}, [1_y]).\]
Since $\phi$ is an equivalence which respects the marked points we have that $\phi(1_x) = 1_y \cdot g$, for a unique $g \in G$, and that 
$\t|_{\s^{-1}(y)} \circ \phi =  \t|_{\s^{-1}(x)}$. Thus,
\[x = \t(1_x) = \t(1_y \cdot g) = \t(1_y)g = yg.\]
proving that $x$ and $y$ are in the same orbit. 

To complete the proof of (ii), we need to show that $\varphi(xg) = \varphi(x)$. Since $xg\in U_x$, the germs $\varphi(xg)$ and $\varphi(x)$ are both represented by source fibers of the same $G$-integration $\Sigma_G(A_{U_x})$.  If $\iota: U_x\rtimes G \to \Sigma_G(A_{U_x})$ is the groupoid morphism integrating $i:U_x\rtimes \gg\to A_{U_x}$,  then right translation $R_{\iota(x,g)}$ determines an equivalence of the marked $G$-realizations representing $\varphi(x)$ and $\varphi(xg)$. Thus $\varphi(xg) = \varphi(x)$. 

Finally, to prove (iii), it is enough to observe that any symmetry of the $G$-realization $(\s^{-1}(x), \wmc|_{\s^{-1}(x)}, \t|_{\s^{-1}(x)}, [1_x])$ is given by a right translation $R_{\gamma}$, where the arrow belongs to the isotropy group $\gamma \in \Sigma_G(A_{U_x})_x$ and is such that $R_\gamma(1_x) = 1_x \cdot g$. This means that $g \in G_x$ and $\gamma = \iota(x,g^{-1})$. Hence, the symmetries of $[x]\in \mathfrak{G}(A)$ are in bijection with the elements $g\in G_x$.
\end{proof}

Notice that the leaves of $A$ in $X$ do not play any role in the description of the moduli space of germs of $G$-realizations. However, they provide information on the \emph{global nature} of solutions of the realization problem. This will be thoroughly explored in the next section when we describe the moduli space of complete $G$-realizations for which $P/G$ is $1$-connected. Before we do so, let us explain how the leaves of $A$ are related to the `gluing' of germs of $G$-realizations. This is intimately related to the notion of \textit{analytically connected germs} described in \cite{Bryant}.

\begin{theorem}\label{thm:globalization} 
Let $A \to X$ be a $G$-structure algebroid with connection, and $x_1,x_2 \in X$. If there exists a connected $G$-realization $(P, (\theta, \omega), h)$ of $A$ and $m_1, m_2 \in P/G$ such that $[x_1] \in \mathfrak{G}(A)$ is represented by the germ of $(P, (\theta, \omega), h,m_1)$ and $[x_2] \in \mathfrak{G}(A)$ is represented by the germ of $(P, (\theta, \omega), h, m_2)$, then $x_1$ and $x_2$ must belong to the same leaf $L \subset X$ of $A$.

Moreover, when $A|_L$ is $G$-integrable the converse also holds.
\end{theorem}

\begin{remark}
The example of extremal K\"ahler metrics to be discussed in Section \ref{sec:example:EK:surfaces} will show that, in general, the converse does not hold without the $G$-integrability assumption.
\end{remark}

\begin{proof}
Let $(P, (\theta, \omega), h)$ be a connected $G$-realization of $A$. We have seen already that the image of $h$ lies inside a leaf $L \subset X$ of $A$. If $x_1,x_2 \in X$ are such that there exists $m_1,m_2 \in P/G$ for which $[x_1] \in \mathfrak{G}(A)$ is isomorphic to the germ of the marked $G$-realization $(P, (\theta, \omega), h, m_1)$, and $[x_2] \in \mathfrak{G}(A)$ is isomorphic to the germ of the marked $G$-realization $(P, (\theta, \omega), h, m_2)$, then there exists $p_1,p_2 \in P$, with $[p_1] = m_1$, $[p_2] = m_2$ and $h(p_1) = x \cdot G_1, h(p_2) = x \cdot G_2$. It then follows that $h(p_1)$ and $h(p_2)$ belong to the common leaf $L$ of $A$ in $X$, and therefore also $x_1, x_2 \in L$.

For the converse, assume that $A|_L$ is $G$-integrable and consider the $G$-realization 
\[(\s^{-1}(x), \wmc|_{\s^{-1}(x)}, \t|_{\s^{-1}}(x))\] 
of $A$, where $\s, \t$ denote the source and target maps of the canonical $G$-integration $\Sigma_G(A|_L)$. If $y \in L$ is another point in the same leaf as $x$, and $\gamma \in \s^{-1}(x)$ is such that $\t(\gamma) = y$, then germs of $G$-realizations corresponding to $x$ and $y$ can be identified with the germs of the marked realizations
\[(\s^{-1}(x), \wmc|_{\s^{-1}(x)}, \t|_{\s^{-1}}(x), [1_x])\]
and
\[(\s^{-1}(x), \wmc|_{\s^{-1}(x)}, \t|_{\s^{-1}}(x), [\gamma])\]
respectively. This concludes the proof.
\end{proof}

%%%%%%%%%%%%%%%%%%%%%%%%%%%%%%%%%%%
\subsection{Moduli space of 1-connected solutions}
%%%%%%%%%%%%%%%%%%%%%%%%%%%%%%%%%%%

We now consider complete $G$-realizations $(P,(\theta,\omega), h)$ of a $G$-structure algebroid $A \to X$ with connection for which $M=P/G$ is 1-connected, which we call a {\bf 1-connected solution}. The associated moduli space is
\[  \mathfrak{C}(A):=\{\text{isomorphism classes of 1-connected solutions}\}. \]

Recall that if a complete $G$-realization of $A$ covers a leaf $L\subset X$ then $A_L$ is $G$-integrable (cf.~Theorem \ref{thm:complete:int}). Hence, if denote by $X_\reg\subset X$ the saturated subset formed by the leaves $L\subset X$ such that $A_L$ is $G$-integrable, we can define a groupoid over $X_\reg$ by
\[ \Sigma_G^\reg(A):=\bigcup_{L\subset X_\reg} \Sigma_G(A_L). \]
In general, $X_\reg$ is a laminated set and $\Sigma_G^\reg(A) \tto X_\reg$ is a laminated groupoid -- it is smooth along the leaves, but may have a very complicated behavior in directions transverse to the leaves. It can happen, for example, that $A$ is $G$-integrable over a leaf $L_0$, but any saturated neighborhood of $L_0$ contains leaves $L$ such that $A_L$ is not $G$-integrable. The best possible scenario occurs when $A\to X$ is $G$-integrable for then $X_\reg=X$ and $\Sigma_G(A)$ is a Lie groupoid. 

The following result describes the moduli space of 1-connected solutions.

\begin{theorem}
\label{thm:moduli:1-connected}
Let $A\to X$ be a $G$-structure algebroid with connection. The groupoid $\Sigma^\reg_G(A) \tto X_\reg$ represents the moduli space  $\mathfrak{C}(A)$ of $1$-connected solutions to the associated Cartan problem. 
\end{theorem}

\begin{proof}
We have seen in Proposition \ref{prop:1-connected:complete} that  that every $1$-connected complete $G$-realization $(P, (\theta, \omega), h)$ of $A$ covering a leaf $L$  is isomorphic to the $G$-realization determined by a source fiber of the canonical $\Sigma_G(A_L)$. By the Universal Property of the Maurer-Cartan form, it follows that the group of symmetries of a $1$-connected complete $G$-realization corresponding to a source fiber $\s^{-1}(x)$ is isomorphic to the isotropy group $\Sigma_G(A)_x$. More generally, the set of isomorphisms between two $1$-connected complete $G$-realizations corresponding to $s$-fibers over two points $x,y\in L$, is in bijection with the set of arrows from $x$ to $y$
\[ \Sigma_G(A)(x,y)=\t^{-1}(y)\cap \s^{-1}(x),\]
so the result follows.
\end{proof}

In particular, when $A$ is $G$-integrable the moduli space  $\mathfrak{C}(A)$ is a geometric stack presented by the Lie groupoid $\Sigma_G(A)$. So it makes sense to talk about smooth families of 1-connected solutions, smooth deformations of 1-connected solutions, etc.

%%%%%%%%%%%%%%%%%%%%%%%%%%%%%%%%%%%
\subsection{Comparing the moduli spaces  $\mathfrak{C}(A)$ and $\mathfrak{G}(A)$}\label{sec:morita} 
%%%%%%%%%%%%%%%%%%%%%%%%%%%%%%%%%%%

Let us compare the two moduli spaces above when
$A$ is assumed to be a $G$-integrable $G$-structure algebroid. 

From the very definition of a $G$-structure groupoid, we have a morphism of Lie groupoids
\[
\xymatrix{
X\rtimes G\ar[r]^{\iota}  \ar@<0.25pc>[d] \ar@<-0.25pc>[d]  & \Sigma_G(A)\ar@<0.25pc>[d] \ar@<-0.25pc>[d]  \\
X\ar@{=}[r] & X
}
\]
giving a map between the associated stacks
\[  \mathfrak{G}(A)\to  \mathfrak{C}(A). \]

This map has a clear geometric meaning: 
\begin{itemize}
\item when $A$ is $G$-integrable every (isomorphism class of a) germ of solution extends to a unique (isomorphism class of a) 1-connected solution, and
\item every equivalence of germs of solutions is the restriction of an equivalence between the corresponding 1-connected solutions.
\end{itemize}
This map is surjective, but it is not injective; it is possible for two non-equivalent germs to extend to isomorphic 1-connected solutions. The reason is, of course, that if we start with a 1-connected solution and we consider the germs obtained by choosing two different points, in general, there is no self-equivalence taking one germ to the other (compare with Theorem \ref{thm:globalization}).

When the $G$-action on $X$ is free we can describe ``smaller" models for the moduli spaces that we have introduced before. On the one hand, for  the moduli space of germs of solutions $\mathfrak{G}(A)$ we have the following proposition.

\begin{prop}
Let $A$ be a $G$-structure algebroid with connection. If the action of $G$ on $X$ is free then the moduli space $\mathfrak{G}(A)$ of germs of $G$-realizations is smooth, namely the manifold $X/G$.
\end{prop}

\begin{proof}
In terms of stacks, this is the well-known statement that when the action is free the action groupoid $X\rtimes G\tto X$ and the identity groupoid $X/G\tto X/G$ are Morita equivalent. Namely, the quotient groupoid morphism
\[ X\rtimes G\to X/G, \quad (x,g)\mapsto [x], \]
is a Morita fibration.
\end{proof}

%map, or that we have a Morita bi-bundle:
%\[\xymatrix{
%X\rtimes G \ar@<0.25pc>[d] \ar@<-0.25pc>[d]  & \ar@(dl, ul) & 
%X\ar[dll]\ar[drr] & \ar@(dr, ur) & X/G
%\ar@<0.25pc>[d] \ar@<-0.25pc>[d]\\T
%X & & & & X/G}\]

On the other hand, the moduli space of 1-connected solutions also has a small model, but this time it is not a manifold in general, it is represented by a Lie groupoid  over $X/G$. To describe it we need the following lemma.

\begin{lemma}
If the $G$-action on $X$ is free then the morphism $\iota: X\rtimes G\to \Sigma_G(A)$ is injective and there is a free and proper action
\[ \Sigma_G(A) \times (G\times G)\to \Sigma_G(A), \quad \gamma\cdot (g,h) = \iota(\t(\gamma),g)\, \gamma\, \iota(\s(\gamma),h)^{-1}.\]
\end{lemma}

\begin{proof}
Since $\iota$ is a groupoid morphism covering the identity, its kernel is contained in the isotropy bundle of $G\ltimes X$. But if the $G$-action on $X$ is free, this isotropy bundle is trivial, so $\iota$ is injective. 

For the second part, observe that the anchor map 
\[ (\t,\s):\Sigma_G(A)\to X\times X\] 
is $G\times G$-equivariant. Since the $G\times G$-action on $X\times X$ is free and proper, so is the $G\times G$-action on $\Sigma_G(A)$.
\end{proof}

The promised small model is the content of the following proposition.

\begin{prop}
If the $G$-action on $X$ is free then 
\[ \G_G(A):=\Sigma_G(A)/(G\times G) \]
has a natural structure of a Lie groupoid over $X/G$.
\end{prop}

\begin{proof}
The structure of the Lie groupoid $\G_G(M) \tto X/G$ is the obvious one induced from $\Sigma_G(A) \tto X$. Its source, target, identity and inverse maps are defined by
\[\bar{\s}([\gamma]) = [\s(\gamma)], \quad \bar{\t}([\gamma]) = [\t(\gamma)], \quad 1_{[x]} = [1_x],\quad [\gamma]^{-1} = [\gamma^{-1}].\]
Perhaps less obvious is the definition of the multiplication. A pair $([\gamma_1],[\gamma_2])$ of arrows of $\G_G(A)$ is composable, by definition, if $\bar{s}([\gamma_1]) = \bar{t}([\gamma_2])$. In this case there exists a unique $g \in G$ such that $\s(\gamma_1) \cdot g= \t(\gamma_2)$. We define
\[[\gamma_1][\gamma_2] = [\gamma_1 \iota(\t(\gamma_2),g)\gamma_2].\]
We leave the verification that these maps are all well defined and smooth, so that  $\G_G(M) \tto X/G$ is a Lie groupoid.
\end{proof}

%We will now show that $\G_G(A) \tto X/G$ represents the moduli stack of $1$-connected complete $G$-realizations of $A$. We will do this by showing that $\G_G(A)$ is Morita equivalent to $\Sigma_G(A)$.Morita equivalent Lie groupoids represent the same quotient stack. We first recall the definition of a Morita equivalence between to Lie groupoids.
%
%Let $\mathcal{G}$ be a Lie groupoid over $M$ and $\mathcal{H}$ a Lie groupoid over $N$. Recall that a \textbf{Morita equivalence} between $\mathcal{G}$ and $\mathcal{H}$ 
%is given by a principal $\mathcal{H}$-$\mathcal{G}$ bi-bundle
%\[\xymatrix{
%  \mathcal{H} \ar@<0.25pc>[d] \ar@<-0.25pc>[d]  & \ar@(dl, ul) & 
%Q\ar[dll]^-{\nu}\ar[drr]_-{\mu} & \ar@(dr, ur) & \mathcal{G} 
%\ar@<0.25pc>[d] \ar@<-0.25pc>[d]\\
%N & & & & M,}\]
%i.e. a manifold $Q$ endowed with  
%\begin{enumerate}
%\item Surjective submersions $\nu: Q\to N$, $\mu: Q\to M$,
%\item A left action of $\mathcal{H}$ on $\nu: Q \to N$, which makes $\mu: Q\to M$ into a principal $\mathcal{H}$-bundle
%in the sense that the action map $\mathcal{H}\times_{N}Q\to Q\times_{M}Q$, $(h, q)\mapsto (hq, q)$, is a diffeomorphism,
%\item Similarly, a right action of $\mathcal{G}$ on $\mu: Q \to M$ which makes $\nu$ into a principal $\mathcal{G}$-bundle,
%\end{enumerate}
%and the left and right actions are required to commute. One says that $\mathcal{G}$ and $\mathcal{H}$ are Morita equivalent if such a bi-bundle exists.

Finally we observe that the groupoids $\Sigma_G(A)$ and $\G_G(A)$ represent indeed the same stack $\mathfrak{C}(A)$:

\begin{theorem}
Let $A$ be a $G$-integrable $G$-structure algebroid and assume that the action of $G$ on $X$ is free.Then the quotient map $\Phi:\Sigma_G(A) \to  \G_G(A)$ is a Morita fibration so the Lie groupoids $\Sigma_G(A)$ and $\G_G(A)$  are Morita equivalent and both present the stack $\mathfrak{C}(A)$ of complete 1-connected solutions.
\end{theorem}

\begin{proof}
In order to show that the morphism
\[
\xymatrix{
\Sigma_G(A) \ar[r]^{\Phi}  \ar@<0.25pc>[d] \ar@<-0.25pc>[d]  &\G_G(A)\ar@<0.25pc>[d] \ar@<-0.25pc>[d]  \\
X\ar[r]_{\phi} & X/G
}
\]
is a Morita fibration we need to verify that:
\begin{enumerate}[(i)]
\item the map 
\[ \hat{\Phi}:\Sigma_G(A) \to \G_G(A)  \tensor[_\s]{\times}{_\phi} X,\quad \gamma\mapsto  (\Phi(\gamma),\s(\gamma)), \]
is a surjective submersion;
\item the kernel of $\Phi$:
\[ \cK:=\{\gamma\in\Sigma_G(A):\Phi(\gamma)=1_x\}, \]
is isomorphic to the submersion groupoid $X\times_\phi X\tto X$.
\end{enumerate}
Item (i) is more or less straightforward. For item (ii) we observe that the groupoid morphism 
\[ \cK\to X\times_\phi X, \quad \gamma\mapsto (\t(\gamma),\s(\gamma)), \]
has inverse the groupoid morphism
\[ X\times_\phi X \mapsto \cK, \quad (y,x)\mapsto \iota(x,g), \]
where $g\in G$ is the unique element such that $y=xg$.
\end{proof}

When the $G$-action is free, the map between the moduli stacks
\[  \mathfrak{G}(A)\to  \mathfrak{C}(A), \]
has another representation in the smaller models above, namely by the obvious groupoid morphism given by the identity section
\[ X/G\to \G_G(A), \quad [x]\mapsto [1_x]. \]

\begin{remark}
When the $G$-action on X is free, the proof above shows that the kernel of the Morita fibration $\Phi:\Sigma_G(A) \to  \G_G(A)$ is the submersion groupoid $X\times_\phi X\tto X$, which in turn is isomorphic to the action groupoid $X\rtimes G\tto X$. Moreover, it is nor hard to see that $\Phi$ yields a Morita bi-bundle
\[\xymatrix{
  \Sigma_G(A) \ar@<0.25pc>[d] \ar@<-0.25pc>[d]  & \ar@(dl, ul) & 
\mathcal{M}\ar[dll]^-{\mu}\ar[drr]_-{\nu} & \ar@(dr, ur) & \G_G(A) 
\ar@<0.25pc>[d] \ar@<-0.25pc>[d]\\
X & & & & X/G}\]
where $\cM$ can be identified with the orbit space of $\Sigma_G(M)$,
\[ \mathcal{M}=\Sigma_G(A)/G, \] 
and $\mu$ (resp., $\nu$) is the map induced by the source (resp., target) of $\Sigma_G(A)$.

When the $G$-action on X is not free most of this fails: $\G_G(A)$ is no longer a groupoid (not even a set theoretic one!), the action groupoid has isotropy so it is not a submersion groupoid, and $\mathcal{M}$ is not a manifold. However, the action of $\Sigma_G(A)$ on $\mathcal{M}$ still makes sense, it is transitive along the fibers of the (continuous) map $\nu: \mathcal{M} \to X/G$, but it is no longer free. Therefore, the action of $\Sigma_G(A)$ on $\mathcal{M}$  fails to be principal, so 
the moduli space $\mathfrak{C}(A)$ cannot be represented by a groupoid over $X/G$, not even a topological groupoid.
%In fact, if we denote by
%\[I([p]) = \{\gamma \in \Sigma_G(A): [p]\cdot\gamma = [p]\}\]
%the set of elements of $\Sigma_G(A)$ which stabilise $[p]$, then
%\[I(p) = p^{-1}\iota(G_{t(p)}\times \{t(p)\})p.\]
%It follows that the action of $\Sigma_G(A)$ on $\mathcal{M}$ is free if and only if the action of $G$ on $X$ is free.
\end{remark}

%%%%%%%%%%%%%%%%%%%%%%%%%%%%%%%%
%%%%%%%%%%%%%%%%%%%%%%%%%%%%%%%%
\section{$G$-structure algebroids and geometric structures}
\label{sec:geom:structures}
%%%%%%%%%%%%%%%%%%%%%%%%%%%%%%%%
%%%%%%%%%%%%%%%%%%%%%%%%%%%%%%%%

Let $K$ be an element of the tensor algebra over $\Rr^n$. If $G\subset\GL(n,\Rr)$ is a subgroup of linear transformations leaving the tensor $K$ invariant, then it is a basic fact that a $G$-structure $\B_G(M)\to M$ determines a tensor field $K_M$ on $M$, which one views as the geometric structure determined by the $G$-structure. Properties of the $G$-structure are reflected on properties of the tensor field $K_M$. For example, the $G$-structure is locally flat if and only if there exist local coordinates on $M$ where the components of $K_M$ are constant; connections on the $G$-structure $\B_G(M)\to M$ are in bijection with linear connections $\nabla$ on $TM$ preserving $K_M$, i.e, such that $\nabla K_M=0$.

\begin{definition}
Given a tensor $K$ in $\Rr^n$, a $G$-structure algebroid is said to be of {\bf type $K$} if $G\subset\GL(n,\Rr)$ is a subgroup of linear transformations leaving the tensor $K$ invariant.
\end{definition}

When $A\to X$ is a $G$-structure algebroid with connection in canonical form, the tensor $K$ in $\Rr^n$ determines an $A$-tensor field via the canonical embeddings
\[ \Rr^n\hookrightarrow A_x=\Rr^n\oplus\gg. \]
We will denote the resulting tensor field by $K_A$.  Then, for any $G$-realization $(P,(\theta,\omega),h)$ of the associated Cartan's realization problem, the corresponding solution $M=P/G$ inherits a tensor $K_M$ of the same type and the linear connections $\nabla$ on $TM$ preserves $K_M$: $\nabla K_M=0$.

For the reminder of this section we will study some important classes of $G$-structure algebroids, and see how the extra geometric structure interacts with the Lie algebroid structure. We will focus on the following important cases:
\begin{enumerate}[1)]
\item metric type: $G\subset O(n,\Rr)$;
\item (almost) symplectic type: $G\subset \Sp(n,\Rr)$;
\item (almost) complex type: $G\subset \GL(n,\Cc)$.
\end{enumerate}
There are of course many other interesting types which can be dealt with similarly. Also, one may have combinations of more that one type, e.g. the (almost)-Hermitian type where $G\subset U(n)$ which combines all 3 types above.

For the reminder of this section we will assume that $A\to X$ is a $G$-structure algebroid with connection in canonical form.

%%%%%%%%%%%%%%%%%%%%%%%%%%%%%%%%%%%%%%
\subsection{Metric type}
%%%%%%%%%%%%%%%%%%%%%%%%%%%%%%%%%%%%%%

\begin{definition}
A  $G$-structure algebroid with connection is said to be of {\bf metric type} if $G\subset \OO(n,\Rr)$ and $c=0$.
\end{definition}

For a $G$-structure algebroid $A\to X$ with connection of metric type we have $\gg\subset\so(n,\Rr)$ and so the curvature $R$ satisfies
\begin{equation}
\label{eq:symmetry:Riemannian}
 \langle R(u,v)w,z\rangle_{\Rr^n}+\langle w,R(u,v)z\rangle_{\Rr^n}=0,
\end{equation}
where $\langle ~,~\rangle_{\Rr^n}$ is the usual euclidean metric. As explained before, we have a positive \emph{semi-definite} symmetric 2-tensor $K_A$ on $A$ given by 
\[
K_A((u,\al),(v,\beta))=\langle u,v\rangle_{\Rr^n},\quad (u,\al),(v,\be)\in A,
\]
Solutions of the Cartan's realization problem are Riemannian manifolds and the connection $\omega$ coincides with the Levi-Civita connection.

We can add to $K_A$ a term to obtain a positive definite symmetric 2-tensor $\tilde{K}_A$ on the Lie algebroid $A$
\begin{equation}
\label{eq:A:metric}
\tilde K_A((u,\al),(v,\beta)):=\langle u,v\rangle_{\Rr^n}+ \langle \al,\be\rangle_\gg\quad (u,\al),(v,\be)\in A,
\end{equation}
where $\langle~,~\rangle_\gg$ is the restriction of the Killing form of $\so(n,\Rr)$ to $\gg$. The advantage of this is that we can use this metric to induce a Riemannian metric on the leaves of $A$: for each $x\in X$ there is a unique metric on $T_xL$ that makes the map
\[ \rho: (\ker\rho|_x)^\perp \to T_xL, \]
an isometry. One can say that the Riemannian metric on $L$ is the unique one that makes $\rho:A\to TL$ an algebroid Riemannian submersion.  We will denote by $\tilde K_L$ the resulting metric on the  leaf $L$ and we will say that the leaf $L$ (respectively, the algebroid $A$) is {\bf metric complete} if the metric $\tilde K_L$ is complete (respectively, is complete for all leaves $L$).

From standard facts about $\OO(n,\Rr)$-structures, we have:

\begin{prop}
\label{prop:metric:Lie:G:grpd}.
Let $A\to X$ be a $G$-structure algebroid with connection  of metric type. Any $G$-realization $(P,(\theta,\omega),h)$ has a unique $G$-invariant metric $\tilde K_P$, which restricts to the Killing metric on the fibers of $P\to P/G=M$, and has the following properties:
\begin{enumerate}[(i)]
\item the induced metric $K_M$ on $M$ has Levi-Civita connection $\omega$ and curvature $R$;
\item the map $h:P\to L$ is a Riemannian submersion onto its image $h(P)\subset L$.
\item there is  a commutative diagram of Riemannian submersions
\[
\xymatrix{
P \ar[d]\ar[r]^h & L\ar[d]\\
M \ar[r] & L/G
}
\]
\end{enumerate}
\end{prop}

\begin{remark}
\label{rem:metric:L/G}
In general, the action of $G$ on $L$ is proper but not free (not even locally free!). So $L/G$ is only a smooth stratified space. However, it still has a metric distance induced from the metric in $L$, and it makes sense to talk about metric completeness. Also, it makes sense to talk about a Riemannian submersion $q:M\to L/G$, where by this one means that $q$ is an open map and that its fibers form a singular Riemannian foliation. 
\end{remark}

If $\G\tto X$ is a $G$-structure groupoid integrating a metric type $G$-structure algebroid $A\to X$, we can use right translations to propagate the fiberwise metric $\tilde K_A$ to a right-invariant metric on the source fibers of $\G\tto X$, which we will denote by $\tilde K_\G$. If $x\in X$ and $L$ is the leaf containing $L$, then we have a commutative diagram of Riemannian submersions
\[
\xymatrix{
\s^{-1}(x) \ar[d]\ar[r] & \s^{-1}(x)/\G_x=L\ar[d]\\
M=\s^{-1}(x)/G \ar[r] & M/\G_x=L/G
}
\]
Note that this is just the diagram of Proposition \ref{prop:metric:Lie:G:grpd} for the special case of the $G$-realization $(\s^{-1}(x),\wmc,\t)$. If the $G$-action on $\s^{-1}(x)$ is free, we can improve on Remark \ref{rem:metric:L/G}: $M/\G_x$ and $L/G$ are isomorphic as (Riemannian) stacks, i.e., the action groupoids $\G_x\ltimes M\tto M$ and $G\ltimes L\tto L$ are Morita equivalent. The Morita equivalence is given by the Riemannian manifold $\s^{-1}(x)$ where both groups $G$ and $\G_x$ act by isometries.

%\begin{proof}
%Perhaps, the only non-obvious fact is that  the metric $K_\G$ is $G$-invariant. Note that this metric is right-invariant (under translations in $\G$), and it is also invariant by the inner action of $G$ on $\G$ (since the inner $G$-action on $A$ is by isometries). We can express the left $G$-action of $g$ on a arrow $\gamma$ as the composition of the inner $G$-action by $g$ with the right action by the arrow $\iota(g,g\s(\gamma))$:
%\[ g\gamma=\iota(g,\t(\gamma))\cdot \gamma \cdot \iota(g^{-1},\s(\gamma)) \cdot \iota(g,g\s(\gamma)).\]
%This shows that $K_\G$ is $G$-invariant.
%\end{proof}

The main result concerning metric type $G$-structure algebroids is the following theorem concerning metric completeness of solutions to Cartan's realization problem, stating that, up to cover, metric complete solutions arise from source fibers of $G$-integrations over metric complete leaves. This shows that metric complete solutions are much more constrained than general solutions and can always be found by $G$-integration.

\begin{theorem}
\label{thm:metric:complete}
Let $A\to X$ be a $G$-structure algebroid with connection of metric type and let $(P,(\theta,\omega),h)$ be a $G$-realization mapping into a leaf $L$. The following are equivalent:
\begin{enumerate}[(a)]
\item The metric $K_M$ on $M=P/G$ is complete;
\item The $G$-realization $(P,(\theta,\omega),h)$ is complete and the metric $\tilde K_L$ on $L$ is complete.
\end{enumerate}
In particular, if $M=P/G$ is metric complete and 1-connected, then $P$ is isomorphic to a source fiber of the canonical $G$-integration.
\end{theorem}

\begin{proof}
Let us start by observing that it is enough to prove the equivalence between (a) and (b) when $M$ is 1-connected. Indeed, given a $G$-realization $(P,(\theta,\omega),h)$, if $\tilde{M}$ denotes the universal cover of $M=P/G$, then we have a pullback diagram
\[
\xymatrix{
\tilde{P}\ar[d]\ar[r]^q& P\ar[d]\\
\tilde{M}\ar[r] & M
}
\]
For the $G$-realization $(\tilde{P},(q^*\theta,q^*\omega),q^*h)$, $\tilde{M}=\tilde{P}/G$ is 1-connected, and we have:
\begin{enumerate}[(i)]
\item $M$ is a complete Riemannian manifold if and only if $M$ is complete;
\item $(\tilde{P},(q^*\theta,q^*\omega),q^*h)$ and $(P,(\theta,\omega),h)$ cover the same leaf $L$;
\item $(\tilde{P},(q^*\theta,q^*\omega),q^*h)$ is a complete $G$-realization if and only if $(P,(\theta,\omega),h)$ is complete.
\end{enumerate}
The last item follows because the infinitesimal action $\tilde{\sigma}:A|_L\to \X(\tilde{P})$ covers the infinitesimal action $\sigma:A|_L\to \X(P)$. 
\medskip

Hence, in order to show that (a) $\Rightarrow$ (b), assume that $M$ is metric complete and 1-connected. Since $G$ is compact, $P$ is also a complete Riemannian manifold and we see that since $h:P\to L$ is a Riemannian submersion, it must be onto.  Now, completeness implies that every infinitesimal isometry of $M$ is complete (see \cite[Theorem VI.3.4]{KN1}), so the Lie group of isometries $I(M)$ acts on $P$ in a proper and free fashion (by lifts) with Lie algebra $\Ker\rho_{h(u)}$, and the action commutes with the $G$-action. Moreover, $h$ induces a covering $h:P/I(M)\to L$ and a fiberwise isomorphism $(h,\theta,\omega):TP/I(M)\to A|_L$. Set $\tilde L:=P/I(M)$ so that $TP/I(M)\to \tilde L$ is a Lie $G$-algebroid isomorphic to the pullback algebroid $h^*(A|_L)$. It follows that the gauge groupoid $(P\times P)/I(M)\tto \tilde L$ is a $G$-integration of $h^*(A|_L)$. This implies that $P$, being a source fiber of a $G$-integration, is a complete $G$-realization of $h^*(A|_L)$. Now we observe that the infinitesimal actions of $A|_L$ and $h^*(A|_L)$ on $P$ fit into a commutative diagram
\[
\xymatrix{
h^*(A|_L)\ar[dr]^{\sigma}\ar[d]\\
A|_L\ar[r]_{\sigma} & TP}
\]
From this it follows that the infinitesimal action $\sigma:A|_L\to \X(P)$ must also be complete, so that $P$ is a complete $G$-realization of $A|_L$. Since $P/G$ is 1-connected, it follows that $P$ must be isomorphic to a source fiber of the canonical $G$-integration, proving the last statement of the theorem.

Conversely, to show that (b) $\Rightarrow$ (a) note that if the $G$-realization $(P,(\theta,\omega),h)$ is complete and $M=P/G$ is 1-connected, then Proposition \ref{prop:1-connected:complete}, shows that $P\simeq \s^{-1}(x)$ for the canonical $G$-integration $\Sigma_G(A)\tto X$, and $h=\t:\s^{-1}(x)\to L$. Under the isomorphism $P\simeq \s^{-1}(x)$, the metric becomes the right-invariant metric $\tilde K_{\Sigma_G(A)}$. This metric is complete if and only if the metric on $L$ is complete. Therefore if (b) holds, the metric on $P$ is complete, and so is the metric on the quotient  $\s^{-1}(x)\to \s^{-1}(x)/G=M$.
\end{proof}

\begin{corol}
Let $A\to X$ be a $G$-structure algebroid with connection of metric type and assume that a leaf $L$ is metric complete (e.g., a compact leaf). Then every complete solution covering $L$ is  metric complete. In particular, every solution $(\s^{-1}(x),\wmc,\t)$ obtained from a $G$-integration $\G_L \tto L$ of $A_L$ is metric complete.
\end{corol}

\begin{example}
If  $A=\Rr^n\times(\Rr^n\oplus\so(n,\Rr))\to \Rr$ is the constant scalar curvature Lie $\SO(n,\Rr)$-algebroid (Example \ref{ex:constant:curvature}) then the metrics
$\tilde K_L$ are complete for all $L$, since these are just the points of $\Rr$. Hence, its canonical $\SO(n,\Rr)$-integration gives all the 1-connected, constant curvature, complete Riemannian manifolds, and one recovers the usual classification: $\Ss^n$ (positve curvature), $\Rr^n$ (zero curvature) and $\mathbb{H}^n$ (negative curvature).

The flat $n$-torus $\Tt^n$ is also a complete Riemannian manifold with zero curvature, but its oriented orthogonal frame bundle is not a strongly complete $G$-realization. So $\Tt^n$ cannot be realized as $\s^{-1}(0)/\SO(n,\Rr)$ for some $\SO(n,\Rr)$-integration.

When $n$ is odd, the real projective space $\mathbb{RP}^n$ has constant positive scalar curvature metric, which is obviously complete. Its oriented orthogonal frame bundle is the group $P\SO(n+1,\Rr)$, which gives a strongly complete $\SO(n,\Rr)$-realization. Hence, $\mathbb{RP}^n=\s^{-1}(x)/\SO(n,\Rr)$ for a $\SO(n,\Rr)$-integration.
\end{example}

%**** LOCALLY FLAT METRIC ALGEBROIDS ****

%%%%%%%%%%%%%%%%%%%%%%%%%%%%%%%%%%%%%%
\subsection{Symplectic type}
%%%%%%%%%%%%%%%%%%%%%%%%%%%%%%%%%%%%%%

In the symplectic case, we will need to pay attention to the integrability condition. We start without any integrability condition:

\begin{definition}
A  $G$-structure algebroid with connection and fiber $\Rr^{2n}\oplus\gg$ is said to be of {\bf almost-symplectic type} if $G\subset \Sp(n,\Rr)$.
\end{definition}

For a $G$-structure algebroid $A\to X$ of almost-symplectic type $\gg\subset\sp(n,\Rr)$ so the curvature $R$ satisfies
\begin{equation}
\label{eq:symmetry:symplectic}
\Omega_{\can}(R(u,v)w,z)+\Omega_{\can}(w,R(u,v)z)=0, 
\end{equation}
where $\omega_{\can}\in\Omega^{2}(\Rr^{2n})$ denotes the canonical symplectic form on $\Rr^{2n}$. We can define a fiberwise 2-form $\Omega_A$ on $A$ by
\begin{equation}
\label{eq:A:form} 
\Omega_A((u,\al),(v,\beta)):=\Omega_{\can}(u,v),\quad (u,\al),(v,\be)\in A,
\end{equation}
Then, as we saw above, for any $G$-realization $(P,(\theta,\omega),h)$ of the associated Cartan's realization problem, the corresponding solution $M=P/G$ inherits a non-degenerate 2-form $\Omega_M$ and for the linear connection $\nabla$ on $TM$ we have $\nabla \Omega_M=0$. Notice that $\Omega_M$ in general will not be closed.

For a Lie algebroid $A\to X$ we will denote by $\d_A:\Omega^\bullet(A)\to \Omega^{\bullet+1}(A)$ the Lie algebroid differential. 

\begin{definition}
An almost-symplectic $G$-structure algebroid $A\to X$ is said to be of {\bf symplectic type} if $\d_A\Omega_A=0$.
\end{definition}

A straightforward computation shows that:

\begin{lemma}
Let $A\to X$ be a  $G$-structure algebroid of almost-symplectic type. Then
\[ \d_A \Omega_A((u,\al),(v,\be),(w,\gamma))=\Omega_{\can}(c(u,v),w)+\Omega_{\can}(c(v,w),u)+\Omega_{\can}(c(w,u),v). \]
\end{lemma}

%\begin{proof}
%Let $s_1=(u,\al)$, $s_2=(v,\be)$ and $s_3=(w,\gamma)$ be constant sections. Denoting by $\bigodot_{a,b,c}$ the cyclic sum over $a,b,c$, we have:
%\begin{align*}
%\d_A \Omega(s_1.s_2,s_3)
%&=\bigodot_{s_1.s_2,s_3} \rho(s_1)(\Omega(s_2,s_3)+\Omega(s_1,[s_2,s_3])\\
%&=\bigodot_{s_1.s_2,s_3}\Omega((u,\al)),[(v,\be),(w,\gamma)])\\
%&=\bigodot_{s_1.s_2,s_3} \left(\Omega_{\can}(u,\be\cdot w-\gamma\cdot v)-\Omega_{\can}(u,c(v,w))\right)\\
%&=\bigodot_{u,v,w}\Omega_{\can}(c(u,v),w))
%\end{align*}
%where we used that $\al,\be,\gamma\in\gg\subset\sp (n,\Rr)$.
%\end{proof}

In particular, every $G$-structure algebroid $A\to X$ of almost-symplectic type which is torsionless is of symplectic type. 

\begin{prop}
\label{prop:symp:Lie:G:grpd}
Let $(P,(\theta,\omega),h)$ be a $G$-realization of a $G$-structure algebroid $A\to X$ with connection of symplectic type. Then the induced form $\Omega_M$ on $M=P/G$ is symplectic and $\omega$ is a symplectic connection.
\end{prop}

\begin{proof}
The Lie algebroid morphism 
\[
\xymatrix{
TP \ar[d]\ar[r]^{(\theta,\omega)} & A|_L\ar[d]\\
P \ar[r]_h & L
}
\]
allows us to pullback the $A$-form $\Omega_A$ to a 2-form $\Omega_P$ on $P$, which is $G$-invariant, closed and whose kernel coincides with the $G$-orbits. Hence, $\Omega_P$ descends to a symplectic form  $\Omega_M$ on $M=P/G$. By standard results for $\Sp(n,\Rr)$-structures, $\omega$ is a symplectic connection.
\end{proof}

\begin{remark}
In the case of a $G$-structure algebroid with connection of metric type we saw that the leaves $L$ and their quotients $L/G$ have natural metrics such that for any $G$-realization we obtain a commutative diagram of Riemannian submersions
\[
\xymatrix{
P \ar[d]\ar[r]^h & L\ar[d]\\
M \ar[r] & L/G
}
\]
It is natural to wonder what happens in the case of a $G$-structure algebroid with connection of symplectic type. Since 2-forms, in general, cannot be pushed forward, one must instead work with their graphs which are Dirac structures. One can then show that in the diagram above all manifolds are Dirac structures and all maps are forward Dirac maps (see e.g. \cite{Bursztyn13}). The Dirac structure on $P$ and $M$ are the graphs of the closed 2-forms $\Omega_P$ and $\Omega_M$ above, and since $\Omega_M$ is symplectic, the resulting Dirac structure on $L/G$ is actually a Poisson structure (which, in general, is stratified by orbits types).
\end{remark}

% **** LOCALLY FLAT SYMPLECTIC ALGEBROIDS ****

%%%%%%%%%%%%%%%%%%%%%%%%%%%%%%%%%%%%%%
\subsection{Complex type}
%%%%%%%%%%%%%%%%%%%%%%%%%%%%%%%%%%%%%%

\begin{definition}
A  $G$-structure algebroid with connection and fiber $\Cc^n\oplus\gg$ is said to be of {\bf almost-complex type} if $G\subset \GL(n,\Cc)$.
\end{definition}

For a $G$-structure algebroid $A\to X$ of almost-complex type $\gg\subset \gl(n,\Cc)$, so the curvature $R$ satisfies
\begin{equation}
\label{eq:symmetry:complex}
R(u,v)(iw)=i R(u,v)(w).
\end{equation}
We define an endomorphism $J_A:A\to A$ by
\begin{equation}
\label{eq:A:J}
J_A(u,\al)=(iu,0),\quad (u,\al)\in A.
\end{equation}
Notice that $J_A^3=-J_A$. As we saw above, for any $G$-realization $(P,(\theta,\omega),h)$ of the associated Cartan's realization problem, the corresponding solution $M=P/G$ inherits an almost complex structure $J_M$ and for the linear connection $\nabla$ on $TM$ we have $\nabla J_M=0$. Notice that $J_M$ in general will not be a complex structure.

Given any endomorphism $J:A\to A$ of a Lie algebroid $A\to X$ one defines its Nijenhuis torsion as usual by
\[ N_J(a,b):=[Ja, Jb]-J[Ja,b]-J[a,Jb]+J^2[a,b]. \]
In the case of a $G$-structure algebroid $A\to X$ of almost-complex type a straightforward computation shows that:

\begin{lemma}
Let $A\to X$ be an almost-complex $G$-structure algebroid. The endomorphism $J_A$ has Nijenhuis torsion
\[
N_{J_A}((u,\al),(v,\be))=(-c(iu,iv)+ic(iu,v)+ic(u,iv)+c(u,v),-R(iu,iv)).
\]
\end{lemma}

This lemma shows that we cannot expect the  Nijenhuis torsion to vanish. The best we can hope is for the $\Cc^n$-component of $N_J$ to vanish and this enough for our purposes.

\begin{definition}
An almost-complex $G$-structure algebroid $A\to X$ is said to be of {\bf complex type} if the Nijenhuis torsion of $J_A$ has vanishing $\Cc^n$-component.
\end{definition}

In particular, every $G$-structure algebroid $A\to X$ of almost-complex type which is torsionless is of complex type. 

\begin{prop}
\label{prop:complx:Lie:G:grpd}
Let $(P,(\theta,\omega),h)$ be a $G$-realization of a $G$-structure algebroid $A\to X$ of complex type. Then the induced 1-1-tensor $J_M$ on $M=P/G$  is a complex structure and $\omega$ is a complex connection.
\end{prop}

\begin{proof}
The Lie algebroid morphism 
\[
\xymatrix{
TP \ar[d]\ar[r]^{(\theta,\omega)} & A|_L\ar[d]\\
P \ar[r]_h & L
}
\]
allows us to pullback $J_A$ to a $G$-invariant endomorphism $J_P:TP\to TP$ satisfying $(J_P)^3=-J_P$. This endomorphism has kernel the tangent space to a $G$-orbit and its Nijenhuis torsion takes values in the tangent to the $G$-orbits:
\[ N_{J_P}(X,Y)\in T(p\cdot G),\quad \forall X,Y\in TP. \] 
Hence, it follows that $J_P$ induces an endomorphism $J_M:TM\to TM$, which is invertible, has vanishing Nijenhuis torsion and satisfies $J_M^3=-J_M$. Hence, $J_M$ defines a complex structure on $M$. By standard results for $\GL(n,\Cc)$-structures, $\omega$ is a complex connection for this complex structure.
\end{proof}

% **** LOCALLY FLAT COMPLEX ALGEBROIDS ****

We can combine all the three types we studied before and introduce:

\begin{definition}
A  $G$-structure algebroid with connection and  fiber $\Cc^n\oplus\gg$ is said to be of {\bf almost-Hermitian type} if $G\subset U(n)$.  It is said to be of {\bf Hermitian type} if, additionally,  the Nijenhuis torsion of $J_A$ has vanishing $\Cc^n$-component. It is said to be of {\bf K\"ahler type} if, additionally, its torsion vanishes: $c=0$.
\end{definition}

For a $G$-structure algebroid $A\to X$ of almost-Hermitian type we have $\gg\subset\uu(n)$ and the curvature $R$ satisfies the symmetries \eqref{eq:symmetry:Riemannian}, \eqref{eq:symmetry:symplectic} and \eqref{eq:symmetry:complex}.  On the other hand, the fiberwise metric $\tilde K_A$ given by \eqref{eq:A:metric}, the 2-form $\Omega_A$ given by \eqref{eq:A:form} and the endomorphism $J_A:A\to A$ given by \eqref{eq:A:J}, are related by
\begin{equation}
\label{eq:A:compatibility}
\Omega_A((u,\al),(v,\be))=\tilde K_A(J_A(u,\al),(v,\be)).
\end{equation}
Hence, it follows from what we saw above that:

\begin{corol}
Let $(P,(\theta,\omega),h)$ be a $G$-realization of a $G$-structure algebroid $A\to X$ of K\"ahler type. Then the triple $(K_M,\Omega_M,J_M)$ is a K\"ahler structure on $M=P/G$.
\end{corol}

%%%%%%%%%%%%%%%%%%%%%%%%%%%%%%%%%%%%%%
%%%%%%%%%%%%%%%%%%%%%%%%%%%%%%%%%%%%%%
\section{Extremal K\"ahler metrics on surfaces}
\label{sec:example:EK:surfaces}
%%%%%%%%%%%%%%%%%%%%%%%%%%%%%%%%%%%%%%
%%%%%%%%%%%%%%%%%%%%%%%%%%%%%%%%%%%%%%

We will now illustrate our results with the problem of classification of extremal K\"ahler metrics on surfaces and the associated Lie $\mathrm{U}(1)$-algebroid with connection discussed in Section \ref{sec:example:EK:surfaces:1}. Since extremal K\"ahler metrics on surfaces are the same thing as Bochner- K\"ahler metrics on surfaces (see Remark \ref{rem:EK:metrics}), our results will reproduce and slightly improve the results of Bryant \cite{Bryant}. Note that this algebroid falls into the class of Lie algebroids with connection of K\"ahler type. 

%%%%%%%%%%%%%%%%%%%%%%%%%%%%%%%%%%%%%%
\subsection{Existence of local solutions}
%%%%%%%%%%%%%%%%%%%%%%%%%%%%%%%%%%%%%%

In order to find the obstructions to $G$-integrability of $A\to X$ we need to understand its orbit foliation. For that, it is convenient to write the anchor in real coordinates. If we let $\al=i\lambda$, $z=a+ib$ and $T=X+iY$, then the anchor becomes
{\small
\begin{align} 
\rho(z,\al)|_{(K,T,U)}&=-2(aX+bY)\frac{\partial}{\partial K}+(aU+\lambda Y)\frac{\partial}{\partial X}+(bU-\lambda X)\frac{\partial}{\partial Y} +-K(aX+bY)\frac{\partial}{\partial U}\notag\\
%&\qquad \qquad \qquad \qquad -2K(aX+bY)\frac{\partial}{\partial U}\\
&=a\left(-2X\frac{\partial}{\partial K}+U\frac{\partial}{\partial X}-KX \frac{\partial}{\partial U}\right)+b\left(-2Y\frac{\partial}{\partial K}+U\frac{\partial}{\partial Y}-KY \frac{\partial}{\partial U}\right)+\notag\\
&\qquad \qquad \qquad \qquad +\lambda \left(Y\frac{\partial}{\partial X}-X\frac{\partial}{\partial Y}\right)\label{eq:anchor:EK}
\end{align}}
Also, the brackets of the constant sections $e_1=(1,0)$, $e_2=(i,0)$ and $e_3=(0,i)$ are
\begin{equation}
\label{eq:bracket:EK}
[e_1,e_2]=Ke_3, \quad [e_1,e_3]=-e_2,\quad [e_2,e_3]=e_1.
\end{equation}
Notice that the action morphism $\iota:X\rtimes i\mathfrak{u}(1)\to A$ is given by
\[ \iota(x,i\lambda)=\lambda e_3. \]

From this expression it is clear that the functions 
\[ I_1=\frac{K^2}{4}-U, \qquad I_2=X^2+Y^2+KU-\frac{1}{6}K^3, \]
are constant on the leaves of $A$. These two functions are independent everywhere except at $X=Y=U=0$ when the anchor vanishes. So this Lie $U(1)$-algebroid has the following leaves and isotropy Lie algebras:
\begin{itemize}
\item the points $(K,0,0,0)$ with isotropy Lie algebra $\so(3,\Rr)$ (if $K> 0$), $\sl(2,\Rr)$ (if $K< 0$) and  $\so(2,\Rr)\ltimes\Rr^2$ (if $K=0$); 
\item the 2-dimensional submanifolds of $\Rr^4$ obtained by $U(1)$-rotations of the curves in $\Rr^3$ which are the level sets of the functions
\[ 
U=\frac{K^2}{4}-c_1,\qquad |T|^2=-\frac{1}{12}K^3+c_1 K+c_2.
\]
\end{itemize}

The restriction $A$ to the family of  $0$-dimensional leaves $\{(K,0,0,0): K\in\Rr\}$ is the constant curvature Lie $U(1)$-algebroid (the case $n=2$ in Example \ref{ex:constant:curvature} and Section \ref{sec:constant:curvature}), hence these leaves are $U(1)$-integrable. 

For the 2-dimensional leaves, observe that the functions $I_1$ and $I_2$ only depend on the radius $|T|^2=X^2+Y^2$. Hence, these leaves are obtained as $U(1)$-rotations of the curves in $\Rr^3$ which are the level sets of these two functions:
\[ 
\left\{
\begin{array}{l}
I_1=c_1  \\
I_2=c_1
\end{array}
\right.
\qquad \Leftrightarrow \qquad
\left\{
\begin{array}{l}
U=\frac{K^2}{4}-c_1  \\
|T|^2=-\frac{1}{12}K^3+c_1 K+c_2
\end{array}
\right.
\]
Hence, we can use $K$ as a parameter. Depending on the values of $c_1$ and $c_2$, this curve is closed/open and bounded/unbounded. These determine if the leaves have topology and hence also monodromy and/or $G$-monodromy. 

The cubic 
\[ p(K)=-\frac{1}{12}K^3+c_1K+c_2\] 
has discriminant
\[ \Delta=\frac{1}{48}(16c_1^3-9c_2^2). \]
Also, note that a 0-dimensional leaf $(K,0,0,0)$ belongs to a common level set $I_1=c_1$, $I_2=c_2$, if and only if 
\[ 
\left\{
\begin{array}{l}
K^2=4c_1 \\
0=-\frac{1}{12}K^3+c_1K+c_2
\end{array}
\right.
\quad \Rightarrow\quad  16c_1^3-9c_2^2=0 \quad \Leftrightarrow\quad \Delta=0. \]

Now, the 2-dimensional leaves are as follows:
\medskip  

%%%
$\bullet$ $\Delta=0$, $c_1=c_2=0$: The cubic $p(K)$ has a triple root. The level set consists of one single leaf, namely the surface obtained by rotating the curve
\[ 
\left\{
\begin{array}{l}
U=K^2 \\
|T|^2=-\frac{1}{12}K^3
\end{array}
\right.\quad K\in ]-\infty,0[.\] 
The value $K=0$ is excluded since it gives the origin $(0,0,0,0)$, which is a 0-dimensional leaf. Therefore, this leaf is topological a cylinder, so it has
\[ \pi_1(L)=\Zz,\quad \pi_2(L)=1. \]
Hence, the extended monodromy is trivial and we have $\Sigma(A)_x^0\simeq \Rr$ and $\pi_0(\Sigma(A)_x)=\Zz$. The restricted $U(1)$-monodromy is also trivial, so $A|_L$ is $U(1)$-integrable.
\medskip

\begin{center}
\includegraphics[height=3cm]{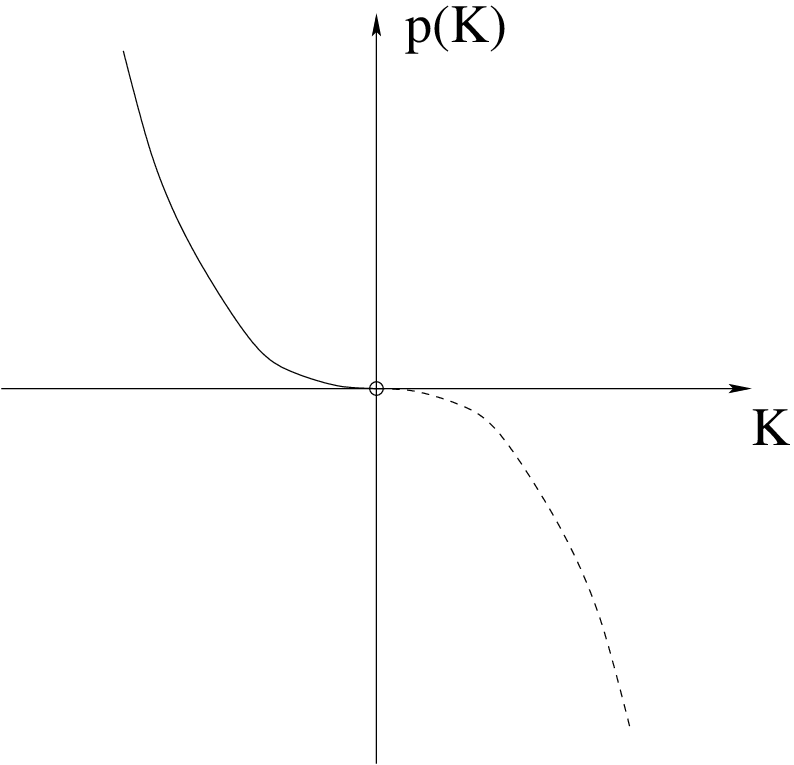}
\end{center}

%%%
\medskip
$\bullet$ $\Delta=0$, $c_2<0$: The cubic $p(K)$ has 1 single real root $-4\sqrt{c_1}$ and 1 double real root $2\sqrt{c_1}$. The level set consists of an isolated point $(2\sqrt{c_1},0,0,0)$, which is a 0-dimensional leaf, and a 2-dimensional leaf, namely the surface obtained by rotating the curve
\[ 
\left\{
\begin{array}{l}
U=\frac{1}{4}K^2-c_1 \\
 |T|^2=-\frac{1}{12} (K-2\sqrt{c_1})^2(K+4\sqrt{c_1})
\end{array}
\right.\quad K\in ]-\infty,-4\sqrt{c_1}].\] 
Topologically, this leaf is a plane, so we have
\[ \pi_1(L)=\pi_2(L)=1. \]
We conclude that $A|_L$ must also be $U(1)$-integrable in this case.
\medskip

\begin{center}
\includegraphics[height=3cm]{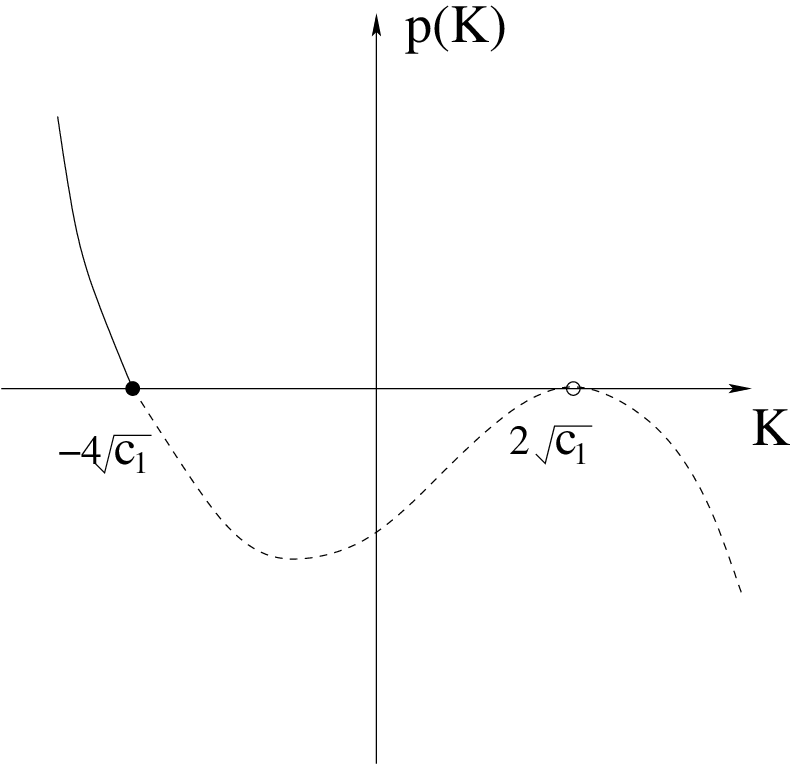}
\end{center}

%%%
\medskip  
$\bullet$ $\Delta=0$, $c_2>0$: The cubic $p(K)$ has 1 single real root $4\sqrt{c_1}$ and 1 double real root $-2\sqrt{c_1}$. The level set then consists of a 0-dimensional leaf $(-2\sqrt{c_1},0,0,0)$, and two 2-dimensional leaves obtained by rotating the curve
\[ 
\left\{
\begin{array}{l}
U=\frac{1}{4}K^2-c_1 \\
 |T|^2=-\frac{1}{12} (K+2\sqrt{c_1})^2(K-4\sqrt{c_1})
\end{array}
\right.\quad K\in ]-\infty,-2\sqrt{c_1}[\cup ]-2\sqrt{c_1},4\sqrt{c_1}].\] 
One of the leaves is a cylinder and the other is a plane. Again we conclude that $A|_L$ is $U(1)$-integrable.
\medskip

\begin{center}
\includegraphics[height=3cm]{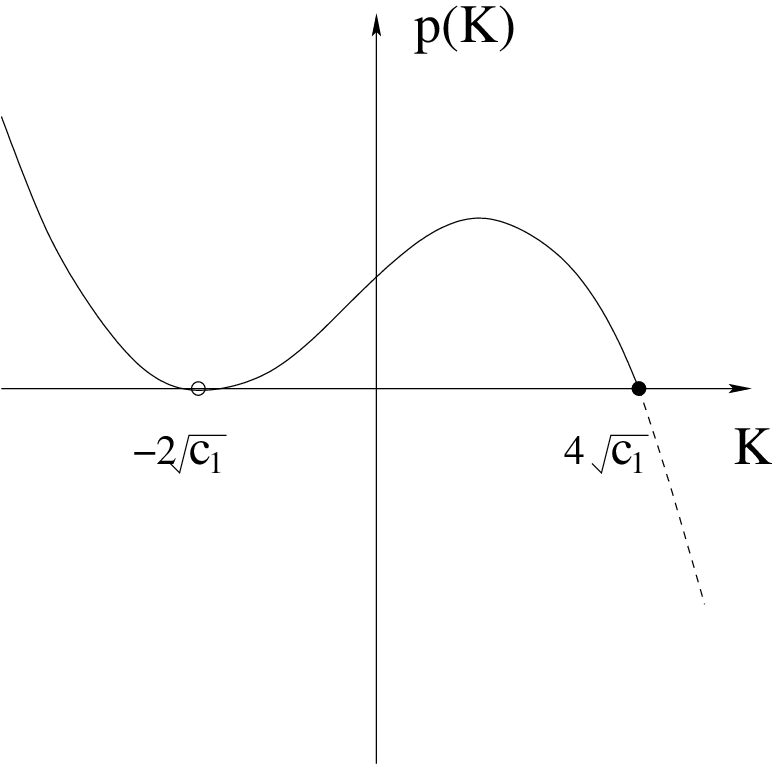}
\end{center}

%%%
\medskip  
$\bullet$  $\Delta<0$: The cubic $p(K)$ has 1 real root $r$ and two complex conjugate roots. The level set consists of a 2-dimensional leaf obtained by rotating the curve
\[ 
\left\{
\begin{array}{l}
U=\frac{1}{4}K^2-c_1 \\
 |T|^2=-\frac{1}{12} (K-r)(K^2+rK+r^2-12c_1)
\end{array}
\right.\quad K\in ]-\infty,r].\] 
This leaf is a plane, so it is $U(1)$-integrable.
\medskip

\begin{center}
\includegraphics[height=3cm]{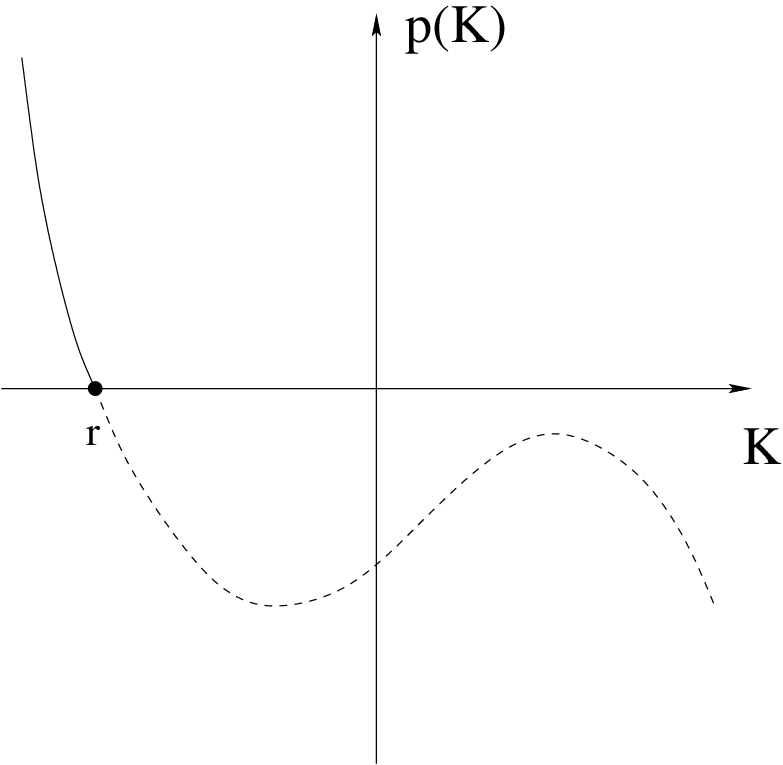}
\end{center}

%%%
\medskip  
$\bullet$  $\Delta>0$: The cubic $p(K)$ has 3 real roots $r_1<r_2<r_3$. The level set consists of two  2-dimensional leaves obtained by rotating the curve
\[ 
\left\{
\begin{array}{l}
U=\frac{1}{4}K^2-c_1 \\
 |T|^2=-\frac{1}{12} (K-r_1)(K-r_2)(K-r_3)
\end{array}
\right.\quad K\in ]-\infty,r_1]\cup [r_2,r_3].\] 
One leaf is a plane and the other leaf is a sphere. So the first one is $U(1)$-integrable, while the second one could fail to be $U(1)$-integrable. 
\medskip

\begin{center}
\includegraphics[height=3cm]{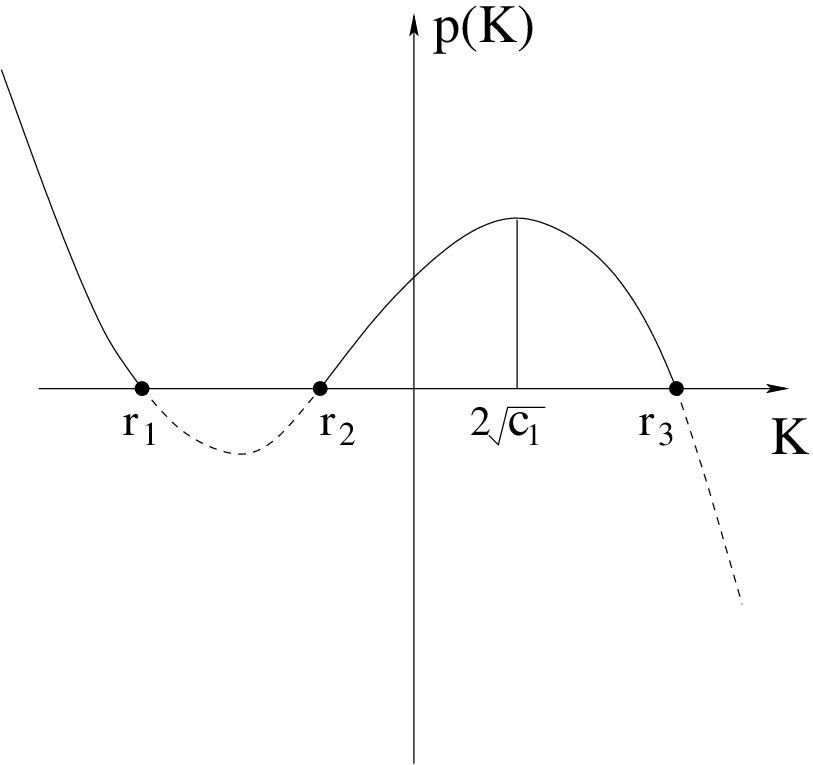}
\end{center}
\medskip

Let us compute the monodromy and the $U(1)$-monodromy of one such 2-sphere. The 2-sphere $L$ has a parameterization given by
\[ \gamma(K,\theta)=(K,p(K)^{\frac{1}{2}}e^{i\theta},K^2/4-c_1),\quad (K,\theta)\in[r_2,r_3]\times [0,2\pi], \]
and the isotropy Lie algebra bundle along the sphere is given by
\[ \ker\rho|_L=\Rr s_0, \]
where $s_0:L\to A|_L$ is the section given in terms of the parameterization by
\[ s_0(K,\theta):=(p(K)^{\frac{1}{2}} i e^{i\theta},(K^2/4-c_1)i). \]
Since the isotropy is abelian any splitting will have center-valued curvature. We can construct a splitting by choosing a fiberwise metric on $A$. The anchor restricts to an isomorphism 
$\rho:(\ker\rho|_L)^\perp \to TL$, and its inverse is a splitting $\sigma:TL\to (\ker\rho|_L)^\perp$. We will choose the fiberwise metric
\[ \langle (z,\al), (w,\be)\rangle:=1/2(z\bar{w}+\bar{z}w)-\al\be.\]
Using this splitting, we find that
\begin{align*}
\sigma\left(\frac{\partial \gamma}{\partial K}\right)&=(-\frac{1}{2}p(K)^{-\frac{1}{2}}e^{i\theta},0)\\
\sigma\left(\frac{\partial \gamma}{\partial \theta}\right)&=\frac{1}{p(K)+(K^2/4-c_1)^2}\left( p(K)^{\frac{1}{2}}(K^2/4-c_1)i e^{i\theta},-p(K) i\right)
\end{align*}
Using these expressions, we check that
\[ \nabla_{\frac{\partial \gamma}{\partial K}}s_0=\left[\sigma\left(\frac{\partial \gamma}{\partial K}\right),s_0\right]=0,\quad \nabla_{\frac{\partial \gamma}{\partial \theta}}s_0=\left[\sigma\left(\frac{\partial \gamma}{\partial \theta}\right),s_0\right]=0. \]
so $s_0$ is a flat section of the isotropy bundle. Moreover, we also compute the curvature of the splitting to be
\[ 
\Omega_\sigma\left(\frac{\partial \gamma}{\partial K}, \frac{\partial \gamma}{\partial \theta}\right)
=\left[\sigma\left(\frac{\partial \gamma}{\partial K}\right),\sigma\left(\frac{\partial \gamma}{\partial \theta}\right)\right]
=\frac{\partial }{\partial K}\left( \frac{K^2/4-c_1}{p(K)+(K^2/4-c_1)^2}\right) s_0
 \]

Now, Proposition \ref{prop:monodromy:comput} shows that a generator of the the monodromy group is obtained by integrating the curvature over the whole leaf. Hence,
\[ \cN=8\pi \Zz\left( \frac{1}{r_3^2-4c_1}+\frac{1}{4c_1-r_2^2}\right)s_0 ,\]
so $\cN$ is discrete and the $A|_L$ is integrable.

\begin{remark}
\label{rem:EK:integrability}
Combining the monodromy for all the cases, we see that that the Lie algebroid $A\to \Rr^4$ associated with extremal K\"ahler metrics on surfaces has uniformly discrete monodromy and hence, by \cite{CrainicFernandes}, it is an integrable Lie algebroid.
\end{remark}

On the other hand, observe that along the $U(1)$-orbit $K=2\sqrt{c_1}$ (the local maximum of $p(K)$) the flat section takes the values $s_0(2\sqrt{c_1},\theta)=(p(2\sqrt{c_1})^{\frac{1}{2}} i e^{i\theta},0)$. Then
\[ \uu(1)\times \{\gamma(2\sqrt{c_1},\theta)\} \subset (\ker\rho|_{\gamma(2\sqrt{c_1},\theta)})^\perp, \]
and $\sigma$ is a $U(1)$-splitting along this orbit. Hence, applying Proposition \ref{prop:G:monodromy:comput}, the $U(1)$-monodromy is obtained by adding to the monodromy the elements of the form $\Zz\int_{\gamma'}\Omega_\sigma~s_0$, 
with $\gamma'$ the disk in $L$ obtained by restricting $\gamma$ to $[2\sqrt{c_1},r_3]\times [0,2\pi]$. Hence, the $U(1)$-monodromy groups along $L$ are
\[ \cN^{U(1)}=\cN\cup 2\pi \Zz\left( \frac{1}{r_3^2/4-c_1}\right)s_0=\left\{ 8\pi\left( \frac{n_1}{r_3^2-4c_1}+\frac{n_2}{4c_1-r_2^2}\right)s_0: n_1,n_2\in\Zz\right\}. \]
It follows that, $A|_L$ is $U(1)$-integrable if and only if $\frac{4c_1-r_2^2}{r_3^2-4c_1}\in\Qq$. This ratio is always $<1$ and fixing $c_1$ it varies continuously with $c_2$, approaching $1$  when $\Delta\to 0$.

The following table summarizes the solutions that one obtains by integration by applying Corollary \ref{cor:local:existence}:
\medskip

\begin{table}[ht]
\caption{Extremal K\"ahler metrics on 1-connected surfaces} % title of Table
\centering
{\small
\begin{tabular}{|l|c|c|}
 \hline & &  \\ [1ex]
  {\bf Conditions} & {\bf $U(1)$-frame bundle}: $\s^{-1}(x)$ & {\bf Solutions}: $\s^{-1}(x)/U(1)$  \\[1ex] \hline & &  \\
 $K=0$ & $\SO(2)\ltimes\Rr^2$ & $\Rr^2$ \\ [1ex] \hline & &  \\
 $K=c>0$ & $\Ss^3$ & $\Ss^2$ \\ [1ex] \hline & &  \\
 $K=c<0$ & $\SO(2,1)$ & $\mathbb{H}^2$ \\ [1ex] \hline & &  \\
 $\Delta=0$, $c_1=c_2=0$ & $(\Rr^2\times\Rr)/\Zz$ & $\Rr^2$\\ [1ex] \hline & &  \\
 $\Delta=0$, $c_2<0$ & $\Rr^2\times\Ss^1$ & $\Rr^2$ \\ [1ex] \hline & &  \\
  \multirow{2}{*}{$\Delta=0$, $c_2>0$} & $(\Rr^2\times\Rr)/\Zz$ & \multirow{2}{*}{$\Rr^2$} \\ [1ex]
  					                   & $(\Rr^2\times\Ss^1)$ & \\ [1ex] \hline & &  \\
 $\Delta<0$ & $\Rr^2\times\Ss^1$ & $\Rr^2$ \\ [1ex] \hline
 & &  \\
  $\Delta>0$          & $\Rr^2\times\Ss^1$ & $\Rr^2$ \\ [1ex]
 (if $\frac{4c_1-r_2^2}{r_3^2-4c_1}=\frac{p}{q}$)    & $\Ss^3$ & $\mathbb{CP}^1_{p,q}$\\ [1ex] \hline 
					                   
\end{tabular}}
\label{table:solutions:EK} % is used to refer this table in the text
\end{table}

\bigskip

%%%%%%%%%%%%%%%%%%%%%%%%%%%%%%%%
\subsection{Complete extremal K\"ahler metrics}
\label{sec:complete:Kahler}
%%%%%%%%%%%%%%%%%%%%%%%%%%%%%%%%
Theorem \ref{thm:metric:complete} allows us to decide for which orbits there exist complete extremal K\"ahler metrics, and in particular which solutions in Table \ref{table:solutions:EK} are complete. For this, given a leaf $L$, we only need to decide if the metric $\tilde K_L$ (see the discussion following \eqref{eq:A:metric}) is complete. This can be done in a case by case analysis.

In all cases, we will use the parameterization
\[ \gamma(K,\theta)=(K,p(K)^{\frac{1}{2}}e^{i\theta},K^2/4-c_1),\quad \theta\in[0,2\pi],\ K\in I\]
for some interval $I$. This will parameterize smoothly the whole leaf $L$ with the possible exception of points where $p(K)=0$ which coincide with the extremes of the interval $I$. We also choose a section $s_0$ of the isotropy Lie algebra bundle $\ker\rho|_L$ given in terms of this parameterization by
\[ s_0(K,\theta):=(p(K)^{\frac{1}{2}} i e^{i\theta},(K^2/4-c_1)i). \]
This is non-zero with the possible exception of points in $L$ where $p(K)=0$ and $K=2\sqrt{c_1}$. Then using the fiberwise metric on $A\to\Rr^4$ given by
\[ \tilde K_A((z,\al), (w,\be))=1/2(z\bar{w}+\bar{z}w)-\al\be, \]
we can identify $(\ker\rho|_L)^\perp$ with $TL$ via a splitting $\sigma$ of the anchor. This yields the metric $K_L$. Explicitly, we find that
\begin{align*}
\sigma\left(\frac{\partial \gamma}{\partial K}\right)&=(-\frac{1}{2}p(K)^{-\frac{1}{2}}e^{i\theta},0)\\
\sigma\left(\frac{\partial \gamma}{\partial \theta}\right)&=\frac{1}{p(K)+(K^2/4-c_1)^2}\left( p(K)^{\frac{1}{2}}(K^2/4-c_1)i e^{i\theta},-p(K) i\right)
\end{align*}
so that if $e_1=\sigma\left(\frac{\partial \gamma}{\partial K}\right)$ and $e_2=\sigma\left(\frac{\partial \gamma}{\partial \theta}\right)$ then
\[ 
\tilde K_L\left(e_i,e_j\right)=\tilde K_A\left(\sigma\left(e_i\right),\sigma\left(e_j\right)\right).
\]
We obtain
\begin{equation}
\label{eq:metric:leaf} 
K_L=\frac{1}{4p(K)}\d K^2+ \frac{p(K)}{p(K)+(K^2/4-c_1)^2}\d\theta^2.
\end{equation}

Observe already that if the interval $I$ has left extreme $-\infty$ then this metric is not complete: we have that the integral
\[ \int^{K_0}_{-\infty} \frac{1}{2\sqrt{p(K)}}\d K<\infty, \]
so a geodesic with $\theta=$constant runs to $-\infty$ in finite time. Now we can look at each of different types of leaves:
\medskip

%%%
$\bullet$ $\Delta=0$, $c_1=c_2=0$:  In this case $p(K)=-\frac{1}{12}K^3$ and $I=]-\infty,0[$, so the metric is not complete.

\medskip
%%%
$\bullet$ $\Delta=0$, $c_2<0$: Now we have $p(K)=-\frac{1}{12} (K-2\sqrt{c_1})^2(K+4\sqrt{c_1})$ and $I=]-\infty,-4\sqrt{c_1}]$ so the metric is not complete.

\medskip
%%%
$\bullet$ $\Delta=0$, $c_2>0$: In this case we have  $p(K)=-\frac{1}{12} (K+2\sqrt{c_1})^2(K-4\sqrt{c_1})$ and we have two branches:
\begin{enumerate}
\item[-] In one branch we have $I=]-\infty,-2\sqrt{c_1}[$ so the metric is not complete.
\item[-] For the other branch we have $I=]-2\sqrt{c_1},4\sqrt{c_1}]$, so we need to look at what happens when $K\to -2\sqrt{c_1}$. In this case we have 
\[ \int^{K_0}_{-2\sqrt{c_1}} \frac{1}{2\sqrt{-p(K)}}\d K=\infty, \]
so a geodesic with $\theta=$constant exists for all time, and one obtains a complete metric on the disk $\mathbb{D}^2$.
\end{enumerate}

\medskip
%%%
$\bullet$  $\Delta<0$: We have the polynomial $p(K)=-\frac{1}{12} (K-r)(K^2+rK+r^2-12c_1)$ in the interval $I=]-\infty,r]$, so this metric is not complete.

\medskip
%%%
$\bullet$  $\Delta>0$: There are two branches. One corresponds to a infinite interval $I=]-\infty,r_1]$ and hence cannot be complete. The other branch is compact and if $\frac{4c_1-r_2^2}{r_3^2-4c_1}=\frac{p}{q}$ it is $U(1)$-integrable, so gives a complete metric on the orbifold $\mathbb{CP}^1_{p,q}$.

\begin{corol}
\label{cor:complete}
The 1-connected complete extremal K\"ahler metrics on a surface are the constant scalar curvature metrics $\Rr^2$, $\Ss^2$, $\mathbb{H}^2$, and two special metrics: one on a disk $\mathbb{D}^2$ and the other on $\mathbb{CP}^1_{p,q}$.
\end{corol}

%%%%%%%%%%%%%%%%%%%%%%%%%%%%%%%%%%%%%%
\subsection{Explicit solutions}
%%%%%%%%%%%%%%%%%%%%%%%%%%%%%%%%%%%%%%

One can try to find explicit $G$-integrations of the Lie $\mathrm{U}(1)$-algebroid with connection associated with extremal K\"ahler surfaces by finding a Lie $\mathrm{U}(1)$-groupoid integrating it, at least upon restriction to appropriate orbits. This is possible, with some work, using a certain identification of this Lie algebroid which was first found by Cahen \& Schwachh\"ofer \cite{CaSch2,CaSch}. In that work the authors claim that this Lie algebroid is not integrable (which is false as we observed in Remark \ref{rem:EK:integrability}). What they show, in fact, is that this Lie algebroid is not $G$-integrable. Still the restriction to an orbit, as we saw in the previous paragraph, is often $G$-integrable and we can find explicit  $G$-integrations, as we now explain.

As before, we denote by $A\to \Rr^4$ the Lie $\mathrm{U}(1)$-algebroid with connection associated with extremal K\"ahler surfaces. The starting point of a new description of this algebroid is the Lie algebra $\hh=\mathfrak{su}(2,1)$ which can be thought of as being made of $3\times 3$ matrices of the form
\[
\left[
\begin{array}{ccc}
i(a-b/2)  & u  & v  \\
-\bar{u}  &  ib & w  \\
\bar{v}  & \bar{w}  &  -i(a+b/2)
\end{array}
\right]
\]
where $a,b\in\Rr$ and $u,v,w\in\Cc$. The dual of any Lie algebra is a Poisson manifold and so $\hh^*=\mathfrak{su}(2,1)^*$ is a Poisson manifold. Using the symmetric, invariant, non-degenerate bilinear form (a multiple of the Killing form of $\mathfrak{su}(2,1)$),
\[ (x,y):=\tr(xy), \] 
we can identify $\mathfrak{su}(2,1)^*\simeq \mathfrak{su}(2,1)$, and henceforth we will use this identification. 

As for any Poisson manifold, the cotangent bundle $T^*\mathfrak{su}(2,1)$ is a Lie algebroid.  As a Lie algebroid, $T^*\mathfrak{su}(2,1)$ is isomorphic to the action Lie algebroid associated with the adjoint representation
\[ T^*\mathfrak{su}(2,1)\simeq \mathfrak{su}(2,1)\ltimes \mathfrak{su}(2,1)\to \mathfrak{su}(2,1). \]
This is not yet the algebroid we are interested in. For that, we consider the 4-dimensional affine subspace $X\subset  \mathfrak{su}(2,1)$ given by
\[
X=\{\left[
\begin{array}{ccc}
i(a-b/2)  & u  & 1-a  \\
-\bar{u}  &  ib & -i\bar{u}  \\
1-a  & iu &  -i(a+b/2)
\end{array}
\right]~:~ a,b\in\Rr, u\in \Cc\}
\]
The affine subspace $X$ is a Poisson transversal: it intersects the adjoint orbits (the symplectic leaves of $\mathfrak{su}(2,1)$) transversely in symplectic submanifolds. It follows that $X$ inherits a Poisson structure $\{\cdot,\cdot\}_X$ and a straightforward computation gives
\[ \{a,b\}_X=0, \quad \{a,u\}_X=\frac{3}{4}ibu, \quad \{b,u\}_X=-iu, \quad  \{u,\bar{u}\}_X=\frac{i}{8}(4-8a+9b^2).\]
Now observe that the Killing form yields the Casimir $C(x)=(x,x)$ for the Poisson manifold $\mathfrak{su}(2,1)$. Therefore, this function restricts to a Casimir on the Poisson transversal $X$, where it is given by the formula
\[ C(a,b,u)=2-4a-\frac{3}{2}b^2.\]
The Lie algebroid we are interested is the quotient Lie algebroid 
\[ A:=T^*X/\langle \d C\rangle\to X. \]
In fact, writing $u=u_1+iu_2$, we can take as a basis of sections for this Lie algebroid the ones defined by
\[ e_1:=\d u_1+\langle \d C\rangle,\quad e_2:=\d u_2+\langle \d C\rangle,\quad e_3:=\d b+\langle \d C\rangle, \]
and we can make the change of variables
\[ X=\frac{3}{4} u_2,\quad Y=-\frac{3}{4} u_1,\quad K=\frac{3}{2}b,\quad U=\frac{3}{64}(4-8a+9b^2), \]
which leads to the following expressions for the bracket of $A$:
\begin{align*}
[e_1,e_2]&=\d\{u_1,u_2\}_X+\langle \d C\rangle=-\frac{1}{2}\d a+\frac{9}{8} b\,\d b+\langle \d C\rangle=\frac{3}{2}b\,\d b+\langle \d C\rangle=Ke_3,\\
[e_1,e_3]&=\d\{u_1,b\}_X+\langle \d C\rangle=-\d u_2+\langle \d C\rangle=-e_2,\\
[e_2,e_3]&=\d\{u_2,b\}_X+\langle \d C\rangle=\d u_1+\langle \d C\rangle=e_1,
\end{align*}
and for the anchor of $A$
\begin{align*}
\rho(e_1)&=\{u_1,\cdot\}_X=-2X\frac{\partial}{\partial K}+U\frac{\partial}{\partial X}-KX \frac{\partial}{\partial U},\\
\rho(e_2)&=\{u_2,\cdot\}_X=-2Y\frac{\partial}{\partial K}+U\frac{\partial}{\partial Y}-KY \frac{\partial}{\partial U},\\
\rho(e_3)&=\{b,\cdot\}_X=Y\frac{\partial}{\partial X}-X\frac{\partial}{\partial Y}.
\end{align*}
Comparing with the expressions \eqref{eq:anchor:EK} and \eqref{eq:bracket:EK}, we conclude that $A=T^*X/\langle \d C\rangle$  is indeed the classifying $G$-structure algebroid of Bochner-K\"ahler surfaces. Note that under this change of variables the Casimir function $C$ becomes, up to a factor, the invariant function $I_1$
\[ C=-\frac{32}{3}(K^2/4-U)=-\frac{32}{3} I_1. \]
The other invariant arises from taking the determinant of an element $x\in X$
\[ \det x=-\frac{i}{4} (4b-8 ab + b^3 + 8 u \bar{u})= -\frac{32i}{9}(X^2+Y^2+UK-\frac{1}{6}K^3)=-\frac{32i}{9}I_2. \]

We can now proceed to integrate $A$ as follows. The cotangent Lie algebroid $T^* \mathfrak{su}(2,1)$, associated with the linear Poisson structure on $\mathfrak{su}(2,1)$, integrates to the action Lie groupoid associated with the adjoint action
\[ \mathfrak{su}(2,1)\rtimes \SU(2,1)\tto \mathfrak{su}(2,1). \]
We have the embedding of Lie algebroids
\[  \mathfrak{su}(2,1)\rtimes\mathfrak{u}(1)\to T^* \mathfrak{su}(2,1),\quad \lambda\mapsto \lambda \d b,\] 
which under the identification of $T^* \mathfrak{su}(2,1)$ with the adjoint action Lie algebroid $\mathfrak{su}(2,1)\rtimes \mathfrak{su}(2,1)$, arises from the embedding of Lie algebras
\[ \mathfrak{u}(1)\hookrightarrow \mathfrak{su}(2,1),\quad \lambda\mapsto \lambda\left[
\begin{array}{ccc}
-\frac{i}{2} & 0  & 0  \\
0 &  i& 0  \\
0 & 0  &  -\frac{i}{2}
\end{array}
\right]
\]
Therefore, this Lie algebroid morphism integrates to a Lie groupoid morphism
\[ \mathfrak{su}(2,1)\rtimes \UU(1)\hookrightarrow \mathfrak{su}(2,1)\rtimes \SU(2,1), \]
arising from the embedding of Lie groups
\[
\UU(1)\hookrightarrow \SU(2,1),\quad
\theta\mapsto 
\left[
\begin{array}{ccc}
e^{-i\frac{\theta}{2}}  & 0  & 0  \\
0 &  e^{i\theta}& 0  \\
0 & 0  &  e^{-i\frac{\theta}{2}}
\end{array}
\right]
\]
By restricting to the Poisson transversal $X\subset \mathfrak{su}(2,1)$ we obtain an integration of $T^*X$:
\[
\hat{\G}:=\{(g,\xi)\in \SU(2,1)\ltimes \mathfrak{su}(2,1): \xi\in X\text{ and } \Ad_g\xi\in X\}\tto X.
\]
Since the action of $U(1)$ by conjugation preserves $X$, this is still a $U(1)$-integration.

We can obtain a $G$-integration of the quotient $A=T^*X/\langle \d C\rangle$ by first find the Lie subgroupoid $\cK\subset \hat{\G}$ integrating the subalgebroid $\langle \d C\rangle\subset T^*X$, which is a normal, bundle of Lie groups inside the isotropy bundle, and then taking the quotient groupoid $\hat{\G}/\cK\tto X$. This yields a Lie groupoid provided $\cK\subset \hat{\G}$ is closed. Every orbit contains a point where $u_1=u_2=0$, and at these points we find
\begin{align*}
&\cK_{(a,b,0,0)}=\{\exp(t\d_{(a,b,0,0)} C):t\in \Rr\}=\{\exp(t 2x):t\in \Rr\}=\\
&=\{
\left[
\begin{array}{ccc}
\frac{\left(\mu+ i a\right) e^{+2\mu t}+\left(\mu- i a\right) e^{-2\mu t}}{2\mu} e^{-ibt} & 0 &  (a-1)\frac{e^{-2\mu t}-e^{+2\mu t}}{2\mu}e^{-ibt}  \\
 0 & e^{2i b t} & 0 \\
 (a-1)\frac{e^{-2\mu t}-e^{+2\mu t}}{2\mu}e^{-ibt}  & 0 &  \frac{\left(\mu+ i a\right) e^{-2\mu t}+\left(\mu- i a\right) e^{+2\mu t}}{2\mu} e^{-ibt}
\end{array}
\right]
:t\in \Rr\}
\end{align*}
where we have set $\mu=\sqrt{1-2 a}$. Now if $\mu$ is real, i.e., if $1-2a\ge 0$, then this subgroup is closed. On the other hand, if $\mu$ is purely imaginary, i.e., if $1-2a< 0$, this subgroup will be closed iff $b/\mu\in\Qq$. To compare this with the results of the previous section, we notice that a straightforward computation shows that the discriminant $\Delta$ at the point $(a,b,0,0)$ is given by
\[ \Delta=-\frac{27}{65536} \left(4-8 a+9 b^2\right)^2(1-2a)=-\frac{3}{16}U^2(1-2a). \]
We recover the result that $G$-integrability can only fail when $\Delta>0$ and a certain ratio is irrational. Note however, than this is not an if and only if result, since in principle one could have $G$-integrations which are not quotients of $\hat{\G}$. 

Hence, on the one hand, this approach can yield explicit solutions. On the other hand, the approach in the previous paragraph allows one to find the explicit orbits for which solutions exists.

%%%%%%%%%%%%%%%%%%%%%%%%%%%%%%%%
\subsection{Symmetries and moduli space of extremal K\"ahler metrics} 
%%%%%%%%%%%%%%%%%%%%%%%%%%%%%%%%

We will now use the results of Section \ref{sec:moduli} to describe the symmetries and the moduli space of extremal K\"ahler metrics. 

First, considering the moduli space of germs of solutions, according to Theorem \ref{thm:moduli:germs}, the action groupoid $X\times U(1)\tto X$ where $X=\Rr\times\Cc\times\Rr$, presents this stack. Recalling that the $U(1)$ action on $X$ is given by
\[(K,T,U)e^{i\theta} = (K,e^{i\theta}T,U). \]
we conclude that a germ of solution corresponding to the point $(K_0,T_0,U_0)$ has symmetry group:
\begin{itemize}
\item the \emph{trivial} group, if $T_0\not=0$;
\item the group $U(1)$, if $T_0=0$.
\end{itemize}
Note that this moduli space is not smooth since the action fails to be free. If we remove all solutions with $T_0=0$ we obtain a smooth moduli space. 

Note also that if $L$ is the sphere in the level set with $\Delta=0$ and $c_2>0$, and we consider the two marked points in this leaf with $T=0$, then there cannot exist any $G$-realization covering both marked points. In fact, if such a $G$-realization existed, then it would cover the whole sphere, contradicting the fact that $L$ is not $G$-integrable. This shows that the converse to Theorem \ref{thm:globalization} in general fails without the $G$-integrability assumption.

The moduli space of 1-connected solutions has a richer structure. According to Theorem \ref{thm:moduli:1-connected}, this moduli space is represented by the 
groupoid $\Sigma^\reg_G(A) \tto X_\reg$ where $X_\reg\subset X$ consists of the union of the leafs for which $A|_L$ is $G$-integrable. Hence, $X_\reg$ is obtained by removing from $X$ the compact leaves in the level sets $I_1=c_1$ and $I_2=c_2$ for which $\Delta>0$ and $\frac{4c_1-r_2^2}{r_3^2-4c_1}\not\in\Qq$.  Then we conclude that:
\begin{itemize}
\item the constant curvature 1-connected solutions $\Rr^2$, $\Ss^2$ and $\mathbf{H}$ are isolated (i.e., have no smooth deformations) and have symmetry groups  $\SO(2)\ltimes\Rr^2$, $\Ss^3$ and $\SO(2,1)$, respectively.
\item the orbifold solutions $\mathbb{CP}^1_{p,q}$ are also isolated and have symmetry group $\Ss^1$.
\item the remaining 1-connected solutions all admit smooth, non-trivial deformations, and the connected component of the identity of their symmetry group is either $\Rr$ or $\Ss^1$.
\end{itemize}
Notice that a 1-connected solution may have a richer symmetry structure than any of its germs, since a symmetry of a germ of solution necessarily fixes the marked point.

%***MODULI SPACE OF COMPLETE 1-CONNECTED SOLUTIONS***

\bibliographystyle{abbrv}
\bibliography{bibliography2}

\begin{thebibliography}{10}

\bibitem{Bryant:notes}
R.~L. Bryant.
\newblock Notes on exterior differential systems.
\newblock {\em ArXiv:1405.3116}.

\bibitem{Bryant}
R.~L. Bryant.
\newblock Bochner-{K}{\"a}hler metrics.
\newblock {\em J. of Amer. Math. Soc.}, 14(3):623--715, 2001.

\bibitem{Bursztyn13}
H.~Bursztyn.
\newblock A brief introduction to {D}irac manifolds.
\newblock In {\em Geometric and topological methods for quantum field theory},
  pages 4--38. Cambridge Univ. Press, Cambridge, 2013.

\bibitem{BIL19}
H.~Bursztyn, D.~Iglesias-Ponte, and J.-H. Lu.
\newblock Dirac geometry and integration of {P}oisson homogeneous spaces.
\newblock {\em ArXiv:1905.11453}.

\bibitem{BNZ20}
H.~Bursztyn, F.~Noseda, and C.~Zhu.
\newblock Principal actions of stacky {L}ie groupoids.
\newblock {\em Int. Math. Res. Not. IMRN}, (16):5055--5125, 2020.

\bibitem{CaSch2}
M.~Cahen and L.~J. Schwachh{\"o}fer.
\newblock Special symplectic connections and {P}oisson geometry.
\newblock {\em Lett. Math. Phys.}, 69:115--137, 2004.

\bibitem{CaSch}
M.~Cahen and L.~J. Schwachh{\"o}fer.
\newblock Special symplectic connections.
\newblock {\em J. Differential Geom.}, 83(2):229--271, 2009.

\bibitem{ChiMerkulovSchwachhofer:inventiones}
Q.-S. Chi, S.~A. Merkulov, and L.~J. Schwachh{\"o}fer.
\newblock On the existence of infinite series of exotic holonomies.
\newblock {\em Invent. Math.}, 126(2):391--411, 1996.

\bibitem{CrainicFernandes}
M.~Crainic and R.~L. Fernandes.
\newblock On integrability of {L}ie brackets.
\newblock {\em Annals of Mathematics}, 157:575--620, 2003.

\bibitem{CrainicFernandes:Poisson}
M.~Crainic and R.~L. Fernandes.
\newblock Integrability of {P}oisson brackets.
\newblock {\em J. of Differential Geometry}, 66:71--137, 2004.

\bibitem{CrainicFernandes:lectures}
M.~Crainic and R.~L. Fernandes.
\newblock Lectures on integrability of {L}ie brackets.
\newblock In {\em Lectures on {P}oisson geometry}, volume~17 of {\em Geom.
  Topol. Monogr.}, pages 1--107. Geom. Topol. Publ., Coventry, 2011.

\bibitem{CrainicYudilevich}
M.~Crainic and O.~Yudilevich.
\newblock Lie pseudogroups \`a la {C}artan.
\newblock {\em ArXiv:1801.00370}.

\bibitem{DuistKolk}
J.~J. Duistermaat and J.~A.~C. Kolk.
\newblock {\em Lie groups}.
\newblock Universitext. Springer-Verlag, Berlin, 2000.

\bibitem{FernandesStruchiner1}
R.~L. Fernandes and I.~Struchiner.
\newblock The classifying {L}ie algebroid of a geometric structure {I}:
  {C}lasses of coframes.
\newblock {\em Trans. Amer. Math. Soc.}, 366(5):2419--2462, 2014.

\bibitem{FernandesStruchiner2}
R.~L. Fernandes and I.~Struchiner.
\newblock The classifying {L}ie algebroid of a geometric structure {II}:
  {$G$}-structures with connection.
\newblock {\em S\~{a}o Paulo J. Math. Sci.}, 15(2):524--570, 2021.

\bibitem{KN1}
S.~Kobayashi and K.~Nomizu.
\newblock {\em Foundations of differential geometry. {V}ol. {I}}.
\newblock Wiley Classics Library. John Wiley \& Sons, Inc., New York, 1996.
\newblock Reprint of the 1963 original, A Wiley-Interscience Publication.

\bibitem{MerkulovSchwachhofer:annals}
S.~A. Merkulov and L.~J. Schwachh{\"o}fer.
\newblock Classification of irreducible holonomies of torsion-free affine
  connections.
\newblock {\em Annals of Mathematics}, 150(1):77 -- 149, 1999.

\bibitem{MoerdijkMrcun}
I.~Moerdijk and J.~Mrcun.
\newblock {\em Introduction to Foliations and Lie Groupoids}, volume~91 of {\em
  Cambridge Studies in Advanced Mathematics}.
\newblock Cambridge University Press, 2003.

\bibitem{Noohi14}
B.~Noohi.
\newblock Fibrations of topological stacks.
\newblock {\em Adv. Math.}, 252:612--640, 2014.

\bibitem{Sch}
L.~J. Schwachh{\"o}fer.
\newblock Connections with irreducible holonomy representations.
\newblock {\em Adv. Math.}, 160(1):1--80, 2001.

\bibitem{Szekelyhidi14}
G.~Sz\'{e}kelyhidi.
\newblock {\em An introduction to extremal {K}\"{a}hler metrics}, volume 152 of
  {\em Graduate Studies in Mathematics}.
\newblock American Mathematical Society, Providence, RI, 2014.

\end{thebibliography}

\end{document}